\documentclass[11pt,reqno]{amsart}
\usepackage[english]{babel}
\usepackage{mathtools}
\usepackage{mathrsfs, dsfont}
\usepackage{enumitem}
\usepackage{xcolor}
\usepackage[pdfencoding=auto, psdextra]{hyperref}
\hypersetup{colorlinks=true,allcolors=[rgb]{0,0,0.6}}
\usepackage[hmargin=2.7cm,vmargin=2.6cm]{geometry}
\usepackage{amsmath}
\usepackage{amssymb}
\newtheorem{thm}{Theorem}[section]
\newtheorem{proposition}[thm]{Proposition}
\newtheorem{lemma}[thm]{Lemma}
\newtheorem{cor}[thm]{Corollary}
\theoremstyle{definition}
\newtheorem{defn}[thm]{Definition}

\theoremstyle{remark}
\newtheorem{remark}[thm]{Remark}
\numberwithin{equation}{section}
\def\N{\mathbb{N}}
\def\Z{\mathbb{Z}}
\def\R{\mathbb{R}}
\def\C{\mathbb{C}}
\def\PP{\mathbb{P}}
\def\T{\mathbb{T}}
\def\Imm{\Im\mathrm{m}}
\def\Ree{\Re\mathrm{e}}
\def\tr{\mathrm {Tr}}
\begin{document}

\title[Gibbs measures as  KMS equilibrium  states]{Gibbs measures as unique KMS equilibrium states of nonlinear  Hamiltonian PDE\MakeLowercase{s}}
\author{Zied Ammari}
\address{Univ Rennes, [UR1], CNRS, IRMAR - UMR 6625, F-35000 Rennes, France.}
\email{zied.ammari@univ-rennes1.fr}
\author{Vedran Sohinger }
\address{University of Warwick, Mathematics Institute.}
\email{V.Sohinger@warwick.ac.uk.}

\date{February 23, 2021.}

%\dedicatory{}


\begin{abstract}
The classical Kubo-Martin-Schwinger (KMS) condition is a fundamental property of statistical mechanics characterizing the equilibrium of infinite classical mechanical systems. It was introduced in the seventies by  G.~Gallavotti and E.~Verboven as an alternative to the Dobrushin-Lanford-Ruelle (DLR) equation. In this article, we consider this concept in the framework  of nonlinear Hamiltonian   PDEs and  discuss its relevance. In particular,  we prove that Gibbs measures are the unique KMS equilibrium  states for such systems. Our proof is based on Malliavin calculus and Gross-Sobolev spaces.  The main feature  of our work is the applicability of our results to the general context of white noise, abstract Wiener spaces and Gaussian probability spaces, as well as to fundamental examples of PDEs like the nonlinear Schr\"odinger, Hartree, and wave (Klein-Gordon) equations.
\end{abstract}

\maketitle
\tableofcontents

\section{Introduction}
\label{sec.intro}

The \emph{Kubo-Martin-Schwinger (KMS) condition} emerged in quantum statistical mechanics as a criterion  characterizing the equilibrium states for infinite quantum systems  \cite{MR1441540,MR0219283}. Ever since, the concept has been considered as a cornerstone in the study of quantum dynamical systems  and more generally in $C^*$ and $W^*$-algebras topic, see e.g. \cite{MR1136257}. In the seventies, G.~Gallavotti and E.~Verboven, suggested a classical analogue to the quantum  KMS  condition suitable for classical mechanical systems and they analysed its relationship with the  Dobrushin-Lanford-Ruelle (DLR) equation, see \cite{MR0449393} and also \cite{MR0443784}. The work of Gallavotti and Verboven generated interest in the study of KMS states for infinite classical systems (see for instance \cite{MR0484128,MR0469018} and the references therein). To the best of our knowledge, only few results are concerned with nonlinear PDEs, namely \cite{MR707998,MR865768,MR815095}.

\medskip

On the other hand, \emph{Gibbs measures} for nonlinear Hamiltonian systems have attracted a lot of interest in the PDE community since \cite{MR1309539,MR1374420,MR1470880,MR939505,MR1335059},
following on the analysis of this problem in the constructive quantum field theory literature \cite{MR0887102,MR0489552}.
Indeed,  these measures turn to be an  effective tool in the study of almost sure existence of global solutions with rough initial data since they provide conservation laws beyond the classical energy spaces, see e.g. \cite{MR3869074,MR3952697,MR3844655,MR2227040,MR2498359} and the references therein. In this approach, the main ingredient is the invariance of the measure, rather than its statistical properties. In principle any invariant measure would produce conceptually similar results. It is therefore desirable to bridge the statistical and PDE points of view with the aim of obtaining a better understanding of stability and ergodic theory for PDE dynamical systems.  Moreover,
in light of recent progress made in quantum statistical mechanics it is quite tempting to investigate thoroughly the structure of classical KMS states for such dynamical systems. In particular, the KMS states are a convenient tool for the study of  thermodynamic limits, multi-phase behavior and ergodic properties.

\medskip

The purpose of this article is to introduce the concept of \emph{KMS equilibrium  states} for Hamiltonian PDEs and to study their main properties and general structure. In this respect, one of the primary problems  that we shall consider is the relationship between KMS states and Gibbs measures. To answer such a question, we consider an abstract framework for Hamiltonian PDEs  within which it is possible to rigorously define a Gibbs measure. First, we show that such a Gibbs measure is a KMS equilibrium state (Theorem \ref{thm.infdN}). Second, we show that, under additional assumptions, the Gibbs measure is the unique KMS equilibrium  state of the Hamiltonian PDE (Theorem \ref{thm.KMSGibbs}). The general framework we consider encloses several fundamental examples that include white noise, abstract Wiener spaces, and Gaussian probability spaces, see Section \ref{sec.lin}. Our analysis applies to nonlinear PDEs like the nonlinear Schr\"odinger, Hartree, and wave (Klein-Gordon) equations as illustrated in Section \ref{sec.Nonlinear}.

\medskip
%%formal introduction%%
Let us formally explain our setup.
A dynamical system is described by a vector field  $X:\mathcal{S} \to\mathcal S$, defined as a mapping over a phase-space $\mathcal S$, and a field equation,
\begin{equation*}
\dot u(t)=X(u(t))\,,
\end{equation*}
where $u:\R\to \mathcal S$ is a solution satisfying a prescribed  initial condition $u(0)=u_0\in\mathcal S$. There are two general approaches for the study of the dynamics. A \emph{deterministic} point of view aims to establish local or global well-posedness results in different functional spaces (i.e.: existence, uniqueness and stability in Hadamard's sense). The main related  questions  in this approach concern periodic and soliton solutions, blow-up solutions and scattering. A second \emph{probabilistic} point of view aims to study the dynamics of ensembles of initial data rather that of a single point in the phase-space. This leads to the consideration of the \emph{Liouville} equation,
\begin{equation}
\label{int.eq.liouv}
\frac{d}{dt}\int_{\mathcal S} F(u)\,\mu_t(du)=\int_{\mathcal S} \langle \nabla F(u), X(u)\rangle \,\mu_t(du)\,,
\end{equation}
where  $\langle \cdot,\cdot\rangle$ denotes a given Euclidean structure on the phase-space $\mathcal S$, $\nabla F$ is a gradient of the smooth function $F$ and $t\in\R\mapsto \mu_t$ is a probability measure solution with a prescribed initial condition $\mu_0$. The main questions within this approach are about existence, uniqueness, asymptotic statistical stability of solutions and  chaotic or ergodic behavior of the dynamical system.  In this context, probability measures on the phase-space $\mathcal S$  are regarded as classical states of the dynamical  system  and  the classical KMS condition is a widely accepted criterion that  singles out the equilibrium states among all possible stationary states of the Liouville equation \eqref{int.eq.liouv}.
In fact, we will say that $\mu$ is a KMS state at inverse temperature $\beta>0$ if and only if
\begin{equation}
\label{int.eq.kms}
\int_{\mathcal S} \{F,G\}(u)\;\mu(du)=\beta \,\int_{\mathcal S} \, \langle \nabla F(u), X(u)\rangle \;G(u)\;\mu(du)\,,
\end{equation}
where $F,G:\mathcal S\to\R$ are two smooth functions  and $\{\cdot,\cdot\}$ denotes a
Poisson structure over the phase-space $\mathcal S$. By taking the function $G\equiv 1$ one remarks that any KMS state is a stationary solution of the Liouville equation.

\medskip
 The two approaches complement each other. The first is suitable when the nonlinear effects  are ``weak"  while the second is more adapted to ``strong" nonlinear effects and turbulence.  Of course such a classification is heuristic and undermines the complexity and the variety of dynamical systems.
On other hand, the Liouville equation is at a crossroads between dispersive PDE, kinetic theory, probability and statistical mechanics. It is therefore quite instructive to unify the different techniques from these fields towards a better understanding of the dynamical behavior of some of the fundamental examples of PDEs  such as the nonlinear Schr\"odinger and wave equations.
Indeed, the general aim of this article is to study the Liouville equation from the following three perspectives.
\begin{itemize}
\item We prove that the nonlinear  vector fields of these equations make sense as the Malliavin  derivative  of energy functionals in the Gross-Sobolev spaces, thus highlighting the fact that global stochastic analysis is a well fitted tool for the study of such deterministic dispersive PDEs.
\item We prove that Gibbs measures are stationary solutions of the Liouville equation which  indicates that the techniques of kinetic theory and  gradient flow on probability measure spaces would be fruitful in this problem, see e.g. \cite{MR2475421}.
\item We show that the Gibbs measures are KMS equilibrium states of the dynamical system and therefore it is  very tempting to study the system near equilibrium and to investigate its statistical stability and asymptotic properties.
\end{itemize}
Beyond the formal program sketched above, there are more precise motivations for our present work. First, characterizing Gibbs measures through the KMS condition would provide an alternative method for the derivation of Gibbs measures from many-body quantum field theories, see \cite{Ammari:2020} and  \cite{MR3366672,FKSS_2020_1,MR4131015,FKSS_2020_3,MR3719544,LNR_4,Sohinger_2019}. Second, the  Gibbs measure as a KMS state and  a stationary solution of the Liouville equation should generally  yield the existence of a global flow  defined almost surely on the phase-space $\mathcal S$, see
\cite{MR3737034,MR3721874,Ammari:2018aa}. These questions will be addressed elsewhere and here we focus on the more fundamental properties of the KMS condition.

\bigskip
For an illustration of our main results, consider the NLS equation on the $2$-dimensional torus  $\T^2$ defined through its classical Hamiltonian,
$$
\mathcal H(u)=\frac{1}{2} \int_{\T^2} |\nabla u|^2+|u|^2 \,dx+\frac{1}{2m} \int_{\T^2} :|u|^{2m}: \,dx\,,
$$
where $:\; :$ denotes Wick  ordering (see Section \ref{sec.Nonlinear}). Furthermore, we note that the nonlinear functional $h^I:H^{-s}(\T^2)\to\R$ defined over the negative Sobolev space with $s>0$ and given by
$$
h^I(u)=\frac{1}{2m} \int_{\T^2} :|u|^{2m}: \,dx\,
$$
belongs to the Gross-Sobolev space $\mathbb D^{1,2}(\mu_{\beta,0})$ where $\mu_{\beta,0}$ is a centered Gaussian measure with covariance operator $\beta^{-1}(-\Delta+\mathds 1)^{-1}$ (see Definition \ref{Gross-Sobolev_spaces} below). Moreover, we prove that the Gibbs measure
$$
\mu_\beta=z_\beta^{-1} e^{-\beta h^I} \mu_{\beta,0}\,,
$$
for $z_\beta=\mu_{\beta,0}(e^{-\beta h^I})$ a normalization constant, is the unique equilibrium KMS state of the NLS dynamical system. In particular, $\mu_\beta$ is a stationary solution of the
 corresponding Liouville equation \eqref{int.eq.liouv} with the vector field
 $$
 X(u)=-i(-\Delta u+u+\nabla h^I(u))\,,
$$
where $\nabla h^I$ is the Malliavin derivative of the nonlinear functional $h^I$ (see Lemma \ref{grad} for the precise definition). The above statements are obtained as a consequence of a more general result (Theorem \ref{thm.KMSGibbs}). Indeed,  consider a complex Hamiltonian system,
$$
h(u)=h_0(u)+h^I(u)\,,
$$
such that $h_0(u)=\frac{1}{2} \langle u, A u\rangle$ with $A$ a positive self-adjoint operator admitting a compact resolvent such that  for some $\alpha\geq 1$,
$$
\tr[A^{-\alpha}]<\infty\,.
$$
Moreover, assume that the nonlinear functional  $h^I\in\mathbb D^{1,2}(\mu_{\beta,0})$ and $e^{-\beta h^I}\in L^2(\mu_{\beta,0})$, where $\mu_{\beta,0}$ is the centered Gaussian measure with covariance operator $\beta^{-1} A^{-1}$. So, we prove that if $\mu$ is a KMS state for this dynamical system which  is absolutely continuous with respect to $\mu_{\beta,0}$ with a density $\varrho =\frac{d \mu}{d \mu_{\beta,0}} \in \mathbb D^{1,2}(\mu_{\beta,0})$, then $\mu$ is the Gibbs measure $\mu_\beta$, i.e.:
$$
 \mu= \mu_{\beta} \equiv \frac{1}{\|e^{-\beta h^I}\|_{L^1(\mu_{\beta,0})}} \;e^{-\beta h^I} \mu_{\beta,0}\,.
$$
The proof of the above result is based on the derivation of a differential equation on the density $\varrho$ given by
$$
\nabla \varrho+\beta \varrho\nabla h^I=0\,.
$$
To solve such an equation, one uses the Malliavin calculus in order to  prove that $c e^{-\beta h^I}$ are the only solutions of the above equation. Uniqueness is then obtained using the normalization of the density $\varrho$.
The details of these arguments are given in Section \ref{sec.KMSGibbs}.

\bigskip
\emph{Overview of the article:} We define, in Section \ref{sec.prop}, the notion of
KMS states and study their main properties in a  general framework.  In particular, we establish a relationship between  Kirkwood-Salzburg type hierarchy equations and KMS states and prove stationarity, convexity and characteristic identities. In Section \ref{sec.lin}, fundamental examples of KMS states are given in terms of Gaussian measures over  countably Hilbert  nuclear spaces and canonical Gaussian measures on Wiener spaces and Gaussian probability spaces. Finally, we prove the equivalence between KMS states and Gibbs measures in Section \ref{sec.KMSGibbs} for:
\begin{itemize}
\item Finite dimensional dynamical systems;
\item Linear complex Hamiltonian systems;
\item Nonlinear complex Hamiltonian systems.
\end{itemize}
To emphasize our main results and show their wide applicability, we consider in Section \ref{sec.Nonlinear} several examples of nonlinear PDEs such as the nonlinear Schr\"odinger (NLS), Hartree and wave equations. When studying the NLS in $1D$, we also address the problem of the focusing nonlinearity. In this case, we prove that suitably localized invariant measures of Gibbs type satisfy a \emph{local KMS condition},  see Section \ref{subsec.focusing} for the precise definition. In Appendix \ref{appx}, we provide a short review of Malliavin calculus. In Appendix \ref{Appendix_B}, we prove some auxiliary facts about Sobolev embedding and discrete convolutions that we use in Section \ref{sec.Nonlinear}.

\bigskip
\emph{Acknowledgements:}
The authors would like to thank Andrew Rout for helpful discussions. V.S. acknowledges support of the EPSRC New Investigator Award grant EP/T027975/1.

\section{KMS states and their main properties}
\label{sec.prop}
 In this section, we define classical KMS states in a general abstract framework and study their  main properties.

\subsection{General framework}
\label{subsec.frw}
There are several possible settings  for the notion of KMS states.
We first give the definition in the most general setting. In the sequel, we adapt this to the probability and PDE context.
So, we start with  a rigged Hilbert space setting  together with a compatible symplectic structure. The latter  allows us to define a Poisson structure on an appropriate algebra of smooth cylindrical functions.

\bigskip
\emph{Rigged Hilbert space:} Consider a rigged Hilbert space $\Phi\subseteq H \subseteq \Phi'$ where $(H, \langle\cdot,\cdot\rangle)$ is a real separable Hilbert space, $\Phi$ is a dense subset of $H$ equipped with the structure of a topological vector space such  that the natural  embedding $i: \Phi \to H$ is  continuous and
$\Phi'$ is the dual of $\Phi$ with respect to the inner product of $H$. Two standard examples are
 $\mathscr S(\R^d)\subseteq L^2(\R^d) \subseteq \mathscr S'(\R^d)$ and $H^s(\R^d)\subseteq L^2(\R^d) \subseteq H^{-s}(\R^d)$ where $\mathscr S(\R^d)$ is the Schwartz space and $H^s(\R^d)$ is the Sobolev space with  a non-negative exponent $s\geq 0$.

\medskip
In all the sequel, $\mathfrak B(\Phi')$ denotes  the Borel $\sigma$-algebra over $\Phi'$ where the latter space is equipped with the weak-$*$ topology. Moreover, $\mathfrak P(\Phi')$ denotes the set of all Borel probability measures on $\Phi'$.

\bigskip
\emph{Symplectic structure:} Assume further that the Hilbert space $H$ is endowed with a non-degenerate continuous symplectic structure $\sigma: H\times H\to \R$, i.e.: $\sigma$ is a continuous bilinear form satisfying $\sigma(u,v)=-\sigma(v,u)$ for all $u,v\in H$ and if $\sigma(u,v)=0$ for all $v\in H$ then $u=0$. Therefore, there exists a unique  bounded linear operator $J:H\to H$ such that
$$
\sigma(u,v)= \langle u, J v\rangle \,, \qquad \forall u,v\in H\,.
$$
In particular, the transpose operator of $J$ is $J^T=-J$. Suppose further that $J$ maps continuously $\Phi$  to itself and consequently $J$ extends uniquely and continuously to $J:\Phi'\to \Phi'$. For now there is no need to introduce a compatible complex structure. Such an assumption will be required in Section \ref{sec.KMSGibbs}.

\bigskip
\emph{Smooth cylindrical test functions:}
Let $\{e_n\}_{n\in\N}$ be a countable linearly independent subset of $\Phi$ such that ${\rm span}_\R\{e_j,j\in\N\}$ is dense in $H$. Then one defines the spaces of smooth cylindrical test functions denoted respectively by $\mathscr{C}_{c,cyl}^{\infty}(\Phi')$, $\mathscr{S}_{cyl}(\Phi')$ and $\mathscr{C}_{b,cyl}^{\infty}(\Phi')$,  as the sets of all functions $F:\Phi'\to \R$ such that there exist $n\in\N$ and a function $\varphi: \R^n\to\R$ satisfying for all $u\in \Phi'$,
\begin{equation}
\label{eq.18}
F(u)=\varphi(\langle u, e_1\rangle,\dots, \langle u, e_n\rangle)\,,
\end{equation}
with $\varphi\in \mathscr{C}^\infty_c(\R^n)$, $\varphi\in \mathscr{S}(\R^n)$ or  $\varphi\in \mathscr{C}_b^\infty(\R^n)$ respectively. Here, we recall that the latter space consists of smooth functions all of whose derivatives are bounded. Obviously, one has the following inclusions
 $$
 \mathscr{C}_{c,cyl}^{\infty}(\Phi')\subset\mathscr{S}_{cyl}(\Phi')\subset\mathscr{C}_{b,cyl}^{\infty}(\Phi')
 \,.
 $$
     Remark that $\mathscr{C}_{c,cyl}^{\infty}(\Phi')$ is stable under multiplication but not stable under addition of its elements, while $\mathscr{C}_{b,cyl}^{\infty}(\Phi')$ is a unital algebra over the field $\R$. Although the representation formula \eqref{eq.18} may not be unique, these classes of smooth functions are quite convenient for the analysis. Indeed, they are endowed with a nice differential calculus. In fact, all the functions in $\mathscr{C}_{b,cyl}^{\infty}(\Phi')$  are differentiable over $\Phi'$ in the direction of $H$. More precisely,  taking $F:\Phi'\to\R$ as in  \eqref{eq.18} then for all $v\in H$,
 \begin{equation}
 \label{eq.20}
 DF(u)[v]=\lim_{v\to 0,v\in H} \frac{F(u+v)-F(u)}{||v||}=
  \sum_{j=1}^n \partial_j\varphi(\langle u, e_1\rangle,\dots, \langle u, e_n\rangle) \, \langle e_j,v\rangle\,.
 \end{equation}
Furthermore, the differential $DF(u)$ is regarded as a continuous $\R$-linear  form  in $\mathcal{L}(H,\R)\simeq H$. In particular, the gradient of $F$ is defined as
\begin{equation}
\label{grad.defn}
\nabla F(u):= \sum_{j=1}^n \partial_j\varphi(\langle u, e_1\rangle,\dots, \langle u, e_n\rangle) \, e_j\,\in \Phi\,,
\end{equation}
and the following product rule is true for all $F,G\in \mathscr C_{b,cyl}^\infty(\Phi')$ and $u\in \Phi'$,
\begin{equation}
\label{prod.rule}
\nabla(FG)(u)=\nabla F(u) \,G(u)+ F(u) \,\nabla G(u)\,.
\end{equation}
It is useful to write the gradient of a smooth cylindrical function using the Fourier transform.
\begin{lemma}
\label{lem.nab}
For any $F\in\mathscr S_{cyl}(\Phi')$ satisfying \eqref{eq.18}, for some $n\in\N$ and $\varphi\in\mathscr S(\R^n)$,
\begin{eqnarray*}
\nabla F(u)&=&-(2\pi)^{-\frac{n}{2}} \sum_{j=1}^n \int_{\R^n}  t_j e_j \;\bigg(\Ree(\hat\varphi)(t_1,\dots,t_n) \;\sin\langle\sum_{j=1}^n t_j e_j ,u\rangle \\&&\hspace{.7in}+ \Imm(\hat\varphi)(t_1,\dots,t_n)\; \cos\langle\sum_{j=1}^n t_j   e_j,u\rangle \bigg)\; dt\,,
\end{eqnarray*}
where $\hat \varphi$ is the Fourier transform of $\varphi$.
\end{lemma}
\begin{proof}
Using \eqref{grad.defn}, one writes
\begin{eqnarray*}
\langle \nabla F(u),v\rangle &=&\sum_{j=1}^n \partial_j\varphi(\langle u, e_1\rangle,\dots, \langle u, e_n\rangle) \, \langle  e_j,v\rangle\\
&=& (2\pi)^{-\frac{n}{2}}  \sum_{j=1}^n \int_{\R^n}\widehat{\partial_j\varphi}(t_1,\dots,t_n) \, e^{i \sum_{j=1}^n t_j \langle u, e_j\rangle}    \, \langle e_j,v\rangle \,dt\,.
\end{eqnarray*}
Since the left hand side is real, then one shows
\begin{eqnarray*}
\langle \nabla F(u),v\rangle &=&i (2\pi)^{-\frac{n}{2}}   \int_{\R^n} \widehat{\varphi}(t_1,\dots,t_n) \, e^{i \sum_{j=1}^n t_j \langle u, e_j\rangle}    \, \langle  \sum_{j=1}^n t_j e_j, v\rangle \,dt\\
&=& - (2\pi)^{-\frac{n}{2}}   \int_{\R^n} \bigg(\Ree(\widehat{\varphi})(t_1,\dots,t_n) \, \sin \langle\sum_{j=1}^n t_j e_j,u\rangle \\&& \hspace{.4in}+ \Imm(\widehat{\varphi})(t_1,\dots,t_n)
\cos\langle\sum_{j=1}^n t_j e_j,u\rangle \bigg)    \, \langle  \sum_{j=1}^n t_j e_j,v\rangle \,dt\,.
\end{eqnarray*}
The last equality proves the claimed identity.
\end{proof}

\bigskip
\emph{Poisson structure:} The above differential calculus enables us to define a Poisson structure over the algebra   $\mathscr{C}_{b,cyl}^{\infty}(\Phi')$. In fact, in this framework we define the Poisson bracket
for all $F,G\in \mathscr{C}_{b,cyl}^{\infty}(\Phi')$ and all $u\in \Phi'$ as
\begin{equation}
\label{eq.poi}
 \{F,G\}(u)= \sigma(\nabla F(u), \nabla G(u))=\langle\nabla F(u), J\nabla G(u)\rangle\,.
\end{equation}
 Using \eqref{grad.defn} one shows that $\{F,G\}$ belongs to the algebra $\mathscr{C}_{b,cyl}^{\infty}(\Phi')$ and obviously the  bracket is  bilinear and skew symmetric. Moreover, one checks that the Leibniz rule and the Jacobi identity are satisfied for all $F,G,R\in \mathscr{C}_{b,cyl}^{\infty}(\Phi')$,
\begin{itemize}
  \item $\{R,FG\}=\{R,F\} ~ G+\{R,G\} ~ F$\,;
  \item $\{F,\{G,R\}\}+\{G,\{R,F\}\}+\{R,\{F,G\}\}\ =0\,.$
\end{itemize}
The Poisson structure is one of the main ingredients that enters in the definition of KMS states.
Similarly as in Lemma \ref{lem.nab}, one can express the Poisson bracket using the Fourier transform.
\begin{lemma}
\label{lem.bra}
For all $F,G\in\mathscr{S}_{cyl}(\Phi')$ such that \eqref{eq.18} is satisfied for both $F$ and $G$ with  $n,m\in\N$ and $\varphi\in\mathscr{S}(\R^n)$, $\psi\in \mathscr{S}(\R^m)$ respectively.  Then
\begin{eqnarray*}
&&\hspace{-.5in}\{F,G\}(u)=-(2\pi)^{-\frac{(n+m)}{2}}\\ &&\int_{\R^{n}} \int_{\R^{m}}  \hat\varphi(t_1,\dots,t_n)   \hat\psi(s_1,\dots,s_m) \;\big\langle
\sum_{j=1}^n t_j e_j, J \sum_{j=1}^m s_j e_j\big\rangle\, \displaystyle\,e^{i \sum_{j=1}^n t_j
\langle u, e_j\rangle +i \sum_{j=1}^m s_j
\langle u, e_j\rangle} \; dtds \,.
\end{eqnarray*}
\end{lemma}
\begin{proof}
This follows from Lemma \ref{lem.nab} and \eqref{eq.poi}.
Since $\varphi$ and $\psi$ are real-valued functions, one has
\begin{eqnarray*}
i   \int_{\R^n} t_j\widehat{\varphi}(t_1,\dots,t_n) \, e^{i \sum_{j=1}^n t_j \langle u, e_j\rangle}   dt&=&
\int_{\R^n}  \Ree(\hat\varphi)(t_1,\dots,t_n) \;\sin\langle\sum_{j=1}^n t_j e_j ,u\rangle \\&&+ \Imm(\hat\varphi)(t_1,\dots,t_n)\; \cos\langle\sum_{j=1}^n t_j   e_j,u\rangle\,dt\,,
\end{eqnarray*}
and similarly
\begin{eqnarray*}
 i   \int_{\R^m} s_j\widehat{\psi}(s_1,\dots,s_n) \, e^{i \sum_{j=1}^m s_j \langle u, e_j\rangle}   ds&=& \int_{\R^m} \Ree(\hat\psi)(s_1,\dots,s_m) \;\sin\langle\sum_{j=1}^m s_j e_j ,u\rangle \\&&+ \Imm(\hat\psi)(s_1,\dots,s_m)\; \cos\langle\sum_{j=1}^m s_j   e_j,u\rangle\, ds.
\end{eqnarray*}
Hence, using Lemma \ref{lem.nab} and \eqref{eq.poi}, one checks that the claimed identity holds true.
\end{proof}

\subsection{KMS equilibrium states}

Consider a Borel vector field $X:\Phi'\to\Phi'$ defining a (formal) dynamical system  given by the differential equation,
\begin{equation}
\label{eq.ds}
\partial_t u(t)=X(u(t))\,,
\end{equation}
where $t\in \R \mapsto u(t)\in \Phi'$ is a curve with prescribed  initial condition
$u(0)=u_0\in \Phi'$. The above equation may not make sense and usually additional assumptions on
the solutions $u(\cdot)$  or the vector field $X$ are required. One can look at the field equation \eqref{eq.ds} from a statistical point of view. So, instead of studying the initial value problem \eqref{eq.ds} for each fixed data $u_0$, one can consider the dynamical evolution of an ensemble of initial datum given by a probability distribution. It turns that the statistical dynamics related to the vector field equation \eqref{eq.ds} are described by the Liouville (transport) equation,
\begin{equation}
\label{eq.trans}
\frac{d}{dt} \int_{\Phi'} F(u) \, d\mu_t(u)=\int_{\Phi'} \langle \nabla F(u), X(u)\rangle \, d\mu_t(u)\,,
\end{equation}
for all $F\in\mathscr{C}_{c,cyl}^{\infty}(\Phi')$ and  where  $t\in \R\mapsto \mu_t\in \mathfrak{P}(\Phi')$
is a curve of statistical solutions with a prescribed initial condition $\mu_0\in\mathfrak P(\Phi')$. Note that the definition of the Liouville equation
requires neither a symplectic nor a Poisson structure.
We remark also that the Liouville equation \eqref{eq.trans} may not make sense without further requirements on the vector field $X$ or on the solutions $\mu_t$. Existence of solutions for the above Liouville equation \eqref{eq.trans} and its relationship with the original field equation \eqref{eq.ds}  is studied in \cite{MR3721874}. In this article, we focus only on stationary solutions of the Liouville equation that
represent dynamical equilibrium. A Borel probability measure $\mu$ on $\Phi'$ is a \emph{stationary solution} of the Liouville equation \eqref{eq.trans} if and only if for all $F\in\mathscr C_{c,cyl}^\infty(\Phi')$ the function $\langle \nabla F(\cdot), X(\cdot)\rangle$ is $\mu$-integrable and
\begin{equation}
\label{eq.stat}
\int_{\Phi'} \langle \nabla F(u), X(u)\rangle \, d\mu=0\,.
\end{equation}
Not all stationary solutions correspond to a statistical equilibrium of the dynamical system. The KMS condition that we shall define below is a widely accepted criterion that defines the notion of  statistical equilibrium and stability.

\medskip
\emph{Kubo-Martin-Schwinger condition:}
The Kubo-Martin-Schwinger (KMS) condition, given below in \eqref{KMS-G}, is a   dynamical characterization of equilibrium measures of the Liouville equation \eqref{eq.trans} at an inverse positive temperature $\beta>0$. These equilibrium  measures will be called KMS states and their definition is rigourously given below.

\begin{defn}[KMS states] \label{KMS-def}
Let $X:\Phi'\to\Phi'$ be a Borel vector field and $\beta>0$. We say that $\mu \in \mathfrak{P}(\Phi')$ is a $(\beta,X)$-KMS state if and only if the function $\langle f, X(\cdot)\rangle$ is $\mu$-integrable for all $f\in \Phi$ and for all $F,G\in\mathscr{C}^\infty_{c,cyl}(\Phi')$, we have
\begin{equation}
\label{KMS-G}
\int_{\Phi'}\, \{F,G\}(u) \,d\mu=\beta \int_{\Phi'} \,\langle \nabla F(u), X(u)\rangle \, G(u)\, d\mu\,,
\end{equation}
with the Poisson bracket $\{\cdot,\cdot\}$ is defined in \eqref{eq.poi}.
\end{defn}

The existence and uniqueness of such KMS equilibrium states is in general  a non trivial question as one can see for instance  in  \cite{MR0449393}. Despite this fact, it is useful to underline the general properties of these measures.

\begin{lemma}
\label{lem.4}
 If $\mu$ is  a $(\beta,X)$-KMS state then the identity \eqref{KMS-G} is true for all $F,G\in\mathscr C_{b,cyl}^\infty(\Phi')$.
\end{lemma}
\begin{proof}
Using a standard pointwise approximation argument of functions $\varphi\in\mathscr C_{b}^\infty(\R^n)$ by sequences of functions in $\mathscr C_{c}^\infty(\R^n)$, the equality \eqref{KMS-G} extends by dominated convergence to all $F,G\in\mathscr C_{b,cyl}^\infty(\Phi')$.
\end{proof}

Not all stationary solutions of the Liouville equation \eqref{eq.trans} are KMS equilibrium states, but the converse is true. Remark that here the time invariance is formulated without appealing to a flow for the field equation  \eqref{eq.ds}.
\begin{proposition}
Any $(\beta,X)$-KMS state is a stationary solution of the Liouville equation \eqref{eq.trans}.
\end{proposition}
\begin{proof}
Thanks to Lemma \ref{lem.4}, it is enough to take in \eqref{KMS-G} the function $G(\cdot)=1$ which belongs to $\mathscr C_{b,cyl}^{\infty}(\Phi')$ and hence obtaining \eqref{eq.stat}.
\end{proof}

One important geometric feature of the set of KMS states is convexity.
\begin{proposition}
\label{prop.conv}
The set of $(\beta,X)$-KMS states is a convex subset of  $ \mathfrak P(\Phi')$.
\end{proposition}
\begin{proof}
Let $\alpha\in[0,1]$ and $\mu,\nu$ two $(\beta,X)$-KMS states. Then one easily  checks that the integrability of the functions $\langle f, X(\cdot)\rangle$  and the identity \eqref{KMS-G} are satisfied with respect to the probability measure $\alpha\mu+(1-\alpha) \nu$.
\end{proof}

A simple identification  of the KMS states in terms of their  characteristic functions is provided below. We note that, in \cite{MR865768}, the identity  \eqref{eq.KMSexp} is regarded as the definition of KMS states.
\begin{thm}
\label{thm.char}
Let $\mu \in \mathfrak{P}(\Phi')$, $X$ a Borel vector field on $\Phi'$ and $\beta>0$ be given.
Then the two following assertions are equivalent.
\begin{itemize}
  \item[(i)] $\mu$ is a $(\beta,X)$-KMS state.
  \item[(ii)] For all $\varphi_1,\varphi_2\in \Phi$, the function $\langle \varphi_1, X(\cdot)\rangle$ is $\mu$-integrable and
  \begin{equation}
  \label{eq.KMSexp}
  \langle \varphi_1,J\varphi_2\rangle \,\int_{\Phi'}\, e^{i \langle u, \varphi_2\rangle} \, d\mu+i\beta \,\int_{\Phi'}\, \langle \varphi_1,X(u)\rangle \,  e^{i \langle u, \varphi_2\rangle} \, d\mu=0\,.
  \end{equation}
\end{itemize}
\end{thm}
\begin{proof}
Assume (i) true and take $f,g\in \Phi$ such that $f=\varphi_1$ and $f+g=\varphi_2$. Then
\begin{eqnarray*}
\langle J\varphi_1,\varphi_2\rangle e^{i \langle u, \varphi_2\rangle}&=&-\langle f, Jg\rangle
e^{i \langle u, f\rangle} e^{i \langle u, g\rangle}\\
&=&-\{\sin(\langle f,u\rangle),\sin(\langle g,u\rangle)\}+\{\cos(\langle f,u\rangle), \cos(\langle g,u\rangle)\}\\&&+i \{\sin(\langle f,u\rangle),\cos(\langle g,u\rangle)\}+i\{\cos(\langle f,u\rangle),\sin(\langle g,u\rangle)\} \,.
\end{eqnarray*}
Hence, integrating the above equality  with respect to $\mu$ and using for each bracket the KMS condition
\eqref{KMS-G} and Lemma \ref{lem.4}, one shows
\begin{eqnarray*}
\langle J\varphi_1,\varphi_2\rangle \,\int_{\Phi'}\, e^{i \langle u, \varphi_2\rangle} \, d\mu=&\displaystyle
\beta\int_{\Phi'}& \langle f, X(u)\rangle \bigg( -\cos\langle f,u\rangle\, \sin\langle g,u\rangle
 -\sin\langle f,u\rangle\, \cos\langle g,u\rangle \\&&+i \cos\langle f,u\rangle\, \cos\langle g,u\rangle-i\sin\langle f,u\rangle\, \sin\langle g,u\rangle\bigg) \; d\mu\\
 =&\displaystyle
\beta\int_{\Phi'}& \langle \varphi_1, X(u)\rangle \big( -\sin\langle f+g,u\rangle+i\cos\langle f+g,u\rangle\big) \; d\mu\\
=&\displaystyle
i\beta\int_{\Phi'}& \langle \varphi_1, X(u)\rangle \, e^{i \langle u, \varphi_2\rangle}\; d\mu.
\end{eqnarray*}
Thus, (ii) is proved. Conversely, suppose  that (ii) holds true, then as before the equation \eqref{eq.KMSexp} gives for all $f,g\in \Phi$,
\begin{equation}
\label{eq.19}
 \langle f,Jg\rangle \,\int_{\Phi'}\, e^{i \langle u, f+g\rangle} \, d\mu=-i\beta \,\int_{\Phi'}\, \langle f,X(u)\rangle \,  e^{i \langle u, f+g\rangle} \, d\mu\,.
\end{equation}
For $F,G\in\mathscr C_{c,cyl}^\infty(\Phi')$ there exists $n,m\in\N$ and $\varphi\in \mathscr C_{c}^\infty(\R^n)$,  $\psi\in \mathscr C_{c}^\infty(\R^m)$ satisfying \eqref{eq.18} respectively. The inverse Fourier transform gives
\begin{eqnarray*}
F(u)=(2\pi)^{-n/2}\int_{\R^n} \hat\varphi(t_1,\dots,t_n) \displaystyle \,e^{i \sum_{j=1}^n t_j
\langle u, e_j\rangle} \; dt\,,\\ G(u)=(2\pi)^{-m/2}\int_{\R^m} \hat\psi(s_1,\dots,s_m) \displaystyle\,e^{i \sum_{j=1}^m s_j
\langle u, e_j\rangle} \; ds\,,
\end{eqnarray*}
where $\hat\varphi$ and $\hat\psi$ are respectively the Fourier transform of $\varphi$ and $\psi$. By Lemma \ref{lem.bra}, one has
\begin{eqnarray*}
&&\hspace{-.5in}\{F,G\}(u)=-(2\pi)^{-\frac{(n+m)}{2}}\\ && \int_{\R^n \times \R^m} \hat\varphi(t_1,\dots,t_n)   \hat\psi(s_1,\dots,s_m) \;\big\langle
\sum_{j=1}^n t_j e_j, J \sum_{j=1}^m s_j e_j\big\rangle\, \displaystyle\,e^{i \sum_{j=1}^n t_j
\langle u, e_j\rangle +i \sum_{j=1}^m s_j
\langle u, e_j\rangle} \; dtds \,.
\end{eqnarray*}
Hence, writing the identity \eqref{eq.19}  with $f=\sum_{j=1}^n t_j e_j$ and $g=\sum_{j=1}^m s_j e_j$ and  multiplying it by $\hat\varphi(t_1,\dots,t_n) \times \hat\psi(s_1,\dots,s_m)$ and then integrating with respect to $t_j$ and $s_j$, one  obtains
\begin{equation}
\label{eq.23}
\int_{\Phi'} \{F,G\}(u) \,d\mu= \beta \int_{\Phi'}  R(u) \, G(u) \,d\mu\,,
\end{equation}
where $R(\cdot)$ is a real-valued function given by
\begin{eqnarray*}
R(u) &=&i (2\pi)^{-\frac{n}{2}} \int_{\R^n} \big\langle \sum_{j=1}^n t_j e_j,X(u)\big\rangle \;\hat\varphi(t_1,\dots,t_n)  \,e^{i \sum_{j=1}^n t_j
\langle u, e_j\rangle} \; dt\\
&=& - (2\pi)^{-\frac{n}{2}} \int_{\R^n} \big\langle \sum_{j=1}^n t_j e_j,X(u)\big\rangle \;\bigg(\Ree(\hat\varphi)(t_1,\dots,t_n) \sin\langle  \;\sum_{j=1}^n t_j e_j ,u\rangle\\ &&\hspace{.5in}+ \Imm(\hat\varphi)(t_1,\dots,t_n)\; \cos\langle\sum_{j=1}^n t_j   e_j,u\rangle \bigg)\; dt\\
&=&\langle \nabla F(u),X(u)\rangle \,.
\end{eqnarray*}
The last equality follows by Lemma \ref{lem.nab}. Thus, the identity  \eqref{eq.23} yields  the KMS condition \eqref{KMS-G}.
\end{proof}

The following statement may be interpreted as the \emph{passivity} of the  dynamical system at equilibrium which means  that the system is unable to perform mechanical work in a cyclic process (see e.g. \cite[Section 5.4.4]{MR1441540} for an analogy with quantum KMS states).
\begin{cor}
If $\mu$ is a $(\beta,X)$-KMS state then for all $\varphi\in \Phi$ and all $F\in\mathscr C_{b,cyl}^\infty(\Phi')$,
$$
\int_{\Phi'}\, \langle \varphi,X(u)\rangle \,d\mu=0\,, \qquad \text{ and } \qquad
\int_{\Phi'}\, \langle \nabla F(u),X(u)\rangle \,F(u) \,d\mu=0\,.
$$
\end{cor}
\begin{proof}
 In order to obtain the first identity, we take $\varphi_2=0$ in \eqref{eq.KMSexp}. In order to obtain the second identity, we take $G=F$ in the KMS condition
\eqref{KMS-G} and apply Lemma \ref{lem.4}.
\end{proof}

\subsection{Stationary and equilibrium  hierarchies}
In the recent article \cite{Ammari:2018aa}, a duality is established  between the Liouville equation \eqref{eq.trans} and a Bose-Einstein hierarchy equation generalizing the Gross-Pitaevskii  and Hartree hierarchies studied for instance in \cite{MR3385343,MR3500833,MR2377632,MR3360742} and the references therein. See also \cite{MR2950759,MR0475514} for similarity with the BBGKY hierarchy of classical mechanics. In this paragraph, we extend the above duality to stationary and equilibrium solutions. Throughout, we assume that $X$ is a Borel vector field on $\Phi'$.

\begin{lemma}
\label{lem.6}
Let $\mu \in \mathfrak{P}(\Phi')$ be such that $\langle \varphi, X(\cdot)\rangle$ is $\mu$-integrable for any $\varphi\in\Phi$.
Then $\mu$ is a stationary solution of the Liouville equation, i.e. it solves \eqref{eq.stat}, if and only if for all
$\varphi\in\Phi$,
\begin{equation}
\label{eq.27}
\int_{\Phi'} \langle \varphi, X(u)\rangle \,e^{i\langle u ,\varphi\rangle} \;d\mu=0\,.
\end{equation}
\end{lemma}
\begin{proof}
Suppose that $\mu$ is a stationary solution, then the identity \eqref{eq.stat} extends to all $F\in\mathscr C^\infty_{b,cyl}(\Phi')$. In particular, taking $F(\cdot)=\cos\langle \cdot,\varphi\rangle$ and $F(\cdot)=\sin\langle \cdot,\varphi\rangle$ in $\mathscr C^\infty_{b,cyl}(\Phi')$ one obtains  for any $\varphi\in {\rm span}\{e_j,j\in\N\}$,
\begin{equation}
\label{eq.28}
\int_{\Phi'} \cos\langle u,\varphi\rangle \, \langle \varphi,X(u)\rangle \,d\mu=\int_{\Phi'} \sin\langle u,\varphi\rangle \, \langle \varphi,X(u)\rangle \,d\mu=0\,.
\end{equation}
Hence, the equality \eqref{eq.27} is proved for all $\varphi\in \Phi$ by a density argument. Conversely, if the identity  \eqref{eq.27} holds true then  taking real and imaginary parts one gets \eqref{eq.28}. Hence, using Lemma \ref{lem.nab} one recovers  \eqref{eq.stat}.
\end{proof}

\begin{lemma}
\label{lem.5}
Let $\mu\in\mathfrak{P}(\Phi')$ and assume that for  any $\varphi\in\Phi$ there exists $0<C<1$ such that for all $k\in\N$,
\begin{equation}
\label{eq.29}
\int_{\Phi'} \big|\langle u,\varphi\rangle^k\, \langle \varphi, X(u)\rangle\big|\;d\mu\leq C^{k+1} k!\;.
\end{equation}
Then $\mu$ is a stationary solution of the Liouville equation \eqref{eq.stat} if and only if for all
$k\in\N$ and $\varphi\in\Phi$,
$$
\int_{\Phi'} \langle u,\varphi\rangle^k \, \langle \varphi, X(u)\rangle \;d\mu=0\,.
$$
\end{lemma}
\begin{proof}
Thanks to the hypothesis \eqref{eq.29}, the function
$$
m:\lambda\mapsto \int_{\Phi'} e^{i \lambda\langle u,\varphi\rangle }\, \langle \varphi, X(u)\rangle \;d\mu\,,
$$
is analytic on the disc $D=\{\lambda\in\C: |\lambda|<C^{-1}\}$.  Hence, using Lemma \ref{lem.6} the measure $\mu$ is a stationary solution of the Liouville equation if and only if
$$
\frac{d^km}{d\lambda^k}(0)= i^k\int_{\Phi'}  \langle u,\varphi\rangle^k\, \langle \varphi, X(u)\rangle \;d\mu\,=0\,.
$$
\end{proof}

Similarly, one proves the following result concerning the equilibrium KMS solutions of the Liouville equation.
\begin{lemma}
\label{lem.7}
Let $\mu\in\mathfrak{P}(\Phi')$ and assume that for  any $\varphi_1,\varphi_2\in\Phi$ there exists $0<C<1$ such that for all $k\in\N$,
\begin{equation}
\label{eq.31}
\int_{\Phi'} \big|\langle u,\varphi_2\rangle\big|^k\;d\mu\leq C^{k+1} k!\,, \qquad
\int_{\Phi'} \big|\langle u,\varphi_2\rangle^k\, \langle \varphi_1, X(u)\rangle\big|\;d\mu\leq C^{k+1} k!\;.
\end{equation}
Then $\mu$ is a $(\beta,X)$-KMS state if and only if for all
$k\in\N$ and $\varphi_1,\varphi_2\in\Phi$,
\begin{equation}
\label{lem.7_identity}
\beta \int_{\Phi'} \langle u,\varphi_2\rangle^k \, \langle \varphi_1, X(u)\rangle \;d\mu=k
\, \langle \varphi_1, J\varphi_2\rangle\;
\int_{\Phi'} \langle u,\varphi_2\rangle^{k-1}  \;d\mu\,;
\end{equation}
and
\begin{equation}
\label{lem.7_identity_2}
\int_{\Phi'} \langle\varphi_1, X(u)\rangle  \;d\mu=0\,.
\end{equation}
\end{lemma}
\begin{proof}
Since the functions
$$
m_1:\lambda\mapsto \int_{\Phi'} e^{i \lambda\langle u,\varphi_2\rangle }\, \langle \varphi_1, X(u)\rangle \;d\mu\,,\qquad m_2:\lambda\mapsto \int_{\Phi'} e^{i \lambda\langle u,\varphi_2\rangle }\;d\mu\,,
$$
are analytic on the disc $D=\{\lambda\in\C: |\lambda|<C^{-1}\}$, one can replace $\varphi_j$ by $\lambda \varphi_j$ for $j=1,2$, and expand the identity
\eqref{eq.KMSexp} as a $\lambda$-power series on $D$ to deduce a relation on the coefficients. Indeed, one has
$$
\sum_{k=0}^\infty \frac{i^k}{k!} \lambda^{k+1} \bigg(\langle \varphi_1, J\varphi_2\rangle \,
\int_{\Phi'} \langle u,\varphi_2\rangle^{k}  \;d\mu\bigg)+i\beta
\sum_{k=0}^\infty \frac{i^k}{k!} \lambda^{k} \bigg(
\int_{\Phi'} \langle u,\varphi_2\rangle^k \, \langle \varphi_1, X(u)\rangle \;d\mu\bigg)\,=0\,.
$$
Such an equation yields the claimed equilibrium moment relations.
\end{proof}

\medskip
Assume further that the operator $J$ defines a \emph{compatible complex structure} over the Hilbert space $H$, i.e.: For all $u,v\in H$
\begin{itemize}
\item  $J^2=-\mathds 1$;
\item  $\sigma(Ju,u)\geq 0$;
\item   $\sigma(Ju,Jv)=\sigma(u,v)$.
\end{itemize}
In particular, $H$ can be considered as a complex Hilbert space,
\begin{equation}
\label{complex_str_1}
(\eta+i\zeta) u:=\eta u-\zeta Ju\,,
\end{equation}
endowed with the inner product
\begin{equation}
\label{complex_str_2}
\langle u,v\rangle_{H_\C}=\langle u, v\rangle+ i\sigma(u,v)\,.
\end{equation}

We now state an equivalence between stationary solutions of the Liouville equation and stationary hierarchy equations, given by \eqref{eq.30} below.
\begin{proposition}
Assume that the Hilbert space $H$ is endowed with a complex structure as above and suppose that
$X(e^{i\theta}u)=e^{i\theta} X(u)$ for all $\theta\in\R$ and $u\in\Phi'$.
 Consider $\mu\in\mathfrak{P}(\Phi')$ which is $U(1)$-invariant and satisfies the estimates \eqref{eq.29}. Then $\mu$ solves \eqref{eq.stat} if and only if the symmetric hierarchy equation
\begin{equation}
\label{eq.30}
\sum_{j=1}^k \int_{\Phi'} \bigg(\big|u^{(j-1)}\otimes  X(u) \otimes u^{(k-j)}\big\rangle\big\langle u^{\otimes k}\big| \, + \,\big|u^{\otimes k}\big\rangle\big\langle u^{(j-1)}\otimes  X(u) \otimes u^{(k-j)}\big| \bigg)\; d\mu=0\,,
\end{equation}
holds for all $k\in\N$.
\end{proposition}
\begin{remark}
We recall that $\mu\in\mathfrak{P}(\Phi')$ is said to be  $U(1)$-invariant if for any $B\in\mathfrak B(\Phi')$ and any $\theta\in\R$ we have,
$$
\mu(\{e^{i\theta} u, u\in B\})=\mu(B)\,.
$$
The identity  \eqref{eq.30} shall be understood in a weak sense, i.e.: For all $\psi_1,\psi_2\in {\rm Sym} ^{k}(\Phi)$, the integrals
$$
\int_{\Phi'} \big\langle\psi_1,u\otimes\cdots \otimes X(u) \otimes \cdots\otimes u\big\rangle\big\langle u^{\otimes k},
\psi_2\big\rangle \; d\mu\,,
$$
and
$$
\int_{\Phi'} \big\langle\psi_1, u^{\otimes k}\big\rangle \big\langle u\otimes\cdots \otimes X(u) \otimes \cdots\otimes u, \psi_2\big\rangle \; d\mu\,,
$$
are well-defined where ${\rm Sym} ^{k}(\Phi)$ is the $k$-fold algebraic symmetric tensor product of $\Phi$.
Moreover, the hierarchy equation  \eqref{eq.30} can be interpreted as a system of infinite coupled equations for the $k$-densities $\{\gamma^{(k)}\}_{k\in\N}$ defined in  the weak sense by
$$
 \gamma^{(k)}=\int_{\Phi'} \,|u^{\otimes k}\rangle\langle u^{\otimes k}| \; d\mu\,,
$$
as  sesquilinear   maps on ${\rm Sym}^k(\Phi)\times {\rm Sym}^k(\Phi)$. For more details on this formulation and relationship with Gross-Pitaevskii hierarchies,  we refer the reader to \cite[Section 3]{Ammari:2018aa}.
\end{remark}

\begin{proof}
Suppose that $\mu$ is a stationary solution of the Liouville equation \eqref{eq.stat}. According to Lemma \ref{lem.5}, one has
$$
\int_{\Phi'} (\Ree\langle u,\varphi\rangle_{H_\C})^k \; \Ree\langle \varphi, X(u)\rangle_{H_\C} \;d\mu=0\,.
$$
Hence, the $U(1)$-invariance of $\mu$ and a polynomial expansion yield,
\begin{eqnarray*}
0=\sum_{j=0}^k \int_{\Phi'}  \binom{k}{j} \langle \varphi, e^{i\theta} u\rangle_{H_\C} ^j
\langle  e^{i\theta} u, \varphi\rangle_{H_\C} ^{k-j} \bigg( \langle \varphi, e^{i\theta} X(u)\rangle_{H_\C}+ \langle e^{i\theta} X(u), \varphi\rangle_{H_\C}\bigg) \; d\mu\,,
\end{eqnarray*}
 All the terms in the above sum are zero because of the $U(1)$-invariance, except the ones obtained by taking $2j+1=k$ and $2j-1=k$ ($k$ should be odd), which can be seen by taking averages in $\theta \in [0,2\pi]$. Therefore, one obtains that for all $p\in\N$,
\begin{equation*}
\int_{\Phi'} \langle\varphi,u\rangle_{H_\C}^{p-1}  \, \langle u,\varphi\rangle_{H_\C}^{p} \, \langle\varphi,X(u)\rangle_{H_\C}\;d\mu+
\int_{\Phi'} \langle\varphi,u\rangle_{H_\C}^{p} \,\langle u,\varphi\rangle_{H_\C}^{p-1} \langle X(u),\varphi\rangle_{H_\C} \;d\mu=0\,,\\
\end{equation*}
so
\begin{eqnarray*}
0&=& \sum_{j=1}^p  \int_{\Phi'} \langle \varphi^{\otimes p},  u^{\otimes (j-1)}\otimes X(u)\otimes u^{\otimes (p-j)} \rangle_{H_\C} \,\langle u^{\otimes p}, \varphi^{\otimes p}\rangle_{H_\C} \\ &&\hspace{.7in} + \,\langle\varphi^{\otimes p},u^{\otimes p}\rangle_{H_\C} \,
\langle  u^{\otimes (j-1)}\otimes X(u)\otimes u^{\otimes (p-j)} , \varphi^{\otimes p}\rangle_{H_\C}\; d\mu\\
&=& \big\langle \varphi^{\otimes p},\sum_{j=1}^p \bigg( \int_{\Phi'}   \big|u^{\otimes (j-1)}\otimes X(u)\otimes u^{\otimes (p-j)} \rangle_{H_\C} \,\langle u^{\otimes p}\,\big| \\ &&\hspace{.7in} + \,\big|u^{\otimes p}\rangle_{H_\C} \,
\langle  u^{\otimes (j-1)}\otimes X(u)\otimes u^{\otimes (p-j)}\big|\,d\mu\bigg)\; \varphi^{\otimes p}\big\rangle_{H_\C}\\ &=&\big\langle \varphi^{\otimes p},Q_{p} \,\varphi^{\otimes p}\big\rangle_{H_\C}\,,
\end{eqnarray*}
where  $Q_{p}$ denotes the sum in the previous line, interpreted as a quadratic form on the $p$-fold algebraic symmetric tensor product space ${\rm Sym}^p(\Phi)$. Thanks to the polarization formula,
\[
\langle \eta^{\otimes p} , Q_p\, \xi^{\otimes p} \rangle_{H_\C} = \int_{0}^{1} \int_{0}^{1} \langle (e^{2i\pi \theta} \eta + e^{2i\pi \varphi}\xi)^{\otimes p} , Q_p\, (e^{2i\pi \theta} \eta + e^{2i\pi \varphi}\xi)^{\otimes p} \rangle_{H_\C} \;e^{2i\pi	(p\theta-p\varphi)} d\theta d\varphi \,,
\]
and the fact that any element in ${\rm Sym}^p(\Phi)$ can be written as combination of $\{\eta^{\otimes p},\eta\in \Phi\}$, one obtains the claimed hierarchy equation \eqref{eq.30}. The converse statement follows by reversing the above arguments.
\end{proof}

We end up this section with an equivalence result between KMS equilibrium states and equilibrium hierarchies.

\begin{thm}
\label{thm.hier}
Assume that the Hilbert space $H$ is endowed with a complex structure as above and suppose that
$X(e^{i\theta}u)=e^{i\theta} X(u)$ for all $\theta\in\R$ and $u\in\Phi'$.
 Consider a $U(1)$-invariant $\mu\in\mathfrak{P}(\Phi')$ satisfying the estimates \eqref{eq.31}. Then $\mu$ is a $(\beta,X)$-KMS state if and only if for all $\varphi_1,\varphi_2\in\Phi$ and $p\in\N$ we have that
\begin{equation}
\label{eq.32}
\frac{\beta}{p+1} \int_{\Phi'} \big\langle\varphi_2,u\big\rangle_{H_\C}^{p+1} \, \big\langle u,\varphi_2\big\rangle_{H_\C}^{p} \, \big\langle  X(u),\varphi_1 \big\rangle_{H_\C}\, d\mu=2i
\big\langle\varphi_2,\varphi_1\big\rangle_{H_\C}\,
\int_{\Phi'} \big\langle\varphi_2,u\big\rangle_{H_\C}^{p} \, \big\langle u,\varphi_2\big\rangle_{H_\C}^p \, d\mu\,.
\end{equation}
\end{thm}
\begin{proof}
Suppose that $\mu$ is a $(\beta,X)$-KMS state. By Lemma \ref{lem.7}, we have that
\eqref{lem.7_identity} holds. By using the $U(1)$ invariance of $X$ and $\mu$, we can rewrite this identity as
\begin{multline}
\label{thm.hier_1}
\frac{\beta}{2^{k+1}}\,\int_{\Phi'} \Big(e^{-i\theta}\,\big\langle u,\varphi_2 \big\rangle_{H_\C}+e^{i\theta}\,\big\langle \varphi_2,u \big\rangle_{H_\C}\Big)^k\,\Big(e^{i\theta}\,\big\langle \varphi_1, X(u) \big\rangle_{H_\C}+e^{-i\theta}\,\big\langle X(u), \varphi_1 \big\rangle_{H_\C}\Big)\,d\mu
\\
=\frac{k}{2^{k-1}}\, \Ree \big\langle \varphi_1,-i \varphi_2 \big\rangle_{H_\C} \, \int_{\Phi'} \Big(e^{-i\theta}\,\big\langle u, \varphi_2 \big\rangle_{H_\C}+e^{i \theta}\,\big\langle \varphi_2,u\big\rangle_{H_\C}\Big)^{k-1}\,d\mu,
\end{multline}
for $\theta \in [0,2\pi]$. We then take the average over $\theta \in [0,2\pi]$ in \eqref{thm.hier_1} to deduce that both sides vanish if $k$ is even and for $k=2p+1$ odd, by using the Newton binomial formula, the above identity is equivalent to
\begin{multline}
\label{thm.hier_2}
\frac{\beta}{p+1} \bigg[\int_{\Phi'} \big\langle\varphi_2,u\big\rangle_{H_\C}^{p+1} \, \big\langle u,\varphi_2\big\rangle_{H_\C}^p \, \big\langle X(u),\varphi_1\big\rangle_{H_\C}\, d\mu +
 \int_{\Phi'} \big\langle\varphi_2,u\big\rangle_{H_\C}^{p} \, \big\langle u,\varphi_2\big\rangle_{H_\C}^{p+1} \, \big\langle \varphi_1, X(u)\big\rangle_{H_\C}\, d\mu\bigg]
\\
=
2 \Big(i \langle \varphi_2,\varphi_1\rangle_{H_\C} -i \big\langle \varphi_1,\varphi_2 \big\rangle_{H_\C}\Big)\,
\int_{\Phi'} \big|\langle\varphi_2,u\big\rangle_{H_\C}|^{2p} \, d\mu
\,.
\end{multline}
The identity \eqref{thm.hier_2} holds for all $\varphi_1, \varphi_2 \in \Phi$. In particular, it holds if we replace $\varphi_2 \mapsto e^{i \theta} \varphi_2$ for any $\theta\in [0,2\pi]$. We hence deduce \eqref{eq.32}.
The converse follows by analogous arguments. Note that the identity \eqref{lem.7_identity_2} is true thanks to the $U(1)$ invariance of the measure $\mu$ and the vector field $X$.
\end{proof}

\section{Gaussian measures and KMS states}
\label{sec.lin}
We show in this section that Gaussian measures in infinite dimensional spaces  are fundamental examples of KMS equilibrium  states. It is possible to study Gaussian measures from different points of view. Here, we consider Gaussian measures on dual nuclear spaces, abstract Wiener spaces and Gaussian probability spaces. Our aim is to outline the fundamental aspects of KMS states and the emphasize their applicability in various contexts.

\subsection{Gaussian measures on  countably Hilbert nuclear spaces}
The general setting given in Subsection \ref{subsec.frw} will be restricted here since we are going to consider  Gaussian measures on the dual space $\Phi'$. Therefore it is useful to require that $\Phi$ is a suitable nuclear space. Before proceeding further, we give the precise assumptions on the spaces. We recall that $H$ is always assumed to be a separable real Hilbert space endowed with a symplectic structure $\sigma$ induced by the operator $J$ (see Subsection \ref{subsec.frw}) satisfying $J\Phi\subseteq \Phi$ and that $H$ is endowed with a Hilbert rigging
$$
\Phi\subseteq H \subseteq \Phi'\,,
$$
such that $\Phi$  is dense in $H$. Assume furthermore  that  $\Phi$ is a \emph{countably Hilbert nuclear space}.
This means that $\Phi$ is a Fr\'echet space whose topology is given by an increasing sequence of compatible Hilbertian norms $\{\|\cdot\|_n, n\in \N\}$ and such that taking  $H_n$ to be the completion of $\Phi$ with respect to the norm $\|\cdot\|_n$, one has  the chain of embeddings,
$$
\Phi\subseteq\cdots  \subseteq H_n\subseteq H_{n-1} \cdots\subseteq H_1\,,
$$
satisfying for all $n\in\N$ the existence of $m\in\N$, $m\geq n$, such  that the embedding
$$
i_{m,n}:(H_m,\|\cdot||_m) \to (H_n,||\cdot||_n)\,,
$$
defines a trace-class operator. Recall that the norms are said  to be \emph{compatible} if for any sequence $(x_k)_k$ in $\Phi$ that is Cauchy for both  $||\cdot||_n$ and $||\cdot||_m$, one has
$$
( \lim_k x_k=0 \text{ in } H_n ) \Leftrightarrow ( \lim_k x_k=0 \text{ in } H_m )\,.
$$
This shows in particular that $i_{m,n}$ is a well defined embedding and hence $H_m$ can be identified with a subset of $H_n$ whenever $n\leq m$. Moreover, the space $\Phi$ is identified with the topological projective limit associated to the projective system $(H_n, i_{n,m})$ such that
$$
\Phi=\bigcap_{n\in\N} H_n= {\displaystyle \varprojlim H_{n}}\,.
$$
For more details on nuclear spaces see e.g. \cite{MR0435834}. The main example for such a setting  is given by the rigging  $\mathscr S(\R^d)\subseteq L^2(\R^d)\subseteq \mathscr S'(\R^d)$ where $\mathscr S(\R^d)$ is a nuclear space endowed for instance with the sequence of norms:
$$
\| \varphi\|_n= \bigg(\sum_{|\alpha|\leq n} \big\| (1+|x|^2)^{n/2} D^\alpha \varphi(x)\big\|^2_{L^2(\R^d)}\bigg)^{1/2}\,.
$$

In this framework it is known that the Minlos theorem provides an elegant generalization of the Bochner theorem. The point is that the (canonical) Gaussian measures on infinite dimensional Hilbert spaces are   not $\sigma$-additive measures on $H$  but only additive cylindrical set measures. However, such cylindrical set measures  extend to  probability measures by means of a radonifying  embedding on a larger space. A convenient statement of the Minlos theorem is given below. Recall that a normalized  positive-definite functional $G:\Phi\to \C$ is a map satisfying:
\begin{itemize}
\item [(i)] $G(0)=1$;
\item [(ii)] For all $n\in\N$,  $\lambda_j\in\C$, $u_j\in\Phi$ for  $j=1,\dots,n$,
$$
\sum_{j,k=1}^n \bar \lambda_j \lambda_k \;G(u_j-u_k)\geq 0\,.
$$
\end{itemize}

\begin{thm}[Minlos' theorem]
\label{thm.minlos}
Assume that $\Phi$ is a countably Hilbert nuclear  space. Then any continuous normalized positive definite  functional $G$ on $\Phi$ is the characteristic function of a unique $\mu \in \mathfrak{P}(\Phi')$
such that for all $w\in \Phi$,
$$
G(w)=\int_{\Phi'} e^{i \langle u, w\rangle} \;d\mu\,.
$$
\end{thm}
For more details on the above theorem, we refer to \cite[Thm. 4.7]{MR1851117} and \cite[Chapter IV]{MR0435834}.

Consider a positive symmetric (bounded or unbounded) operator $A:D(A)\subseteq H\to H$ such that $A\geq c \mathds 1$ for some constant $c>0$ and $D(A)\supset\Phi$. In particular, $A$ is invertible with $A^{-1}$ being a bounded operator on $H$.
\begin{cor}
\label{cor.1}
Let $\beta>0$ be given.
 There exists a unique $\mu_{\beta,0}\in\mathfrak{P}(\Phi')$ such that its characteristic function is given for all $v\in H$ by
$$
\hat\mu_{\beta,0}(v)=\int_{\Phi'} e^{i \langle v, u\rangle} \;d\mu_{\beta,0}=e^{-\frac{1}{2\beta} \langle v, A^{-1} v\rangle}\,.
$$
\end{cor}

\medskip
Recall the spaces of cylindrical smooth functions \eqref{eq.18}, the gradient \eqref{grad.defn} and Poisson structure on $\mathscr C_{b,cyl}(\Phi')$ \eqref{eq.poi}, as well as the definition of KMS states in Definition \ref{KMS-def} from Subsection \ref{subsec.frw}.
\begin{thm}
\label{thm.nucgaus}
The Gaussian measure $\mu_{\beta,0}$ provided by Corollary \ref{cor.1} is a $(\beta,X)$-KMS state for the linear dynamical system given by the vector field $X=J A$.
\end{thm}
\begin{proof}
In order to prove that $\mu_{\beta,0}$ is a KMS state, we will use Theorem \ref{thm.char}. Let $\varphi_1,\varphi_2\in \Phi$ be given. Then, using the Cauchy-Schwarz inequality one easily checks that the function
$\langle \varphi_1, X(\cdot)\rangle$ is $\mu_{\beta,0}$-integrable,
$$
\int_{\Phi'} \big|\langle AJ\varphi_1, u \rangle\big|\, d\mu_{\beta,0}\leq \bigg( \int_{\Phi'} \langle AJ\varphi_1, u \rangle^2 \;d\mu_{\beta,0}\bigg)^{1/2}<\infty\,,
$$
since $AJ\varphi_1\in H$ and  all the second moments of the Gaussian measure $\mu_{\beta,0}$ are finite (i.e.: $\langle f,\cdot\rangle\in L^2(\mu_{\beta,0})$ for all $f\in H$, see Theorem \ref{thm.gauss} and Remark \ref{rem.mu}).
Using Corollary \ref{cor.1}, observe that
  \begin{eqnarray*}
  i\int_{\Phi'}\, \langle \varphi_1,X(u)\rangle \,  e^{i \langle u, \varphi_2\rangle} \, d\mu_{\beta,0} &=&
  \frac{d}{ds} \bigg(
  \int_{\Phi'}\,  \,  e^{i \langle u, -s A J\varphi_1+\varphi_2\rangle} \, d\mu_{\beta,0}\bigg) \,_{\big|s=0} \\
  &=&   \frac{d}{ds} \bigg(e^{-\frac{1}{2\beta} \langle -s A J\varphi_1+\varphi_2 , A^{-1}  (-s A J\varphi_1+\varphi_2)\rangle}\bigg) \,_{\big|s=0}\\
  &=& \frac{1}{\beta} \langle J\varphi_1,\varphi_2\rangle e^{-\frac{1}{2\beta} \langle \varphi_2 , A^{-1}  \varphi_2\rangle}\\
  &=& \frac{1}{\beta} \langle J\varphi_1,\varphi_2\rangle \;\int_{\Phi'} e^{i \langle u, \varphi_2\rangle} \, d\mu_{\beta,0}\,.
  \end{eqnarray*}
  This proves the identity \eqref{eq.KMSexp} and hence $\mu_{\beta,0}$ is a $(\beta,X)$-KMS state.
\end{proof}

\begin{remark}[White noise]
An interesting example for the above Theorem \ref{thm.nucgaus} is the so-called \emph{white noise measure}.
According to Minlos' Theorem \ref{thm.minlos}, there exists a unique probability measure $\mu_{wn}$ on $\mathscr S'(\R)$ having the characteristic functional
$$
 \hat \mu_{wn}(u)= \int_{\mathscr S'(\R)} e^{i \langle u, w\rangle} \;d\mu_{wn}(w)=e^{-\frac{1}{2}\| u\|_{L^2(\R)}^2}\,,
$$
 named the \emph{canonical Gaussian measure} or \emph{white noise measure} on $\mathscr S'(\R)$ corresponding to the choice $\beta=1$, $\Phi=\mathscr S(\R)$, $A=\mathds 1$ and $J$ is any operator inducing a non-degenerate symplectic structure on $L^2(\R)$ such that $J\mathscr S(\R)\subseteq\mathscr S(\R)$.
\end{remark}

\subsection{Wiener and Gaussian probability spaces}
The result in Theorem \ref{thm.nucgaus} extends to  abstract Wiener   spaces and  Gaussian probability spaces. Indeed, one can  prove that the canonical Gaussian measure in both cases  is a $(\beta,X)$-KMS state for $\beta=1$ and for a given linear dynamical system.

\medskip
\emph{Abstract Wiener space:}   Let $\mathbb B$ be a separable Banach space such that the Hilbert space $H$ is embedded into $\mathbb B$ through  an  injective continuous linear map $i : H \to \mathbb B$. Assume that
           the map $i$ radonifies the canonical Gaussian cylinder set measure on $H$. Then $(i,H, \mathbb B)$ is called an abstract Wiener space (see e.g. \cite{MR0265548,MR1851117}). This means that there exists a Borel probability measure $\mu_{ws}$ on $\mathbb B$ such that its characteristic function is given for all $u\in H$ by
            $$
            \hat\mu_{ws}(u)=\int_{\mathbb B} e^{i \langle u, w\rangle} \;d\mu_{ws}(w)=e^{-\frac{1}{2}\| u\|^2}\,.
            $$
            As in Theorem \ref{thm.nucgaus}, one shows that the \emph{canonical Gaussian measure} $\mu_{ws}$ on $\mathbb B$ is a $(\beta,X)$-KMS state  for the dynamical system induced by  the vector field $X=J$ with  $\beta=1$ and $J$ is any operator implementing  a symplectic structure on $H$.

\medskip
\emph{Gaussian probability space:}  is a complete probability space $(\Omega, \Sigma, \mathbb P)$ with a family of  centered Gaussian random variables $W(f):\Omega\to\R$ indexed by a separable Hilbert space $H$ such that for all $f,g\in H$,
     \begin{equation}
     \label{W_identity}
     \mathbb E( W(f) W(g))=\langle f, g\rangle\,.
     \end{equation}
     In particular, the map $f\in H\mapsto  W(f)\in L^2(\Omega,\mathbb P)$ is a linear isometry. For more details on Gaussian probability spaces, we refer the reader to the book \cite{MR2200233}. As before we are going to prove that the probability measure  $\mathbb P$ is a $(\beta,X)$-KMS state  for a certain dynamical system with an inverse temperature $\beta=1$. Let $\{e_j\}$ be an orthonormal basis of $H$ and define the linear operator $J$ as,
      \begin{equation}
      \label{J_e}
      J e_{2j-1}=e_{2j}, \;\;  J e_{2j}=-e_{2j-1}, \quad \forall j\in\N\,,
      \end{equation}
        Then, $J$ induces  a symplectic structure on $H$. Furthermore, consider $(\alpha_j)_{j\in\N}$ a sequence of positive real numbers such that
      \begin{equation}
      \label{sum_alpha_j_inverse}
      \sum_{j=1}^\infty \alpha_j^{-1}<\infty\,.
      \end{equation}
      Using this sequence, one can define a Hilbert rigging $H_+\subseteq H \subseteq H_-$ by taking
$$
H_+=\big\{ u\in H | \sum_{j=1}^\infty \alpha_j \langle u, e_j\rangle^2 <\infty\big\}\,,
$$
as a Hilbert space endowed with  the inner product given for any $u,v\in H_+$ by,
$$
\langle u,v\rangle_{H_+}=
\sum_{j=1}^\infty  \alpha_j\langle u, e_j\rangle \langle e_j, v\rangle \,;
$$
and considering  $H_-$ as the dual of $H_+$ with respect to the inner product of $H$. Remark that the norm on $H_-$ is given by,
$$
\|u\|_{H_-}=\bigg(\sum_{j=1}^\infty \alpha_j^{-1}\langle u,e_j\rangle^2\bigg)^{1/2}\,.
$$

We note the following analogue of Theorem \ref{thm.nucgaus} in the context of Gaussian spaces.
\begin{lemma}
\label{isonorm_lemma}
For all $f,g\in H_{+}$, we have
\begin{equation}
\label{isonorm}
\langle g,J f\rangle \, \mathbb E\Big(e^{i W(g)}\Big)+i\,\mathbb E\Big(W(-J f)\, e^{i W(g)}\Big)=0\,.
\end{equation}
\end{lemma}
\begin{proof}
The idea of the proof is similar to that of Theorem \ref{thm.nucgaus}, except that now we do not have a vector field at our disposal. Instead, we use the Gaussian structure. We start by observing that for all $f \in H$, we have
\begin{equation}
\label{Wick_rule_identity}
\mathbb{E} \Big(e^{i W(f)}\Big)=e^{-\frac{1}{2}\|f\|^2}\,.
\end{equation}
In order to deduce identity \eqref{Wick_rule_identity}, we note that by Wick's rule and \eqref{W_identity}, we have that for all $k \in \N$
\begin{align*}
\mathbb{E} \Big((W(f))^k\Big)=
\begin{cases}
\frac{(2n)!}{n!\,2^n} \|f\|^{2n} &\text{if } k=2n \text{ is even}
\\
0 &\text{if $k$ is odd.}
\end{cases}
\end{align*}
For $f,g$, we compute
\begin{equation}
\label{isonorm_1}
\frac{d}{ds} \,\mathbb{E} \Big(e^{iW(-sJf+g)}\Big)\,_{\big|s=0}=i\mathbb{E} \Big(W(-Jf)\,e^{iW(g)}\Big)\,.
\end{equation}
On the other hand, by using \eqref{Wick_rule_identity}, we can rewrite \eqref{isonorm_1}
as
\begin{equation}
\label{isonorm_2}
\frac{d}{ds} \,e^{-\frac{1}{2}\|-sJf+g\|^2}\,_{\big|s=0}= -\langle g,Jf \rangle \, \mathbb{E}\Big(e^{iW(g)}\Big)\,.
\end{equation}
The identity \eqref{isonorm} follows from \eqref{isonorm_1} and \eqref{isonorm_2}.
\end{proof}

In order to see the above identity  \eqref{isonorm}  as a KMS condition similar to \eqref{eq.KMSexp}, one needs to introduce a vector field $X$ that is interpreted as an element of the space $L^2(\Omega,\PP; H_-)$ of square integrable $H_{-}$-valued functions.
\begin{lemma}
\label{eq.vec}
 Let $ X_n=\sum_{j=1}^n W(e_j)\, Je_j\in L^2(\Omega,\PP; H)$.   Then the sequence $(X_n)_{n\in\N}$ converges  to an element $X\in L^2(\Omega; H_-)$, i.e.:
$$
X=\sum_{j=1}^\infty W(e_j) \,Je_j\,\in  L^2(\Omega,\PP;H_-)\,.
$$
\end{lemma}

\begin{proof}
It is enough to show that $(X_n)_{n\in\N}$ is a Cauchy sequence in $ L^2(\Omega,\PP;H_-)$. Indeed, one has
\begin{eqnarray*}
\|X_n-X_m\|^2_{L^2(\Omega,\PP;H_-)}&=& \int_\Omega \, \|X_n-X_m\|^2_{H_{-}}\;{\rm d}\mathbb P \\
&=&  \int_\Omega \,\sum_{j=1}^\infty \alpha_j^{-1} \langle X_n-X_m, e_j\rangle^2\;{\rm d}\mathbb P
\end{eqnarray*}
Using the definition of $X_n$ and applying \eqref{J_e}, one notices that  for $j\in\N$,
$$
\langle X_n-X_m, e_j\rangle= 1_{[m+2,n+1]\cap 2\N}(j)\;W(e_{j-1})-1_{[m,n-1]\cap 2\N+1}(j)\;W(e_{j+1})\,.
$$
Thus, one concludes  by \eqref{sum_alpha_j_inverse}
\begin{eqnarray*}
\|X_n-X_m\|^2_{L^2(\Omega,\PP;H_-)}=   \sum_{j=m+1}^n \alpha_j^{-1} \mathbb E\big(W(e_j)^2\big) \underset{n,m\to\infty}{\longrightarrow} 0\,.
\end{eqnarray*}
\end{proof}
As in Subsection \ref{subsec.frw}, one defines the class of smooth compactly supported cylindrical functions $F\in\mathscr{C}^\infty_{c,cyl}(\Omega)$ as all the functions $F:\Omega\to\R$ satisfying
$$
F=\varphi(W(e_1),\dots, W(e_n))\,,
$$
for some $n\in\N$ and $\varphi\in \mathscr C_{c}^\infty(\R^n)$. Similarly, one can introduce a gradient for these functions given as below,
$$
\nabla F=\sum_{j=1}^n \partial_j\varphi(W(e_1),\dots, W(e_n)) \;W(e_j) \,e_j\in L^2(\Omega,\PP;H)\,.
$$
Hence, one can also introduce a Poisson bracket for any $F,G\in\mathscr{C}^\infty_{c,cyl}(\Omega)$ as,
\begin{equation*}
\{F,G\}= \big\langle \nabla F, J\nabla G\big\rangle \in L^2(\Omega,\PP)\,.
\end{equation*}
\begin{proposition}
Let $(\Omega,\Sigma,\mathbb P)$ be a Gaussian probability space with $J$ and $X$ defined as before. Then $\mathbb P$ is a $(1,X)$-KMS state in the following sense: For all $F,G\in\mathcal{C}(\Omega)$,
\begin{equation}
\label{eq.24}
\mathbb E\big( \{F,G\}\big)= \mathbb E\big( \langle \nabla F, X\rangle \, G\big)\,.
\end{equation}
\end{proposition}
\begin{proof}
Since the map $f\mapsto W(f)$ is a linear isometry from $H$ to $L^2(\Omega,\PP)$, one checks that for $f \in H_{+}$
\begin{eqnarray*}
\langle f, X\rangle &=& \lim_n \sum_{j=1}^n \langle f, J e_j\rangle \,W(e_j) =  \lim_n  \,W\bigg( \sum_{j=1}^n \langle-J f, e_j\rangle e_j\bigg)=W(-J f)\,.
\end{eqnarray*}
Therefore, the equality \eqref{isonorm} reads,
$$
\langle g,J f\rangle \, \mathbb E\Big(e^{i W(g)}\Big)+i\,\mathbb E\Big(
       \langle f, X\rangle \, e^{i W(g)}\Big)=0\,.
$$
Following the same lines of the proof of Theorem \ref{thm.char}, one proves the KMS condition   \eqref{eq.24}.
\end{proof}

\begin{remark}
The identity  \eqref{eq.24} can  be regarded as a generalization of the KMS condition
\eqref{KMS-G} to Gaussian probability spaces or more generally to  stochastic processes.
\end{remark}

\section{The Gibbs-KMS equivalence}
\label{sec.KMSGibbs}
In this section, we address the problem of equivalence between Gibbs measures and KMS states. It is quite instructive to first consider finite dimensional dynamical systems since  they  provide significant insight into the problem. Afterwards, we  consider in Subsection \ref{sub.sec.sobolevset} the case of complex linear infinite dimensional dynamical systems; while nonlinear infinite dynamical systems are treated in the last Subsection \ref{subsec.nl}.

\subsection{Finite dimensional dynamical systems}
\label{sec:fdim}

Let $E$ be a  Hermitian space of  dimension $n$  endowed with a scalar product $\langle \cdot, \cdot\rangle$ which is anti-linear with respect to the left component. Fix an orthonormal basis $\{e_1,\dots,e_n\}$. One can consider $E$ as  a  \emph{Euclidean vector space} with respect to the scalar product
$$
\langle \cdot, \cdot \rangle_{E,\R}:=\Ree\langle \cdot, \cdot\rangle\,.
 $$
For convenience, we simply denote by $E_\R$  the Euclidean vector space  $(E, \langle \cdot, \cdot \rangle_{E,\R})$. Notice that if we set $f_j=i e_j$, for $j=1,\cdots,n$, then  $\{e_1,\dots,e_n, f_1,\dots,f_n\}$ is an  orthonormal basis of $E_\R$ and we have the decomposition
\begin{equation}
\label{eq.2}
E_\R= K\oplus i K \,,
\end{equation}
where $K={\rm span}_\R\{e_1,\dots,e_n\}$. Moreover, $E_\R$ is isomorphic to the direct sum $K\oplus K$ through the canonical $\R$-linear mapping:
\begin{equation}
\label{eq.4}
 \forall x,y\in K,\qquad  E_\R \ni x+iy \rightleftharpoons x\oplus y\in K\oplus K\,.
\end{equation}
Within this isomorphism the complex structure in $E$ is implemented  in $K\oplus K$ by the linear operator
\[
J=\left[
\begin{matrix}
0 & 1\\
-1 & 0
\end{matrix}
\right]
\]
 such that $J: K\oplus K \to K\oplus K$ and  $J u\oplus v =v\oplus -u$. In particular, $J^2=-\mathds{1}$ and $J$ corresponds, via the above isomorphism, to  the multiplication by the complex $-i$ on $E$.
In the sequel, we will sometimes use the identification $E_\R\simeq K\oplus K$ without making reference to the isomorphism \eqref{eq.4}.

\bigskip
\emph{Symplectic structure:}
The  Hermitian space  $E$ is naturally equipped with a canonical non-degenerate symplectic form:
$$
\sigma(\cdot,\cdot):=\Imm\langle \cdot, \cdot\rangle\,.
$$
In particular, the following relation holds true for all $u,v\in E$,
\begin{equation}
\label{eq.11}
\sigma(u, v)=\langle i u, v\rangle_{E,\R}=\langle u, J v\rangle_{K\oplus K}\,.
\end{equation}
Moreover, since $K=\{u\in E\mid \sigma (u,v)=0 {\mbox{ for all }} v\in K\}$ then $K$ is a Lagrangian subspace and the isomorphism \eqref{eq.4} provides a polarization of the phase-space $E$ into  canonical
position and momentum coordinates.

\bigskip
\emph{Poisson structure:}
Consider two smooth real-valued functions $F,G\in\mathscr{C}^\infty(E)$. The Poisson bracket is defined by,
\begin{eqnarray}
\label{Poisson_bracket_definition}
\{ F, G\}(u):= \sum_{j=1}^n \frac{\partial F}{\partial e_j} (u) \;\;\frac{\partial G}{\partial f_j} (u) -
\frac{\partial G}{\partial e_j} (u) \;\; \frac{\partial F}{\partial f_j} (u) \,,
\end{eqnarray}
where the partial derivatives are given by
$$
\frac{\partial F}{\partial e_j} (u)=\underset{\lambda\to 0, \lambda\in\R}{\lim} \frac{F(u+\lambda e_j)-F(u)}{
\lambda}\,, \qquad \frac{\partial G}{\partial f_j} (u)=\lim_{\lambda\to 0, \lambda\in\R} \frac{G(u+\lambda f_j)-G(u)}{\lambda}\,.
$$
Such a  bracket is skew symmetric  and satisfies both the Leibniz rule and the Jacobi identity.
It is sometimes useful to use the derivatives with respect to the complex coordinates. For this, we define the  Wirtinger derivatives\footnote[1]{The standard definition has  $1/2$ in front of the derivatives but here we overlook this factor.} by
\begin{eqnarray}
\label{eq.3}
\frac{\partial F}{\partial z_j} (u):= \frac{\partial F}{\partial e_j} (u)-i \frac{\partial F}{\partial f_j} (u)\,, \qquad \qquad \frac{\partial F}{\partial \bar z_j} (u) :=\frac{\partial F}{\partial e_j} (u)+i \frac{\partial F}{\partial f_j} (u)\,.
\end{eqnarray}
Hence, one can write the Poisson bracket as
$$
\{ F, G\}(u)= \frac{1}{2i} \;\sum_{j=1}^n \frac{\partial F}{\partial z_j} (u) \;\;\frac{\partial G}{\partial \bar z_j} (u) -
\frac{\partial G}{\partial z_j} (u) \;\; \frac{\partial F}{\partial \bar z_j} (u) \,.
$$
One can also write the Poisson bracket using the symplectic form $\sigma$ in \eqref{eq.11}. In fact, consider  a   Fr\'echet differentiable function $F: E\to\R$. Then its real differential is a $\R$-linear form given  for all $v\in K\oplus K$, such that $v=\sum_{j=1}^n  v_j e_j \oplus \sum_{j=1}^n w_j e_j$, by
$$
\nabla F(u)[v]= \sum_{j=1}^n  v_j \,\frac{\partial F}{\partial e_j}(u)  + w_j  \,\frac{\partial F}{\partial f_j}(u) \,.
$$
Hence, it  can be identified with the following  element of $K\oplus K\simeq E_R$,
\begin{equation}
\label{eq.10bis}
\nabla F(u)= \sum_{j=1}^n  \;\frac{\partial F}{\partial e_j}(u) \;e_j\oplus \frac{\partial F}{\partial f_j}(u) \;e_j\,.
\end{equation}
Thus, one checks that for all Fr\'echet differentiable functions  $F,G:E\to\R$,
\begin{equation}
\label{eq.12}
\{ F, G\}(u)=\sigma(\nabla F(u), \nabla G(u))\,.
\end{equation}

\bigskip
\emph{Hamiltonian system:}
Consider a  function $h: E_\R\simeq K\oplus K\to\R$ of class $\mathscr{C}^1$. Then as above, the differential of $h$ is given by,
\begin{equation}
\label{eq.10}
\nabla h(u)= \sum_{j=1}^n  \;\frac{\partial h}{\partial e_j}(u) \;e_j + \frac{\partial h}{\partial f_j}(u) \;f_j\,\equiv \sum_{j=1}^n  \;\frac{\partial h}{\partial e_j}(u) \;e_j\oplus \frac{\partial h}{\partial f_j}(u) \;e_j\,.
\end{equation}
Define the $\partial_{\bar z}$ operator as
$$
\partial_{\bar z} F(u)=\sum_{j=1}^n \frac{\partial F}{\partial \bar z_j}(u) \;e_j\,,
$$
then using the Wirtinger's derivatives in \eqref{eq.3}, one remarks
$$
-i \partial_{\bar z} h(u)\equiv  J\nabla h(u)\,.
$$

A Hamiltonian dynamical system on the phase-space $E$ is then  defined by means of the energy functional $h$ and the associated  continuous vector field $X: E \to E$ given for all $u\in E$ by
\begin{equation}
\label{X(u)}
X(u)=- i \partial_{\bar z} h(u)\equiv J \nabla h(u)\,.
\end{equation}
Indeed, the Hamiltonian system is governed by  the vector field equation,
\begin{equation}
\label{eq.1}
\dot{u}(t)=X(u(t))\,,
\end{equation}
where $u: I\subseteq \R\to E$ is a $\mathscr{C}^1$ curve and $I$ is a time interval. The differential equation \eqref{eq.1} is complemented by an initial condition $u(t_0)=u_0\in E$ at a fixed  initial time $t_0\in I$. Since the vector field $X$ is only continuous, one cannot apply the Cauchy-Lipschitz theorem and the existence of a smooth flow is not guaranteed. Nevertheless, the Peano existence theorem provides at least the existence of local solutions for the equation \eqref{eq.1}.

\bigskip
\emph{Gibbs measure:} In order to define the Gibbs measure for the above Hamiltonian system,  we assume that
\begin{equation}
\label{assump.1}
z_{\beta} := \int_E e^{-\beta h(u)} \, dL <+\infty\,,
\end{equation}
for some $\beta>0$ and where $ dL$ is the Lebesgue measure on $E$. In this case, we define the Gibbs measure of the Hamiltonian system  \eqref{eq.1}, at inverse temperature $\beta>0$, as the Borel probability measure given by
\begin{equation}
\label{Gibbs.fd}
\mu_\beta= \frac{ e^{-\beta h(\cdot)} \, dL}{\int_E e^{-\beta h(u)} \, dL}\equiv  \frac{1}{z_{\beta}}\,e^{-\beta h(\cdot)} \, dL\,.
\end{equation}
Notice that $z_{\beta}>0$. When the Hamiltonian system \eqref{eq.1} admits a smooth global flow, we know by the classical Liouville theorem that the Lebesgue and the Gibbs measures are invariant with respect to this flow.

\bigskip
\emph{KMS states:}
The general framework presented in Section \ref{subsec.frw} is applicable in the finite dimensional setting. We henceforth consider $\Phi=E_\R=\Phi'$ with the vector field $X:E\to E$ derived from  the Hamiltonian functional $h:E\to\R$ as in \eqref{X(u)} and consider the KMS states as in Definition \ref{KMS-def}.
Specifically, we say that $\mu \in \mathfrak{P}(E)$ is a $(\beta,X)$-KMS state if and only if:
\begin{equation}
\label{KMS.fdbis}
\int_E\, \{F,G\}(u) \; d\mu=\beta \,\int_E \,\Ree\langle \nabla F(u), X(u)\rangle \; G(u)\; d\mu\,,
\end{equation}
for any compactly supported smooth functions $F,G\in\mathscr{C}^\infty_c(E)$.
The following lemma is useful to express the above KMS condition in terms of the Poisson bracket.

\begin{lemma}
\label{lem1}
For any $F\in\mathscr{C}_c^\infty(E)$ and $u\in E$, we have
$$
\Ree \langle \nabla F(u), X(u)\rangle=\{F,h\}(u)\,.
$$
\end{lemma}
\begin{proof}
Using the isomorphism \eqref{eq.4} and  the identity \eqref{eq.10}, one can write the vector field \eqref{X(u)} as
\begin{equation*}
X(u)
=J\sum_{j=1}^n  \;\frac{\partial h}{\partial e_j}(u) \;e_j\oplus \frac{\partial h}{\partial f_j}(u) \;e_j
= \sum_{j=1}^n  \;\frac{\partial h}{\partial f_j}(u) \;e_j\oplus -\frac{\partial h}{\partial e_j}(u) \;e_j\,.
\end{equation*}
Similarly, using \eqref{eq.10bis}, one obtains
\begin{equation*}
\Ree \langle X(u), \nabla F(u)\rangle=\sum_{j=1}^n  \;\frac{\partial F}{\partial e_j}(u) \;\frac{\partial h}{\partial f_j}(u) - \frac{\partial h}{\partial e_j}(u) \; \frac{\partial F}{\partial f_j}(u)
=\{F,h\}(u)\,.
\end{equation*}
\end{proof}
Thus, by Lemma \ref{lem1}, the KMS condition \eqref{KMS.fdbis} is equivalent to the identity
\begin{equation}
\label{KMS.fd}
\int_E\, \{F,G\}(u) \; d\mu=\beta \,\int_E \,\{F,h\}(u) \; G(u)\; d\mu\,,
\end{equation}
for any compactly supported  smooth functions $F,G\in\mathscr{C}^\infty_c(E)$.

\begin{thm}
\label{thm.fd}
Consider a function $h:E\to\R$ of class $\mathscr{C}^1$ on the phase-space $E$ and assume that \eqref{assump.1} holds  for some $\beta>0$. Then $\mu \in \mathfrak{P}(E)$ satisfies the KMS condition \eqref{KMS.fd} if and only if $\mu$ is the Gibbs measure $\mu_\beta$ in \eqref{Gibbs.fd}.
\end{thm}
\begin{proof}
Let us check that the Gibbs measure $\mu_\beta$ satisfies the KMS condition \eqref{KMS.fd}. Indeed, using the Fubini theorem and an integration by parts one shows,
\begin{eqnarray*}
\int_E\, \frac{\partial F}{\partial e_j} (u) \;\;\frac{\partial G}{\partial f_j} (u) \; d\mu_\beta &=& - \frac{1}{z_\beta} \int_E\,  G(u)\, \frac{\partial }{\partial f_j} \left( \frac{\partial F}{\partial e_j} (u)
\; e^{-\beta h(u)} \right)\; dL\\
&=& - \int_E\, G(u) \, \frac{\partial^2 F}{\partial f_j\partial e_j} (u) \;d\mu_\beta +\beta
\int_E\, G(u)\, \frac{\partial F}{\partial e_j} (u) \; \frac{\partial h}{\partial f_j}(u)\, d\mu_\beta\,,
\end{eqnarray*}
and
\begin{eqnarray*}
\int_E\, \frac{\partial G}{\partial e_j} (u) \;\;\frac{\partial F}{\partial f_j} (u) \; d\mu_\beta &=& - \frac{1}{z_\beta}  \int_E\,  G(u)\, \frac{\partial }{\partial e_j} \left( \frac{\partial F}{\partial f_j} (u)
\; e^{-\beta h(u)} \right)\; dL\\
&=& - \int_E\, G(u) \, \frac{\partial^2 F}{\partial e_j\partial f_j} (u) \;d\mu_\beta +\beta
\int_E\, G(u)\, \frac{\partial F}{\partial f_j} (u) \; \frac{\partial h}{\partial e_j}(u)\, d\mu_\beta\,.
\end{eqnarray*}
Hence, by
\eqref{Poisson_bracket_definition}, we have
$$
\int_E\, \{F,G\}(u) \; d\mu=\beta \,\int_E \,\{F,h\}(u) \; G(u)\; d\mu_\beta\,.
$$
Conversely, consider a Borel probability measure $\mu$ such that the KMS condition \eqref{KMS.fd} is satisfied. Then remark that for any $F,G\in \mathscr{C}_{c}^\infty(E)$, we have by the Leibniz rule that
$$
\big\{F , G e^{-\beta h(u)}\big\}= \big\{F, G \big\}\, e^{-\beta h(u)}-\beta \big\{F,h\big\} \, G(u) e^{-\beta h(u)}\,.
$$
Therefore,
$$
 \big\{F , G e^{-\beta h(u)}\big\} \, e^{\beta h(u)}= \big\{F, G \big\}-\beta \big\{F,h\big\} \, G(u) \,,
$$
and notice that the above right hand side is integrable with respect to the measure $\mu$.
Hence, the KMS condition \eqref{KMS.fd}  gives
$$
\int_E \,   \big\{F e^{-\beta h(u)}, G \big\} \; e^{\beta h(u)} \; d\mu=0\,.
$$
Since $ e^{\beta h(\cdot)}$ is a positive Borel function, the map
$$
B\mapsto\nu(B):=\int_{B} e^{\beta h(u)} d\mu\,,
$$
defined for all Borel sets $B$ of $E$, gives a Borel measure on $E$. So, one obtains that  for any $F,G\in \mathscr{C}_{c}^\infty(E)$,
$$
\nu\left( \big\{F e^{-\beta h(u)}, G \big\}\right)=\int_{E} \big\{F e^{-\beta h(u)}, G \big\}(u) \,d\nu=0\,.
$$
But since the classical Hamiltonian $h$ is a $\mathscr{C}^1$-function, one obtains for all $F\in \mathscr{C}_c^1(E)$ and $G\in \mathscr{C}_c^\infty(E)$,
$$
\nu\left( \big\{F, G \big\}\right)=0\,.
$$
This condition implies that $\nu$ is a multiple of the Lebesgue measure. Indeed, take
$G(\cdot)=\Ree\langle e_j,  \cdot \rangle\; \varphi(\cdot)$ or $ G(\cdot)=\Ree\langle    f_j, \cdot \rangle\; \varphi(\cdot)$ with $\varphi\in  \mathscr{C}_0^\infty(E)$ being equal to $1$ on an open set containing the support of $F$. Then the Poisson brackets give,
$$
\{ F,G\}=- \frac{\partial F}{\partial f_j}(u)\,, \qquad \text{ or } \qquad \{ F,G\}=\frac{\partial F}{\partial e_j}(u)\,.
$$
So, in a distributional sense the derivatives in all the directions of the measure $\nu$ are zero and therefore $d\nu=c \,dL$ for some constant $c>0$.  Using the normalization condition for $\mu$, one concludes that
$$
c^{-1}=\int_E \, e^{-\beta h(u)} \, dL=z_{\beta}\,,
$$
and consequently
$$
\mu=\frac{1}{z_{\beta}}\,e^{-\beta h(\cdot)} \,dL \;  =\mu_\beta\,.
$$

\end{proof}

\begin{remark}
Later on we will see that the above  Theorem \ref{thm.fd} can be extended  to non-smooth vector fields. Indeed, one notes that Theorem \ref{thm.KMSGibbs} below applies with minor modifications to  the finite dimensional setting.
\end{remark}

\subsection{Linear infinite dimensional dynamical systems}
\label{sub.sec.sobolevset}
For applications in PDEs it is convenient to work in a more concrete setting than the one from Section \ref{sec.lin}. In particular, we suppose that $H$ is a separable complex Hilbert space. Hence,  $H$ is naturally equipped with a natural symplectic structure $\sigma(\cdot,\cdot)=\Imm\langle \cdot,\cdot\rangle$, a real scalar product $\langle \cdot,\cdot\rangle_{H,\R}:=\Ree\langle \cdot,\cdot\rangle$ and  a compatible complex structure. Note that $H$ as a real Hilbert space will be denoted by  $H_\R$.

\medskip

\medskip
\emph{Complex linear Hamiltonian system:}
Consider a  positive  operator $A:D(A)\subseteq H\to H$ such that,
\begin{equation}
\label{assum.inf.1}
\exists c>0, \quad A  \geq c \mathds{1}\,.
\end{equation}
 The linear Hamiltonian dynamical system is given by the quadratic energy functional,
\begin{equation}
\label{fre-Ham1}
h:D(A^{1/2})\to \R, \qquad h(u)=\frac{1}{2} \langle u, A u\rangle\,.
\end{equation}
So, the  vector field  in this case is the linear operator  $X_0: D(A)\to H $,
$$
X_0(u)=-i A u,
$$
leading to the linear differential  equation governing the dynamics of the system,
\begin{equation}
\label{fre-Ham2}
\dot{u}(t)=X_0(u(t))=-i A u(t)\,.
\end{equation}

\medskip
\emph{Compact resolvent:}  We suppose additionally that the operator  $A$ admits a compact resolvent. Therefore, there exists an orthonormal basis of $H$ composed of eigenvectors $\{e_j\}_{j\in\N}$ of $A$ associated respectively to their eigenvalues $\{\lambda_j\}_{j\in\N}$ such that  for all $j\in\N$,
\begin{equation}
\label{assum.inf.3}
A e_j=\lambda_j \, e_j\,.
\end{equation}
Furthermore, assume the following assumption:
\begin{equation}
\label{assum.inf.2}
\exists s\geq 0 : \quad \sum_{j=1}^\infty \frac{1}{\lambda_j^{1+s}} <+\infty.
\end{equation}
Remark that if we set $ f_j=i \, e_j$ for all $j\in\N$, then  $\{e_j, f_j\}_{j\in\N}$ is an O.N.B of $H_\R$.

\bigskip
\emph{Weighted Sobolev spaces:}
One can introduce weighted Sobolev spaces  using the operator $A$ as follows.  For any $r\in\R$, define  the inner product:
\[
\forall x,y \in \mathcal{D}(A^{\frac{r}{2}}) ~,\qquad \langle x,y \rangle_{H^r} := \langle A^{r/2} x,A^{r/2}y \rangle\,.
\]
Let $H^{s}$ denote  the Hilbert space $(\mathcal{D}(A^{s/2}), \langle \cdot,\cdot\rangle_{H^s})$ where $s\geq 0$ is the exponent in \eqref{assum.inf.2}, while $H^{-s}$ denotes the completion of the pre-Hilbert space $(\mathcal{D}(A^{-s/2}), \langle \cdot,\cdot\rangle_{H^{-s}})$.  Hence, one has the canonical continuous and dense embeddings (Hilbert rigging),
\begin{equation}
\label{eq.22}
H^{s} \subseteq H \subseteq H^{-s}\,.
\end{equation}
Remark that $H^{-s}$ identifies also with the dual space of $H^{s}$  relatively to the inner product of $H$.

\bigskip
\emph{Cylindrical smooth functions:} Using the O.N.B. $\{e_j,f_j\}_{j\in\N}$, one considers  the spaces of smooth cylindrical functions as in Subsection \ref{subsec.frw}.  More specifically, consider for $n\in\N$ the following mapping $\pi_n:H^{-s} \to  \R^{2n}$ given by
\begin{equation}
\label{eq.pi}
\pi_n(x)=(\langle x, e_1\rangle_{H,\R}, \dots, \langle x, e_n\rangle_{H,\R}; \, \langle x, f_1\rangle_{H,\R}, \dots, \langle x, f_n\rangle_{H,\R})\,.
\end{equation}
Then we define  $\mathscr{C}_{c,cyl}^{\infty}(H^{-s})$, respectively $\mathscr{C}_{b,cyl}^{\infty}(H^{-s})$, as the set of  all functions  $F:H^{-s}\to \R$  such that
\begin{equation}
\label{eq.7}
F=\varphi \circ \pi_n
\end{equation}
for some $n\in\N$ and $\varphi\in \mathscr C_c^\infty(\R^{2n})$, respectively $\varphi\in \mathscr C_b^\infty(\R^{2n})$.   In particular,
the  gradient  of $F$ at the point $u\in H^{-s}$ is given by
\begin{equation}
\label{eq.5}
\nabla F(u)=\sum_{j=1}^n \partial_j^{(1)} \varphi(\pi_n(u)) \, e_j+ \partial_j^{(2)} \varphi(\pi_n(u)) \,f_j \in\ H^{-s}\,.
\end{equation}
where $\partial_1^{(1)} \varphi,\dots, \partial_n^{(1)} \varphi$ and $\partial^{(2)}_1 \varphi,\dots, \partial^{(2)}_n \varphi $ are the partial derivatives of $\varphi$ with respect to the $n$ first and $n$ second coordinates  respectively. It also useful to introduce the following mapping,
\begin{equation}
\label{eq.proj}
\begin{aligned}
P_n: H^{-s}&\rightarrow  E_n\\
u &\mapsto \sum_{j=1}^n \langle e_j, u\rangle \; e_j\,,
\end{aligned}
\end{equation}
where $E_n= {\rm span}_\mathbb{C} \{e_1,\dots,e_n\}$ a finite dimensional  subspace of $H^{-s}$.
  The Euclidean structure of $E_n$ is the canonical one such that  $\{e_j,f_j\}_{j=1\dots,n}$ is an  O.N.B.

\bigskip
\emph{Poisson structure:}
We now precisely describe the Poisson structure over the algebra of smooth cylindrical functions $\mathscr{C}_{b,cyl}^{\infty}(H^{-s})$. Consider  $F,G\in \mathscr{C}_{b,cyl}^{\infty}(H^{-s})$ such that for all $u\in H^{-s}$,
\begin{equation}
\label{eq.9}
F(u)=\varphi\circ\pi_n(u)\,, \qquad \qquad G(u)=\psi\circ\pi_m(u)\,,
\end{equation}
 where $\varphi\in \mathscr{C}_{b}^{\infty}(\R^{2n})$ and $\psi\in \mathscr{C}_{b}^{\infty}(\R^{2m})$ for some $n,m\in \N$. Then, for all such $F,G\in \mathscr{C}_{b,cyl}^{\infty}(H^{-s})$,
  \begin{equation}
 \label{pois.bra}
\{F,G\}(u):= \sum_{j=1}^{\min(n,m)} \partial^{(1)}_{j}\varphi(\pi_n (u)) \;\partial^{(2)}_{j}\psi(\pi_m (u))
- \partial^{(1)}_{j}\psi(\pi_m (u)) \; \partial^{(2)}_{j}\varphi(\pi_n (u))\,.
\end{equation}

\bigskip
\emph{Gibbs measure:}
The Hamiltonian system \eqref{fre-Ham1}-\eqref{fre-Ham2}  admits a Gibbs measure at inverse temperature $\beta>0$, formally given by
$$
\mu_{\beta,0}\equiv\frac{ e^{-\beta h(\cdot)} \, du}{\int e^{-\beta h(u)} \, du}\,,
$$
and rigourously defined as a Gaussian measure on the Hilbert space $H^{-s}$ for the exponent $s\geq 0$ such that the assumption \eqref{assum.inf.2} is satisfied.
Recall that one says that $m\in H^{-s}$  is the \emph{mean-vector} of  $\mu \in \mathfrak{P}(H^{-s})$ if for any $f\in H^{-s}$ the function $u\mapsto \langle f, u\rangle_{H^{-s},\R}$ is $\mu$-integrable and
$$
\langle f, m\rangle_{H^{-s},\R}=\int_{H^{-s}} \, \langle f, u\rangle_{H^{-s},\R}  \; d\mu\,.
$$
When $m=0$, one says that $\mu$ is a \emph{zero-mean} or \emph{centered} measure. Additionally, the \emph{covariance operator} of the Borel probability measure $\mu$ on  $H^{-s}$ is a linear operator $Q: H^{-s}_\R\to H^{-s}_\R$ such that for any $f,g\in H^{-s}$ the function $u\mapsto \langle f , u\rangle_{H^{-s},\R}  \,\langle u , g\rangle_{H^{-s},\R} $ is $\mu$-integrable and
$$
\langle f, Q \,g\rangle_{H^{-s},\R}=\int_{H^{-s}}  \,\langle f , u-m\rangle_{H^{-s},\R}  \,\langle u-m , g\rangle_{H^{-s},\R} \;d\mu\,.
$$
For more details on Gaussian measures over Hilbert spaces, we refer the reader to the book by  Bogachev \cite[Chapter 2]{MR1642391}. In particular, the following result is well-known.
\begin{thm}
\label{thm.gauss}
Let $\beta>0$ and assume that the assumptions \eqref{assum.inf.1} and \eqref{assum.inf.2} are satisfied. Then there exists a unique zero-mean Gaussian measure on $H^{-s}$, denoted $\mu_{\beta, 0}$, such that its covariance operator is $ \beta^{-1} A^{-(1+s)}$, i.e.: for all $f,g\in H^{-s}$
\begin{equation}
\label{eq.cor}
\frac 1 \beta \langle f, A^{-(1+s)} g\rangle_{H^{-s},\R}=\int_{H^{-s}}  \,\langle f , u\rangle_{H^{-s},\R}  \;\langle u , g\rangle_{H^{-s},\R} \;d\mu_{\beta,0}\,,
\end{equation}
or equivalently for all $f,g\in H^s$,
\begin{equation}
\label{eq.corbis}
\frac 1 \beta \langle f, A^{-1} g\rangle_{H,\R}=\int_{H^{-s}}  \,\langle f , u\rangle_{H,\R}  \;\langle u , g\rangle_{H,\R} \;d\mu_{\beta,0}\,.
\end{equation}
 Moreover, the characteristic function of $\mu_{\beta,0}$ is given  for any $w\in H^{s}$ by,
\begin{equation}
\label{eq.foH}
\hat\mu_{\beta,0}(w)=\int_{H^{-s}} \,e^{i \langle w,u\rangle_{H,\R}} \; d\mu_{\beta,0}=
e^{-\frac{1}{2\beta} \langle w, A^{-1} w\rangle_{H}}\,,
\end{equation}
or equivalently for any $v\in H^{-s}$ we have
\begin{equation}
\label{eq.foHs}
\int_{H^{-s}} \,e^{i \langle v,u\rangle_{H^{-s},\R} }\; d\mu_{\beta,0}=
e^{-\frac{1}{2\beta} \langle v, A^{-(1+s)} v\rangle_{H^{-s}}}\,.
\end{equation}
\end{thm}
\begin{remark}
\label{rem.mu} The following observations are useful.
\begin{itemize}
\item [(i)]  The Gaussian measure $\mu_{\beta,0}$ given above coincides with the one  provided by Corollary \ref{cor.1} if one considers $\Phi=\displaystyle{\cap_{r>0} D(A^{r})}$. In particular, it is not difficult to  prove that $\Phi$ is a  countably Hilbert nuclear space and   $\mu_{\beta,0}(H^{-s})=1$.
\item [(ii)] Remark that in particular, one has
\begin{equation}
\label{eq.trac}
\tr_{H_{\R}}[\beta^{-1} A^{-(1+s)}]=\sum_{j=1}^\infty \frac{1}{\beta\lambda_j^{1+s}}= \int_{H^{-s}}  \,\|u\|^2_{H^{-s}} \;d\mu_{\beta,0}\,.
\end{equation}
\item [(iii)] Note that, according to \eqref{eq.corbis}, the random variable  $\langle f, \cdot\rangle_{H,\R}  \in  L^2(\mu_{\beta,0})$ for all $f\in H^{-1}$ in the sense that $\langle f, \cdot\rangle_{H,\R} :=\lim_n\langle f_n, \cdot\rangle_{H,\R}  $ in $L^2(\mu_{\beta,0})$ with $(f_n)_{n\in\N}$ any norm approximating sequence in $H^{s}$ of $f\in H^{-1}$.
\end{itemize}
\end{remark}
It is convenient  to characterize $\mu_{\beta,0}$ using a position and momentum coordinates system.
So, we define two sequences of image measures given by
$$
\nu^n_{\beta,0}:=(\pi_n)_\sharp\mu_{\beta,0},\qquad \text{ and }  \qquad \mu_{\beta,0}^n:={\displaystyle (P_n)_\sharp\mu_{\beta,0}}
$$
respectively on $\R^{2n}$  and $E_n$.  Here, $(\cdot)_{\sharp}$ denotes the pushforward. We can explicitly compute these measures.
\begin{lemma}
\label{lem.2}
The Gaussian measure $\mu_{\beta,0}$ satisfies the following relations for all $n\in\N$.
\begin{eqnarray}
\label{eq.13}
\nu_{\beta,0}^n&=(\pi_n)_\sharp \mu_{\beta,0}=&\displaystyle \prod_{j=1}^n \frac{\beta\lambda_j}{2 \pi} \;e^{- \beta \frac{\lambda_j}{2} (x_j^2+y^2_j)}  \; dx_j dy_j \,,\\
\mu_{\beta,0}^n&=(P_n)_\sharp \mu_{\beta,0}= &\frac{e^{-\frac{\beta}{2} \langle \cdot , A \cdot\rangle} \,dL_{2n}}{\int_{E_n} \, e^{-\frac{\beta}{2} \langle u , A u\rangle} \, dL_{2n}}\,,
\label{eq.13bis}
\end{eqnarray}
where $L_{2n}$ is the Lebesgue measure on the Euclidean space $E_n$ of dimension $2n$.
\end{lemma}

\begin{proof}
Let $\kappa_n$ denote the measure in the right hand side of \eqref{eq.13}.
One easily computes the characteristic function of $\kappa_n$,
\begin{eqnarray*}
\hat\kappa_n(\zeta_1,\dots,\zeta_n; \eta_1,\dots,\eta_n)&=&\int_{\R^{2n}} e^{i \sum_{j=1}^n
(x_j\zeta_j+y_j\eta_j)} \; d\kappa_n\\
&=&\prod_{j=1}^{n} e^{-\frac{1}{2\lambda_j \beta}(\zeta_j^2+\eta_j^2)}\,.
\end{eqnarray*}
On the other hand, by Theorem \ref{thm.gauss}, one checks
\begin{eqnarray*}
\hat\nu_{\beta,0}^n(\zeta_1,\dots,\zeta_n; \eta_1,\dots,\eta_n)&=&\displaystyle\int_{H^{-s}}
e^{i \sum_{j=1}^n (\langle u, e_j\rangle_{H,\R}\,\zeta_j+\langle u,f_j\rangle_{H,\R}\,\eta_j)}
\; d\mu_{\beta,0}\\
&=& \hat\kappa_n(\zeta_1,\dots,\zeta_n; \eta_1,\dots,\eta_n)\,.
\end{eqnarray*}
So, this shows that $\kappa_n=\nu_{\beta,0}^n$. Similarly, to prove the second relation, it is enough to note that the characteristic function of the right hand side of \eqref{eq.13bis} is given by
$$
e^{-\frac 1 2 \langle w, (\beta A)^{-1} w\rangle_{H,\R}}=\hat \mu_{\beta,0}^n(w)\,, \qquad\forall w\in E_n,
$$
where the equality follows from \eqref{eq.foH}.
\end{proof}

Let $i_n: E_n\to H^{-s}$, $i_n(u)=u$ for all $u\in E_n$,   be the canonical  embedding of $E_n$  into $H^{-s}$. It is useful to introduce for any  $\mu\in\mathfrak{P}(H^{-s})$ the image measures on $H^{-s}$ given by
\begin{equation}
\label{mun}
\mu_n=(i_n\circ P_n)_\sharp\mu\,.
\end{equation}
\begin{lemma}
\label{lem.3}
For any $\mu\in\mathfrak{P}(H^{-s})$, the sequence $(\mu_n)_{n\in\N}$ converges   narrowly  to $\mu$ on $\mathfrak{P}(H^{-s})$.
\end{lemma}
\begin{proof}
We note that the maps  $i_n\circ P_n:H^{-s}\to H^{-s}$ are linear and continuous. One checks that for all $u\in H^{-s}$,
$$
||i_n\circ P_n(u)-u||_{H^{-s}}^2=\sum_{j=n+1} ^\infty \lambda_j^{-s} (\langle u, e_j\rangle_{H,\R}^2
+\langle u, f_j\rangle_{H,\R}^2)\,.
$$
Hence, this proves that the sequence  $(i_n\circ P_n(u))_{n\in\N}$ converges towards $u$ in $H^{-s}$. Consequently,  one shows  that by dominated convergence, for any continuous bounded function $F:H^{-s}\to\R$, we have
 $$
 \lim_n \int_{H^{-s}} F(u) \,d\mu_n(u)= \lim_n \int_{H^{-s}} F(i_n\circ P_n(u)) \,d\mu
  =\int_{H^{-s}} F(u) \,d\mu (u)\,.
 $$
\end{proof}

\medskip
Within the framework of this subsection, we prove a KMS-Gibbs equivalence result.
Before doing this, we remark that for all $F\in \mathscr{C}^\infty_{c,cyl}(H^{-s})$ satisfying \eqref{eq.7},
$$
\Ree\langle X(u), \nabla F(u)\rangle=
\sum_{j=1}^n \partial_j^{(1)} \varphi(\pi_n(u)) \,\langle -i A u, e_j\rangle_{H,\R}+ \partial_j^{(2)} \varphi(\pi_n(u)) \,\langle -i A u,f_j\rangle_{H,\R}\,,
$$
is a well-defined continuous bounded function on $H^{-s}$. Therefore, the KMS condition \eqref{KMS-G} in
Definition \ref{KMS-def} makes  sense.

\begin{thm}
\label{thm.infdG}
 Suppose that  assumptions \eqref{assum.inf.1} and \eqref{assum.inf.2} hold. Let $X_0$ be the vector field given by $X_0(u)=-i Au$ and $\beta>0$. Then $\mu \in \mathfrak{P}(H^{-s})$ is a  $(\beta, X_0)$-KMS state if and only if $\mu$ is the Gaussian  measure $\mu_{\beta,0}$ provided by Theorem \ref{thm.gauss}.
\end{thm}
\begin{proof}
Let $\mu \in \mathfrak{P}(H^{-s})$ satisfy the KMS condition  \eqref{KMS-G}.
Consider the image measure $ \tilde\mu_n=(P_n)_\sharp\mu \in \mathfrak{P}(E_n)$.
 For $F,G\in \mathscr C^\infty_{c,cyl}(H^{-s})$ as in \eqref{eq.9} with $n=m$, one remarks that for any $u\in H^{-s}$,
 $$
 \{F,G\}(P_n (u))= \{F,G\}(u)\,,  \qquad    \Ree\langle \nabla F(P_ n(u)),X_0(P_n(u))\rangle=
 \Ree\langle \nabla F(u), X_0(u)\rangle\,.
 $$
 Hence, the KMS condition \eqref{KMS-G} reads as,
$$
\int_{H^{-s}}\, \{F,G\}( P_n u) \; d\mu=\beta \,\int_{H^{-s}} \,\Ree\langle \nabla F(P_ n(u)), X_0( P_n(u))\rangle \; G( P_n (u))\; d\mu\,,
$$
and consequently
$$
\int_{E_n}\, \{F,G\}(w) \; d\tilde\mu_n(w)=\beta \,\int_{E_n} \,\Ree\langle \nabla F(w),X_0(w)\rangle \; G(w)\; d\tilde\mu_n(w)\,.
$$
This means that the Borel probability measure $\tilde\mu_n$ on $E_n$ satisfies the KMS condition \eqref{KMS.fd}  in finite dimensions  with the continuous vector field $X_{0,n}(u)=-i Au$  and the $\mathscr C^1$ energy functional $h_{0,n}(u)=\frac{1}{2}\langle u, A u\rangle$ for $u\in E_n$. So, by Theorem \ref{thm.fd} one concludes that
$$
\tilde\mu_n= \frac{e^{-\frac{\beta}{2} \langle \cdot , A \cdot\rangle } \,dL_{2n}(\cdot)}{\int_{E_n} \, e^{-\frac{\beta}{2} \langle u , A u\rangle} \, dL_{2n}(u)} \;.
$$
Now, using Lemma \ref{lem.2} one obtains
$$
\tilde\mu_n=\mu_{\beta,0}^n=(P_n)_\sharp\mu_{\beta,0}\,.
$$
Moreover, by applying Lemma \ref{lem.3} for the Borel probability measure $\mu$, recalling \eqref{mun} and applying Lemma \ref{lem.2} again one obtains
$$
\mu_n=(i_n\circ P_n)_\sharp\mu=(i_n)_\sharp\tilde\mu_n=(i_n)_\sharp\mu_{\beta,0}^n=(i_n\circ P_n)_\sharp\mu_{\beta,0} \underset{n\to\infty}{\longrightarrow} \mu_{\beta,0}.
$$
Since by Lemma \ref{lem.3}, $\mu_n$ converges  narrowly to $\mu$, we deduce that
$$
\mu=\mu_{\beta,0}\,.
$$
Conversely, Theorem \ref{thm.nucgaus} and Remark \ref{rem.mu}-(i) show that $\mu_{\beta,0}$ is a $(\beta,X_0)$-KMS state. Note that thanks to the complex structure on $H$, one has that  $J\equiv -i$ and $X_0=-iA \equiv JA$.
\end{proof}

\subsection{Nonlinear infinite dimensional dynamical systems}
\label{subsec.nl}
In this part we address the question of equilibrium (KMS) states  for nonlinear Hamiltonian PDEs and their equivalence to Gibbs measures. We consider the same setting and notation as in Subsection \ref{sub.sec.sobolevset} above. In particular, $\mu_{\beta,0}$ is the Gaussian measure provided by Theorem \ref{thm.gauss}.

\medskip
In order to explicitly define an abstract nonlinear dynamical system that encloses the most important examples of PDEs that we wish to explore, we use the framework of Malliavin calculus and Gross-Sobolev spaces. First, we explain the main ideas behind these concepts and refer the reader to the book \cite{MR2200233} for further details. Note that the spaces used here are slightly different from the ones in the above reference.

\begin{lemma}[Malliavin derivative]
\label{grad}
For $p\in [1,\infty)$ the linear operator $\nabla$ with domain $\mathscr{D}=\mathscr C^\infty_{c,cyl}(H^{-s})$ and
\begin{eqnarray*}
\nabla:   \mathscr{D}\subset L^p(\mu_{\beta,0}) &\longrightarrow & L^p(\mu_{\beta,0}; H^{-s})\,,\\
F &\longmapsto & \nabla F
\end{eqnarray*}
where $\nabla F$ is given by \eqref{eq.5}, is closable.
\end{lemma}

\begin{proof}
Let $F_n\in \mathscr C^\infty_{c,cyl}(H^{-s})$, $n\in\N$, be a sequence such that $F_n\to 0$ in $L^p(\mu_{\beta,0})$ and
$\nabla F_n\to Y$ in $L^p(\mu_{\beta,0}; H^{-s})$. In order to prove that the operator  $\nabla$ is closable, one needs to show that $Y=0$. Indeed, using Proposition \ref{lem.ibp}
\footnote[2]{Here, we are applying Proposition \ref{lem.ibp} for   functions in $ \mathscr C^\infty_{c,cyl}(H^{-s})$ and  $ \mathscr C^\infty_{b,cyl}(H^{-s})$ for which the norm \eqref{eq.normD} is finite. At this step, we do not need to apply the full strength of Proposition \ref{lem.ibp}.}
one proves for any $G\in\mathscr C^\infty_{b,cyl}(H^{-s})$, $\varphi\in {\rm span}_\R\{e_j,f_j; j=1,\dots,k\}$ and $\varepsilon>0$,
$$
\int_{H^{-s}} \, \widetilde G(u) \langle \nabla F_n(u), \varphi\rangle \, d\mu_{\beta,0} = \int_{H^{-s}} \,
F_n(u) \bigg (-\langle \nabla \widetilde G(u),\varphi\rangle + \beta \,\widetilde G(u) \,\langle u, A\varphi \rangle\bigg)\, d\mu_{\beta,0}\,,
$$
where $\widetilde G=G \,e^{-\varepsilon \langle \cdot, A\varphi\rangle^2}\in \mathscr C^\infty_{b,cyl}(H^{-s})$ satisfies  $\widetilde G \,\langle \cdot, A\varphi\rangle\in \mathscr C^\infty_{b,cyl}(H^{-s})\subseteq L^q(\mu_{\beta,0})$ where  $q$ is the H\"older conjugate of $p$. Taking the limit  $n\to\infty$ of both sides and using the H\"older inequality, one obtains
$$
\int_{H^{-s}} \, \widetilde G(u) \,\langle Y, \varphi\rangle \, d\mu_{\beta,0} =0\,.
$$
Letting $\varepsilon \to 0$ and using the density of the space $\mathscr C^\infty_{b,cyl}(H^{-s})$ in $L^p(\mu_{\beta,0})$, one shows for any $\varphi\in {\rm span}_\R\{e_j,f_j; j=1,\dots,k\}$,
$$
\langle Y, \varphi\rangle=0\,, \qquad \mu_{\beta,0}-a.s.
$$
Hence, using the separability of $H^s$ and a density argument one proves  that $Y=0$ almost surely with respect to $\mu_{\beta,0}$.
\end{proof}

In light of Lemma \ref{grad}, one can introduce the following \emph{Gross-Sobolev} spaces.

\begin{defn}[Gross-Sobolev spaces]
\label{Gross-Sobolev_spaces}
For $p\in[1,\infty)$, we denote the closure domain of the linear operator $\nabla$ from Lemma \ref{grad} by $\mathbb D^{1,p}(\mu_{\beta,0})$. On $\mathbb D^{1,p}(\mu_{\beta,0})$, we consider the norm
\begin{equation}
\label{eq.normD}
\| F\|^p_{\mathbb D^{1,p}(\mu_{\beta,0})} := \|F\|_{L^{p}(\mu_{\beta,0})}^p+ \| \nabla F\|_{L^p(\mu_{\beta,0}; H^{-s})}^p\;.
\end{equation}
\end{defn}

By Lemma \ref{grad}, we obtain that $\mathbb D^{1,p}(\mu_{\beta,0})$ endowed with the above graph norm \eqref{eq.normD} is a Banach space. Furthermore, if $p=2$, it is a  Hilbert space with the inner product
$$
\langle F, G\rangle_{\mathbb D^{1,2}(\mu_{\beta,0})}= \langle F, G\rangle_{L^{2}(\mu_{\beta,0})}+ \langle \nabla F, \nabla G\rangle_{L^2(\mu_{\beta,0}; H^{-s})}\;.
$$

The abstract nonlinear dynamical system that we shall consider is defined  as a pair consisting of a linear operator $A$ satisfying \eqref{assum.inf.1}, \eqref{assum.inf.3}-\eqref{assum.inf.2} and a Borel nonlinear energy functional $ h^I: H^{-s}\to \R$ satisfying for some $\beta>0$ the following hypothesis:
\begin{align}
\label{hyp.i}
e^{-\beta h^I(\,\cdot)}\in L^1(\mu_{\beta,0})\, \qquad \text{ and } \qquad h^I\in \mathbb D^{1,2}(\mu_{\beta,0})\,.
\end{align}
 More specifically,  the vector field of the system is given by
 \begin{equation}
 \label{eq.vecnl}
 X(u)=-iA u-i \nabla h^I(u)\,,
 \end{equation}
 defining a field equation  in the interaction representation given  through the non-autonomous differential equation for $v(t):=e^{itA}u(t)$,
$$
\dot v(t)= e^{it A} X^I(e^{-it A}v(t))\,,
$$
where $X^I=-i\nabla h^I:H^{-s}\to H^{-s}$ is a Borel vector field belonging to $L^2(\mu_{\beta,0};H^{-s})$ and $t\in \R \to v(t)\in H^{-s}$ is a stochastic process solution.
Remark that by the Cauchy-Schwarz  inequality the assumption \eqref{hyp.i} implies that $X^I\in L^1(\mu_{\beta,0};H^{-s})$.

\medskip
According to Definition \ref{KMS-G},  a Borel probability measure $\mu$ on $H^{-s}$ is a $(\beta,X)$-KMS state of the dynamical system induced by the vector field  $X=-iA+X^I$, at inverse temperature $\beta>0$, if and only if for all $F,G\in\mathscr{C}^\infty_{c,cyl}(H^{-s})$,
\begin{equation}
\label{eq.KMSNL}
\int_{H^{-s}}\, \{F,G\}(u) \; d\mu=\beta \,\int_{H^{-s}} \,\Ree\langle  \nabla F(u),-iAu+X^I(u) \rangle \;G(u)\; d\mu\,.
\end{equation}

\begin{thm}
\label{thm.infdN}
Assume that the assumption \eqref{hyp.i} is true  and furthermore for any $\varphi\in H^s$ the function
$ \langle\nabla h^I,\varphi\rangle e^{-\beta h^I}\in L^1(\mu_{\beta,0})$. Then the Gibbs measure
\begin{equation}
\label{thm.infdN_measure}
\mu_{\beta}=\frac{e^{-\beta h^I} \,\mu_{\beta,0}}{\int_{H^{-s}} e^{-\beta h^I} d\mu_{\beta,0}}\,,
\end{equation}
 is a $(\beta,X)$-KMS state satisfying  \eqref{eq.KMSNL} where $X$ is the vector field given in \eqref{eq.vecnl}.
\end{thm}
\begin{proof}
Let  $F,G\in\mathscr{C}_{c, cyl}^\infty(H^{-s})$ be such that for some $p,q\in\N$ and $\varphi\in\mathscr{C}_{c}^\infty(\R^{2p}), \psi\in \mathscr{C}_{c}^\infty(\R^{2q})$, we have
$$
F(u)=\varphi\circ \pi_p(u)\, ,\qquad G(u)=\psi\circ \pi_q(u)\,,
$$
where $\pi_p,\pi_q$ are the mappings in \eqref{eq.pi}. Hence, according to  \eqref{pois.bra} one writes
\begin{eqnarray*}
\int_{H^{-s}}\, \{F,G\} \;  d\mu_{\beta}=
\sum_{j=1}^{p} \frac{1}{z_{\beta}}
\int_{H^{-s}}\, \big(\partial_{e_j}F \;\partial_{f_j}G- \partial_{e_j}G \; \partial_{f_j}F\big) \, e^{-\beta h^I}\; d\mu_{\beta,0}\,,
\end{eqnarray*}
where $\partial_{e_j}, \partial_{f_j}$ are directional derivatives
and
\begin{equation}
\label{z_beta}
z_{\beta}=\int_{H^{-s}} e^{-\beta h^I}\,d \mu_{\beta,0}
\end{equation}
is the normalization constant in \eqref{thm.infdN_measure}. We note that there exists a sequence of functions $\theta_k\in\mathscr C_b^1(\R)$ such that $\theta_k(x)\to e^{-\beta x}$ and $\theta'_k(x)\to -\beta e^{-\beta x}$ pointwise for all $x\in\R$ with
\begin{equation}
\label{eq.th.est}
0\leq\theta_k (x)\leq e^{-\beta x}\, \qquad \text{ and } \qquad  |\theta'_k(x)|\leq  c e^{-\beta x}\,,
\end{equation}
for some constant $c>0$ and for  all $k$ large enough. Indeed, we can take $\theta_k(x)=e^{-\beta x}$ if $x\geq -k$ and $\theta_k(x)=\arctan (-\beta e^{\beta k} (x+k))+e^{\beta k}$ if $x<-k$.
By applying the dominated convergence theorem and Proposition \ref{lem.ibp} with $\varphi=f_j$ and $\varphi=e_j$ respectively, we have
\begin{eqnarray*}
\int_{H^{-s}}\, \{F,G\} \;  d\mu_{\beta}&=& \lim_k
\sum_{j=1}^{p} \frac{1}{z_{\beta}}\,
\int_{H^{-s}}\, \big(\partial_{e_j}F \;\partial_{f_j}G- \partial_{e_j}G \; \partial_{f_j}F\big) \, \theta_k( h^I)\; d\mu_{\beta,0}\\
&=& \lim_k
\sum_{j=1}^{p}\frac{1}{z_{\beta}}
\int_{H^{-s}}\, G \;\bigg(-\partial_{f_j}\big(\partial_{e_j}F \,\theta_k(h^I)\big)+\beta \langle u,A f_j\rangle
\langle \nabla F,e_j\rangle \,\theta_k(h^I)\bigg)\\
&&
 \hspace{.7in} +G \;\bigg(\partial_{e_j}\big(\partial_{f_j}F \,\theta_k(h^I)\big)-\beta \langle u,A e_j\rangle
\langle \nabla F,f_j\rangle \,\theta_k(h^I)\bigg)\; d\mu_{\beta,0}\,.
\end{eqnarray*}
Note that we also used Lemma \ref{lem.chi} in order to deduce that $\theta_k(h^I)$, $\partial_{e_j}F \,\theta_k(h^I)$ and $\partial_{f_j}F \,\theta_k(h^I)$ belong to $\mathbb D^{1,2}(\mu_{\beta,0})$.
Moreover, one observes that
\begin{eqnarray*}
\sum_{j=1}^p \partial_{e_j}F \,\partial_{f_j}h^I- \partial_{f_j}F  \,\partial_{e_j}h^I &=& \Ree\langle \nabla F, -i\nabla h^I\rangle\,,\\
\sum_{j=1}^p \langle u,A f_j\rangle
\langle \nabla F,e_j\rangle -  \langle u,A e_j\rangle
\langle \nabla F,f_j\rangle &=& \Ree\langle \nabla F, -iAu\rangle \;,
\end{eqnarray*}
Thus, using the assumptions of the Theorem and dominated convergence, one obtains
\begin{eqnarray*}
\int_{H^{-s}}\, \{F,G\} \;  d\mu_{\beta}&=& \lim_k \frac{1}{z_{\beta}}
\int_{H^{-s}}\, G \;\bigg(\Ree\langle \nabla F, i\nabla h^I\rangle \;\theta'_k(h^I) +\beta\Ree\langle \nabla F, -iAu\rangle \;\theta_k(h^I)\bigg)  \; d\mu_{\beta,0}\\
&=& \beta \,\frac{1}{z_{\beta}}
\int_{H^{-s}}\, G \;\bigg(\Ree\langle \nabla F, -i\nabla h^I\rangle +\Ree\langle \nabla F, -iAu\rangle\bigg) e^{-\beta h^I} \; d\mu_{\beta,0}\,.
\end{eqnarray*}

So, this proves the KMS condition \eqref{eq.KMSNL} for the Gibbs measure $\mu_{\beta}$.

\end{proof}

Our next main result shows that the dynamical system at hand admits a  unique KMS state which is the Gibbs measure $\mu_{\beta}$. But before stating such a result, we need to prove some  preliminary results.

\begin{proposition}
\label{prop.grad}
Assume \eqref{hyp.i} is true.   Let $\mu$ be a Borel probability measure  on $H^{-s}$ which is absolutely continuous with respect to $\mu_{\beta,0}$ , i.e. there exists a non-negative density $\varrho\in L^1(\mu_{\beta,0})$ such that for all Borel sets, we have
  $$
  \mu(B)=\int_{B} \varrho(u) \,d\mu_{\beta,0}\,.
  $$
  Assume further that the density $\varrho\in \mathbb D^{1,2}(\mu_{\beta,0})$. If $\mu$ is a $(\beta,X)$-KMS state satisfying   \eqref{eq.KMSNL}  then the density $\varrho$ satisfies the equation,
  \begin{equation}
\label{eq.21}
\fbox{ $\nabla\varrho+\beta \varrho \,\nabla h^I =0\,$ }\,,
\end{equation}
in $L^1(\mu_{\beta,0};H^{-s})$ with $\nabla$ is the Malliavin derivative from Lemma \ref{grad}.
\end{proposition}

\begin{proof}
Consider $\mu \in \mathfrak{P}(H^{-s})$ satisfying the KMS condition  \eqref{eq.KMSNL}  and the above hypothesis.  There exists a sequence $\varrho_n\in \mathscr{C}^\infty_{b,cyl}(H^{-s})$ such that $\varrho_n\to \varrho$,
 $\partial_{e_j}\varrho_n\to \partial_{e_j}\varrho$ and  $\partial_{f_j}\varrho_n\to \partial_{f_j}\varrho$
 in $L^2(\mu_{\beta,0})$ for all $j\in\N$.
 Then using the Leibniz  rule, one proves
 \begin{eqnarray*}
 \lim_n\int_{H^{-s}} \{F, G\varrho_n\} \, d\mu_{\beta,0}&=&\lim_n \int_{H^{-s}} \{F, G\}\,\varrho_n \, d\mu_{\beta,0}
 +\int_{H^{-s}}  \{F, \varrho_n \}\, G \, d\mu_{\beta,0}\\
 &=& \int_{H^{-s}} \{F, G\}\, d\mu + \int_{H^{-s}}  \Ree\langle \nabla F, -i \nabla\varrho\rangle \, G \, d\mu_{\beta,0}\,.
 \end{eqnarray*}
On the other hand, the KMS condition satisfied by the measure $\mu_{\beta,0}$ yields
\begin{eqnarray*}
 \lim_n\int_{H^{-s}} \{F, G\varrho_n\} \, d\mu_{\beta,0}&=&\beta \lim_n\int_{H^{-s}} \Ree \langle\nabla F, X_0\rangle \, G\varrho_n \, d\mu_{\beta,0}\\
 &=& \beta \int_{H^{-s}} \Ree \langle\nabla F, X_0\rangle \, G \, d\mu\,,
\end{eqnarray*}
where $X_0(u)=-i Au$. Hence, the two above equalities give
$$
\beta \int_{H^{-s}} \Ree \langle\nabla F, X_0\rangle \, G \, d\mu= \int_{H^{-s}} \{F, G\}\, d\mu + \int_{H^{-s}}  \Ree\langle \nabla F, -i\nabla\varrho\rangle \, G \, d\mu_{\beta,0}\,.
$$
Since $\mu$ satisfies the KMS condition with the vector field $X=X_0+X^I$, then one concludes
for all $G\in\mathscr{C}^\infty_{c,cyl}(H^{-s})$,
$$
 \int_{H^{-s}} \Ree\langle \nabla F, \beta \varrho \, X^I -i \nabla\varrho\rangle \,  G\, d\mu_{\beta,0} \,=0.
$$
We recall \eqref{eq.vecnl} and apply a standard density argument to deduce that
\begin{equation*}
 \nabla\varrho+\beta \varrho \,\nabla h^I =0\,,
\end{equation*}
$\mu_{\beta,0}$-almost surely and  as an element of $L^1(\mu_{\beta,0}; H^{-s})$.
\end{proof}

\begin{lemma}
\label{lem.log}
Assume \eqref{hyp.i} is true and $e^{-\beta h^I}\in L^2(\mu_{\beta,0})$.  Let $\varrho$ be the density in Proposition \ref{prop.grad} and take any convex combination,
\begin{equation}
\label{eq.26}
\tilde\varrho=\alpha \varrho+ \frac{(1-\alpha)}{\|e^{-\beta h^I}\|_{L^1(\mu_{\beta,0})}} \,e^{-\beta h^I}\,,
\end{equation}
with $\alpha\in (0,1)$. Then
$$
\log(\tilde \varrho)\in \mathbb{D}^{1,2}(\mu_{\beta,0})\qquad \text{ and } \qquad \nabla \log (\tilde \varrho)=\frac{\nabla\tilde\varrho}{\tilde \varrho}.
$$
\end{lemma}
\begin{proof}
One checks that $e^{-\beta h^I}\in \mathbb D^{1,1}(\mu_{\beta,0})$ and $\nabla e^{-\beta h^I}+\beta e^{-\beta h^I} \nabla h^I=0$. Indeed, take $\theta_k$ the same function as in the proof of Theorem \ref{thm.infdN}. Then by dominated convergence one has
\begin{eqnarray*}
\lim_k \|\theta_k(h^I)-e^{-\beta h^I}\|_{L^1(\mu_{\beta,0})}=0\,, \qquad \lim_{k,k'} \|\nabla\theta_k(h^I)-\nabla \theta_{k'}(h^I)\|_{L^1(\mu_{\beta,0};H^{-s})}=0\,.
\end{eqnarray*}
Hence, the sequence $\theta_k(h^I)$ converges to $e^{-\beta h^I}$ in $\mathbb{D}^{1,1}(\mu_{\beta,0})$.
Thus, one concludes as a consequence of Proposition \ref{prop.grad} that
$$
\nabla \tilde\varrho+\beta \tilde\varrho \,\nabla h^I=0\,.
$$
Since $\tilde \varrho>0$ then
$$
\frac{\nabla\tilde\varrho}{\tilde\varrho}=-\beta \nabla h^I\in L^2(\mu_{\beta,0};H^{-s})\,.
$$
There exists a sequence of functions $\kappa_k\in\mathscr C_b^1(\R)$ such that $\kappa_n(x)\to \log(x)$ and $\kappa'_n(x)\to 1/x$ pointwise for $x>0$  and furthermore,
$$
|\kappa_n(x)|\leq |\log(x)|,\qquad|\kappa_n'(x)|\leq \frac{c}{|x|},
$$
for some constant $c>0$ and all $x>0$. Indeed, we can take $\kappa_n(x)=\log(x)$ for $ 1/n\leq x\leq n$, $\kappa_n(x)=-n/x+1+\log(n)$ for $x>n$  and  $\kappa_n(x)=n\arctan(x-1/n)-\log(n)$ for $x <1/n$.  By Lemma \ref{lem.chi} one knows that $\kappa_n(\tilde\varrho) \in \mathbb D^{1,2}(\mu_{\beta,0})$ and $\nabla \kappa_n(\tilde\varrho)=\kappa_n'(\tilde \varrho) \nabla \tilde \varrho$. Hence,  dominated convergence yields
\begin{eqnarray*}
\lim_n\int_{H^{-s}}\bigg| \kappa_n(\tilde \varrho)- \log(\tilde \varrho)\bigg|^2 \; d\mu_{\beta,0}=
\lim_n\int_{H^{-s}}\bigg\| \kappa_n'(\tilde \varrho) \nabla \tilde \varrho-\frac{\nabla \tilde \varrho}{\tilde \varrho}\bigg\|^2_{H^{-s}} \; d\mu_{\beta,0} =0\,.
\end{eqnarray*}
Therefore, one concludes that $\log(\tilde \varrho)\in \mathbb D^{1,2}(\mu_{\beta,0})$ and $ \nabla
\log(\tilde \varrho)= \frac{\nabla \tilde \varrho}{\tilde \varrho}$.
\end{proof}

\begin{thm}
\label{thm.KMSGibbs}
Assume that \eqref{hyp.i} is true and $e^{-\beta h^I}\in L^2(\mu_{\beta,0})$.   Let $\mu$ be a Borel probability measure  on $H^{-s}$ which is absolutely continuous with respect to $\mu_{\beta,0}$, i.e. there exists a non-negative density $\varrho\in L^1(\mu_{\beta,0})$ such that for all Borel sets,
  $$
  \mu(B)=\int_{B} \varrho(u) \,d\mu_{\beta,0}\,.
  $$
  Assume further that $\varrho \equiv \frac{d \mu}{d \mu_{\beta,0}} \in \mathbb D^{1,2}(\mu_{\beta,0})$. Then  $\mu$ is a $(\beta,X)$-KMS state   for the vector field $X$ in \eqref{eq.vecnl} if and only if $\mu$ is equal to the Gibbs measure
$$
\mu_{\beta}=\frac{e^{-\beta h^I} \,\mu_{\beta,0}}{\int_{H^{-s}} e^{-\beta h^I(u)} d\mu_{\beta,0}}\,.
$$
\end{thm}
\begin{proof}
Sufficiency follows from Theorem \ref{thm.infdN}. Take $\tilde\varrho$ the convex combination density in Lemma \ref{lem.log} and remark that $\tilde \varrho>0$. By Lemma \ref{lem.log} and
the assumption \eqref{hyp.i}, one knows that the function
$$
F=\log(\tilde\varrho)+\beta h^I\in \mathbb D^{1,2}(\mu_{\beta,0})\,.
$$
Moreover, one has
$$
\nabla \big( \log(\tilde\varrho)+\beta h^I\big)=\frac{\nabla\tilde\varrho}{\tilde\varrho}+\beta \nabla h^I= 0\,.
$$
Hence, using Proposition \ref{lem.cst} one concludes that for some constant $c\in\R$
$$
\log(\tilde\varrho)+\beta h^I=c\,,
$$
$\mu_{\beta,0}$-almost surely. Finally, using the normalization of the density $\varrho$ one shows that
$$
\varrho=\frac{e^{-\beta h^I}}{\|e^{-\beta h^I}\|_{L^1(\mu_{\beta,0})}}\,.
$$
\end{proof}

\begin{remark}[Relative entropy]
In statistical mechanics it is common to characterize the Gibbs measure $\mu_\beta$ by means of relative entropy functional. So, it is not surprising to find a link between our analysis based on KMS states and
the concept of entropy. In particular, note that the functional $F(\varrho)=\log(\varrho)+\beta h^I$ used in the proof of Theorem \ref{thm.KMSGibbs} is similar to the integrand that one  found  in the formula of the \emph{relative entropy},
$$
E_{\mu_{\beta,0}}(\mu)= \int_{H^{-s}} \, \varrho \log(\varrho) \,d\mu_{\beta,0}+\beta\int_{H^{-s}} h^I d\mu=\int_{H^{-s}} \,  \big(\log(\varrho)+\beta h^I\big)\, \varrho \,d\mu_{\beta,0}\,,
$$
where $\varrho$ is the density satisfying  $\mu=\varrho \mu_{\beta,0}$. We note that
$$
E_{\mu_{\beta,0}}(\mu)=E_{\mu_{\beta}}(\mu)-\log(z_\beta)\,,
$$
where $z_\beta$ is given by \eqref{z_beta} and one knows that  $E_{\mu_{\beta}}(\mu)$ is non-negative with $E_{\mu_{\beta}}(\mu)=0$ if and only if $\mu=\mu_\beta$. In particular, this means that $\mu_\beta$ is the unique minimizer of the relative entropy.
\end{remark}

\section{Nonlinear PDEs}
\label{sec.Nonlinear}
In this section we apply the concept of KMS states to various  examples of nonlinear PDEs, namely
 to the nonlinear Schr\"odinger, Hartree, and wave (Klein-Gordon) equation.  The construction of invariant Gibbs measures for such equations is  well understood. In particular, the analysis is based on  probabilistic tools,  truncation to a finite number of Fourier modes, and nonlinear stability estimates (see e.g.~\cite{MR1309539,MR3869074,MR785258,MR939505,MR3952697,MR2498359} and the references therein). Here we emphasize that the  nonlinearities appearing in the above equations belong to the Gross-Sobolev spaces. Thus, it is possible to appeal to the Malliavin calculus and to apply our results.

\subsection{Nonlinear Schr\"odinger equations}
\label{sub.sec.NLS}
Gibbs measures for NLS equations are well-studied, due to the fact that they are useful tools for establishing existence of global solutions and well posedness for rough datum, see e.g. \cite{MR1374420, MR1470880, MR1777342,MR3869074,MR3844655,MR4133695}  and the references therein.

\medskip
Consider the Hilbert space $H=L^2(\T^d)$ where $\T^d=\R^d/(2\pi\mathbb Z^d)$ is the flat $d$-dimensional torus and define the Sobolev weighted spaces $H^{-s}$, as in Subsection \ref{sub.sec.sobolevset}, by means of the positive self-adjoint operator
\begin{equation}
\label{eq.38}
 A=-\Delta+\mathds 1\,,
\end{equation}
where $\Delta$ is the Laplacian  on $\T^d$. So, the family $\{e_k=e^{ik x}\}_{k\in\mathbb Z^d}$ forms an O.N.B of eigenvectors for the operator $A$ which admits a compact resolvent.
Throughout this section, we consider
\begin{equation}
\label{s_choice}
s>\frac{d}{2}-1\,.
\end{equation}
Note that  \eqref{assum.inf.3}-\eqref{assum.inf.2} are satisfied for $s$ as in \eqref{s_choice}. In particular, in the one dimensional  case we can take $s=0$. Therefore, according to Subsection \ref{sub.sec.sobolevset}, the Gaussian measure $\mu_{\beta,0}$ given by Theorem \ref{thm.gauss} is a well defined Borel probability measure on $H^{-s}$ and it is the unique $(\beta,X_0)$-KMS state for the vector field $X_0=-i A$ and for any inverse temperature $\beta>0$. We now analyze the KMS-condition in the context of various nonlinear Schr\"{o}dinger-type equations, which we describe in detail below.

In the sequel, we write $\langle x \rangle=\sqrt{1+|x|^2}$ for the Japanese bracket. Furthermore, we write $A \lesssim B$ if there exists $C>0$ such that $A \leq CB$. If $C$ depends on the parameters $a_1, \ldots, a_k$, we write $A \lesssim_{a_1,\ldots,a_k} B$. We write $A \gtrsim B$ if $B \lesssim A$. Finally, if $A \lesssim B$ and $B\lesssim A$, we write $A \sim B$.

\bigskip
\paragraph{\textbf{1. The Hartree equation on $\T^1$}}
When $d=1$, we consider $V:\T^1\equiv\T\to\R$ a pointwise nonnegative even $L^1$ function. The \emph{Hartree nonlinear functional} is given as
\begin{equation}
\label{eq.nl.3}
h^I(u)=\frac{1}{4}\int_{\T}\int_{\T} |u(x)|^2 \,V(x-y)\, |u(y)|^2 \,dx\,dy \geq 0\,.
\end{equation}

\paragraph{\textbf{2. The Hartree equation on $\T^d$, $d=2,3$}}
When $d=2,3$, we need to renormalise the interaction by means of \emph{Wick-ordering} (see e.g.\ \cite{MR1470880,MR3719544,Sohinger_2019}). We summarise the construction here.
Given $n \in \mathbb{N}$, we recall the projection map in \eqref{eq.proj} that we take in our case to be
\begin{equation*}
P_n=\sum_{|k| \leq n} |e_k\rangle\langle e_k|\,,
\end{equation*}
and define for $x \in \mathbb{T}^d$
\begin{equation}
\label{sigma_n}
\sigma_{n,\beta} := \int_{H^{-s}} |P_n u(x)|^2\, d \mu_{\beta,0}= \sum_{|k| \leq n} \frac{1}{\beta (|k|^2+1)} \sim
\begin{cases}
\frac{\log n}{\beta}\,, \quad &\mbox{if } d=2\\
\frac{n}{\beta}\,, \quad &\mbox{if } d=3\,.
\end{cases}
\end{equation}
Note that $\sigma_{n,\beta}$ is independent of $x$.
Let us henceforth use the shorthand
\begin{equation}
\label{u_n_definition}
u_n:=P_n u
\end{equation}
and consider the Wick ordering with respect to $\mu_{\beta,0}$
\begin{equation}
\label{Wick_ordering_n}
:|u_n|^2: \,=|u_n|^2-\sigma_{n,\beta}\,.
\end{equation}
We observe that the above construction depends on $\beta$, but we suppress this in the notation.
We let
\begin{equation}
\label{eq.nl.3_Wick_truncated}
h^I_{n,\beta}(u):=\frac{1}{4}\int_{\T^d}\int_{\T^d} :|u_n(x)|^2: \,V(x-y)\, :|u_n(y)|^2: \,dx\,dy\,.
\end{equation}
Here and in the sequel, we write $:|u_n(x)|^2:$ instead of $:|u_n|^2: (x)$ for \eqref{Wick_ordering_n} evaluated at $x$.
We work with even $V \in L^1(\mathbb{T}^d)$ such that there exist $\epsilon>0$ and $C>0$ with the property that for all $k \in \mathbb{Z}^d$ the following estimates hold.
\begin{equation}
\label{V_hat_estimates}
\begin{cases}
0 \leq \hat{V}(k) \leq \frac{C}{\langle k \rangle^{\epsilon}} &\mbox{if } d=2
\\
0 \leq \hat{V}(k) \leq \frac{C}{\langle k \rangle^{2+\epsilon}} &\mbox{if } d=3\,.
\end{cases}
\end{equation}
In particular, $V$ is assumed to be of positive type (i.e.\ $\hat{V}$ is pointwise nonnegative).
Under the assumptions \eqref{V_hat_estimates}, the arguments in \cite{MR1470880} show that \eqref{eq.nl.3_Wick_truncated} converges in $L^p(\mu_{\beta,0})$, for all $p\geq 1$, to
\begin{equation}
\label{eq.nl.3_Wick}
h^I(u) =\lim_n h^I_{n,\beta}(u)\,\equiv \frac{1}{4}\int_{\T^d}\int_{\T^d} :|u(x)|^2: \,V(x-y)\, :|u(y)|^2: \,dx\,dy.
\end{equation}
Let us note that, in the recent work \cite{Deng_Nahmod_Yue_2021}, the authors extend the result of \cite{MR1470880} for $d=3$ to potentials satisfying $0 \leq \hat{V}(k) \leq \frac{C}{\langle k \rangle^{1-\epsilon}}$. We do not consider this extension in our current paper.

We recall the details of the proof of \eqref{eq.nl.3_Wick} in Appendix \ref{Appendix_B}. We refer to \eqref{eq.nl.3_Wick} as the \emph{Wick-ordered Hartree nonlinear functional}. Since $V$ is of positive type, we have that
\begin{equation}
\label{eq.nl.3_Wick_positive}
h^I(u) \in [0,\infty) \qquad \mu_{\beta,0} \; \mbox{- almost surely}.
\end{equation}
Note that \eqref{eq.nl.3} and \eqref{eq.nl.3_Wick_positive} imply that $e^{-\beta h^I}\in L^2(\mu_{\beta,0})$ in dimension $d=1,2,3$.

\paragraph{\textbf{3. The NLS equation on $\T$}}
In the one dimensional case, the assumption \eqref{assum.inf.2} is satisfied for $s=0$ and the nonlinear functional is given by
\begin{equation}
\label{eq.nl.4}
h^I(u)=\frac{1}{q}\;\int_\T |u(x)|^q\, dx \geq 0\,.
\end{equation}
for  $q=2r$ with $r \in \N, r \geq 2$.

\paragraph{\textbf{4. The NLS equation on $\T^2$}}
On $\T^2$, we consider the general Wick-ordered nonlinearity. Given $r \in \N$, and recalling \eqref{sigma_n}
we define

\begin{equation}
\label{eq.39}
:|u_n|^{2r}: \,=(-1)^{r}r!\;\sigma_{n,\beta}^r\;L_r\bigg(\frac{|u_n|^2}{\sigma_{n,\beta}}\bigg)\,,
\end{equation}
where $L_r$ is the $r$-th Laguerre polynomial.
Note that this is a generalization of \eqref{Wick_ordering_n} since $L_1(x)=-x+1$. For a given $s>0$, one can consider the nonlinear Borel functional
$h^I:H^{-s}\to\R$ defined as the following limit in $L^2(\mu_{\beta,0})$,
\begin{equation}
\label{eq.nl.1}
h^I(u)=\lim_n h^I_n(u)= \lim_n\frac{1}{2r} \int_{\T^2}  :|u_n|^{2r}: \,dx\equiv \frac{1}{2r} \int_{\T^2} :|u|^{2r}: \,dx\,.
\end{equation}
We refer the reader to \cite{MR3844655} for a self-contained proof of \eqref{eq.nl.1}.

\bigskip
For the nonlinear functionals introduced above, the following statement holds true.
\begin{proposition}
\label{lem.9}
The nonlinear Borel functionals $h^I(u)$ given by \eqref{eq.nl.3}, \eqref{eq.nl.3_Wick}, \eqref{eq.nl.4} and \eqref{eq.nl.1} belong to the Gross-Sobolev spaces $\mathbb D^{1,p}(\mu_{\beta,0})$ for all $1\leq p<\infty$.
\end{proposition}
\begin{proof}
We prove each case separately.

\noindent
\paragraph{\textbf{(i) The Hartree equation on $\T$}}
It is well-known that for $V \in L^1(\T)$, we have that
\begin{equation}
\label{h_bound}
h^I(u) \in L^p(\mu_{\beta,0}).
\end{equation}
We note that it suffices to prove \eqref{h_bound} when $p \geq 2$ as the claim for $p \in [1,2)$ then follows from H\"{o}lder's inequality.
More precisely, by using the Cauchy-Schwarz inequality, Young's inequality, and the Sobolev embedding $H^{\zeta}(\T) \hookrightarrow L^4(\T)$ for $\zeta \in (\frac{1}{4},\frac{1}{2})$ in \eqref{eq.nl.3}, we get that
\begin{equation*}
0 \leq h^I(u) \lesssim_\zeta \|V\|_{L^1}\,\|u\|_{H^{\zeta}}^4\,,
\end{equation*}
and we deduce \eqref{h_bound} by arguing similarly as for \eqref{eq.trac} above.

A direct calculation shows that
\begin{equation}
\label{gradient_1}
\nabla h^I(u)=(V*|u|^2)\,u\,.
\end{equation}
For $s \in (-\frac{1}{2},0)$, we let $\zeta:=-s \in (0,\frac{1}{2})$.
By using \eqref{gradient_1} and by applying Lemma \ref{product_lemma} twice for sufficiently small $\alpha$, we deduce that for some $\zeta' \in (\zeta,\frac{1}{2})$
\begin{multline}
\label{gradient_1_bound_1}
\|\nabla h^I(u)\|_{L^p(\mu_{\beta,0}; H^{-s})}=\big\|\|(V*|u|^2)\,u\|_{H^{\zeta}}\big\|_{L^p(\mu_{\beta,0})} \lesssim_{\zeta,V,p} \big\|\|u\|_{H^{\zeta'}}^3\big\|_{L^p(\mu_{\beta,0})}
\\
=\bigg(\int \|u\|_{H^{\zeta'}}^{3p} \, d \mu_{\beta,0}\bigg)^{1/p}<\infty\,.
\end{multline}
In \eqref{gradient_1_bound_1}, we used the observation that $\|V*f\|_{H^{\theta}} \leq \|\hat{V}\|_{\ell^{\infty}}\,\|f\|_{H^{\theta}} \leq \|V\|_{L^1}\,\|f\|_{H^{\theta}}$.

\noindent
\paragraph{\textbf{(ii) The Hartree equation on $\T^d$, $d=2,3$}}

We note that \eqref{eq.nl.3_Wick} implies that $h^I(u) \in L^p(\mu_{\beta,0})$. As was noted earlier, \eqref{eq.nl.3_Wick} can be deduced from the arguments of \cite{MR1470880} under the assumptions given by \eqref{V_hat_estimates}.
When $d=3$, a detailed proof of this fact is given in \cite[Lemma 1.4 (i)]{Sohinger_2019}. Note that, here, one assumes that $V \in L^q(\T^3)$ for $q>3$, which follows from \eqref{V_hat_estimates} (see \cite[(29)]{MR1470880} and \cite[(1.44)-(1.45)]{Sohinger_2019}). When $d=2$, this fact is shown in detail in \cite[Lemma 1.4 (ii)]{Sohinger_2019} if, in addition to satisfying, one assumes that $V$ is pointwise nonnegative. In Appendix \ref{Appendix_B}, we present the proof of \eqref{eq.nl.3_Wick} from \cite{MR1470880} which does not require pointwise nonnegativity of $V$.

A direct calculation shows that
\begin{equation*}
\nabla h_n^I(u)=P_n\bigg[\bigg(\int V(\cdot-y) \, :|u_n(y)|^2:\,d y\bigg)\, u_n \bigg] =P_n \big[(V \,*\, :|u_n|^2:\,)\, u_n \big]\,.
\end{equation*}
As was noted earlier, it suffices to consider $p \geq 2$.
For $s$ as in \eqref{s_choice}, we want to show that $(h_n^I(u))$ is a Cauchy sequence in $L^p(\mu_{\beta,0};H^{-s})$.
%We first show that this sequence is bounded in $L^p(\mu_{\beta,0};H^{-s})$, and we then explain how to deduce the Cauchy property from the proof.

By Minkowski's inequality, we have
\begin{align*}
\notag
\|\nabla h_n^I(u)\|_{L^p(\mu_{\beta,0};H^{-s})}&\sim \Bigg\|\Bigg(\sum_{k\in \Z^d, |k|\leq n} \langle k \rangle^{-2s}\,\Big| \big[(V \,*\, :|u_n|^2:\,)\, u_n \big]\,\widehat{\,}\,(k)\Big|^{2}\Bigg)^{1/2}\Bigg\|_{L^p(\mu_{\beta,0})}
\\
&\leq
\Bigg(\sum_{k\in \Z^d, |k|\leq n} \langle k \rangle^{-2s}\,\big\| \big[(V \,*\, :|u_n|^2:\,)\, u_n \big]\,\widehat{\,}\,(k)\big\|_{L^p(\mu_{\beta,0})}^{2}\Bigg)^{1/2}\,.
\end{align*}
Likewise, for $n \geq m$, we have
\begin{align}
\notag
&\|\nabla h_n^I(u)-\nabla h_m^I(u)\|_{L^p(\mu_{\beta,0};H^{-s})}
\\
\notag
&\lesssim
\Bigg(\sum_{k\in \Z^d, |k|\leq m} \langle k \rangle^{-2s}\,\big\| \big[(V \,*\, :|u_n|^2:\,)\, u_n \big]\,\widehat{\,}\,(k)-\big[(V \,*\, :|u_m|^2:\,)\, u_m \big]\,\widehat{\,}\,(k)\big\|_{L^p(\mu_{\beta,0})}^{2}\Bigg)^{1/2}
\\
\label{gradient_Cauchy_bound_1}
&+\Bigg(\sum_{k\in \Z^d, m<|k|\leq n} \langle k \rangle^{-2s}\,\big\| \big[(V \,*\, :|u_n|^2:\,)\, u_n \big]\,\widehat{\,}\,(k)\big\|_{L^p(\mu_{\beta,0})}^{2}\Bigg)^{1/2}\,.
\end{align}
In what follows, we view the Gaussian measure $\mu_{\beta,0}$ as the probability measure induced by the map
\begin{equation}
\label{phi_omega}
\omega \in \Omega \mapsto \phi(x) \equiv \phi_{\beta}^{\omega}(x)=\frac{1}{\sqrt{\beta}}\,\sum_{k \in \Z^d} \frac{g_k(\omega)}{\langle k \rangle}\,e^{ik \cdot x}\,,
\end{equation}
where $(g_k)_{k \in \mathbb{Z}^d}$ is a sequence of independent standard complex Gaussian random variables (centred with variance equal to $1$) on a probability space $(\Omega, \Sigma, \mathbb P)$.
Recalling \eqref{gradient_Cauchy_bound_1} and \eqref{phi_omega}, we consider for fixed $k \in \Z^d$ the expression
\begin{equation}
\label{gradient_bound_2}
\|[(V \,*\, :|u_n|^2:\,)\, u_n]\,\widehat{\,}\,(k)\|_{L^p(\mu_{\beta,0})}=\|[(V \,*\, :|\phi_n|^2:\,)\, \phi_n]\,\widehat{\,}\,(k)\|_{L^p(\Omega)}\,,
\end{equation}
where $\phi_n \equiv \phi^{\omega}_{n,\beta} := P_n \phi_{\beta}^{\omega}$.

We recall the following estimate from \cite[Theorem I.22]{MR0489552}.
\begin{lemma}
\label{Wiener_chaos}
Let the random variable $\psi$ be a polynomial in $(g_j)_{j \in \Z^d}$ of degree $m \in \N$. Then, for all $p \geq 2$, we have that
\begin{equation*}
\|\psi\|_{L^p(\Omega)} \leq (p-1)^{m/2}\,\|\psi\|_{L^2(\Omega)}\,.
\end{equation*}
\end{lemma}
We note that $(V *\, :|\phi_n|^2:\,)\, \phi_n]\,\widehat{\,}\,(k)$ is a polynomial in $(g_j)_{j \in \Z^d}$ of degree three. Therefore, by Lemma \ref{Wiener_chaos},  we have that
\begin{equation}
\label{gradient_bound_3}
\eqref{gradient_bound_2} \leq (p-1)^{\frac{3}{2}}\,\|[(V \,*\, :|\phi_n|^2:\,)\, \phi_n]\,\widehat{\,}\,(k)\|_{L^2(\Omega)}\,.
\end{equation}
Recalling \eqref{Wick_ordering_n} and \eqref{phi_omega}, we have
\begin{equation}
\label{gradient_bound_4}
(\,:|\phi_n|^2:\,)\,\widehat{\,}\,(k)=
\Bigg(\frac{1}{\beta}\,\sum_{|\ell|\leq n} \frac{|g_{\ell}(\omega)|^2-1}{\langle \ell \rangle^2}\Bigg) \, 1_{k=0}+
\Bigg(\frac{1}{\beta}\,\sum_{|\ell_1|\leq n, |\ell_2|\leq n, \ell_1-\ell_2=k}\frac{g_{\ell_1}(\omega)\,\overline{g_{\ell_2}(\omega)}}{\langle \ell_1 \rangle\,\langle \ell_2 \rangle}\Bigg) \, 1_{k \neq 0}\,.
\end{equation}
Therefore, by \eqref{gradient_bound_4}, we get that for $k \in \Z^d$ with $|k| \leq n$, we have
\begin{multline}
\label{gradient_bound_5}
[(V \,*\, :|\phi_n|^2:\,)\, \phi_n]\,\widehat{\,}\,(k)= \frac{\widehat{V}(0)}{\beta^{3/2}}\,\Bigg(\sum_{|\ell|\leq n} \frac{|g_{\ell}(\omega)|^2-1}{\langle \ell \rangle^2}\Bigg)\,\frac{g_k(\omega)}{\langle k \rangle}
\\
+\frac{1}{\beta^{3/2}}\,\mathop{\sum_{|\ell_1|\leq n, |\ell_2|\leq n, |\ell_3|\leq n}}_{\ell_1-\ell_2+\ell_3=k, \ell_1 \neq \ell_2} \widehat{V}(\ell_1-\ell_2)\,\frac{g_{\ell_1}(\omega)\,\overline{g_{\ell_2}(\omega)}\,g_{\ell_3}(\omega)}{\langle \ell_1 \rangle\,\langle \ell_2 \rangle\,\langle \ell_3 \rangle}=:I_n(k) + II_n(k)\,.
\end{multline}
We now analyse each of the terms $I_n(k)$ and $II_n(k)$ separately.

\paragraph{\emph{Analysis of $I_n(k)$}}
By using H\"{o}lder's inequality, and Lemma \ref{Wiener_chaos}, we have that
\begin{equation*}
\|I_n(k)\|_{L^2(\Omega)} \lesssim_{\beta} \Bigg\|\sum_{|\ell|\leq n} \frac{|g_{\ell}(\omega)|^2-1}{\langle \ell \rangle^2}\Bigg\|_{L^4(\Omega)}\,\Bigg\|\frac{g_k(\omega)}{\langle k \rangle}\Bigg\|_{L^4(\Omega)} \lesssim
\Bigg\|\sum_{|\ell|\leq n} \frac{|g_{\ell}(\omega)|^2-1}{\langle \ell \rangle^2}\Bigg\|_{L^2(\Omega)}\,\bigg\|\frac{g_k(\omega)}{\langle k \rangle}\bigg\|_{L^4(\Omega)}\,,
\end{equation*}
which by using the fact that $|g_n|^2-1$ are independent and of mean zero is
\begin{equation}
\label{I_bound}
\lesssim \Bigg(\sum_{|\ell|\leq n} \frac{1}{\langle \ell \rangle^4}\Bigg)^{1/2}\,\frac{1}{\langle k \rangle} \lesssim \frac{1}{\langle k \rangle}\,.
\end{equation}
By analogous arguments as for \eqref{I_bound}, we deduce that for $n \geq m$, we have
\begin{equation}
\label{I_Cauchy}
\|I_n(k)-I_m(k)\|_{L^2(\Omega)} \lesssim_{\beta} \Bigg(\sum_{m<|\ell|\leq n} \frac{1}{\langle \ell \rangle^4}\Bigg)^{1/2}\,\frac{1}{\langle k \rangle} \lesssim \frac{1}{\langle m \rangle^{\theta}\,\langle k \rangle}\,,
\end{equation}
for some $\theta>0$.

\paragraph{\emph{Analysis of $II_n(k)$}}

Let us first compute
\begin{multline}
\label{II_n(k)_L^2_norm_squared}
\|II_n(k)\|_{L^2(\Omega)}^2=\frac{1}{\beta^3}\,\int \,\mathop{\sum_{|\ell_1|\leq n, |\ell_2|\leq n, |\ell_3|\leq n}}_{\ell_1-\ell_2+\ell_3=k, \ell_1 \neq \ell_2} \,\mathop{\sum_{|\ell'_1|\leq n, |\ell'_2|\leq n, |\ell'_3|\leq n}}_{\ell'_1-\ell'_2+\ell'_3=k, \ell'_1 \neq \ell'_2} \widehat{V}(\ell_1-\ell_2)\, \widehat{V}(\ell'_1-\ell'_2)\,
\\
\times
\frac{g_{\ell_1}(\omega)\,\overline{g_{\ell_2}(\omega)}\,g_{\ell_3}(\omega)}{\langle \ell_1 \rangle\,\langle \ell_2 \rangle\,\langle \ell_3 \rangle}\,\frac{\overline{g_{\ell'_1}(\omega)}\,g_{\ell'_2}(\omega)\,\overline{g_{\ell'_3}(\omega)}}{\langle \ell'_1 \rangle\,\langle \ell'_2 \rangle\,\langle \ell'_3 \rangle}\,d \omega\,.
\end{multline}
We can use Wick's theorem to deduce that
\begin{equation}
\label{A_n,B_n,C_n}
\eqref{II_n(k)_L^2_norm_squared}\leq A_n(k)+B_n(k)+C_n(k)\,,
\end{equation}
where
\begin{align}
\label{A_n(k)}
A_n(k)&:=\frac{1}{\beta^3}\,\mathop{\sum_{|\ell_1|\leq n, |\ell_2|\leq n, |\ell_3|\leq n}}_{\ell_1-\ell_2+\ell_3=k, \ell_1 \neq \ell_2}\,\big(\widehat{V}(\ell_1-\ell_2)\big)^2\,\frac{1}{\langle \ell_1 \rangle^2 \langle \ell_2 \rangle^2 \langle \ell_3 \rangle^2}
\\
\label{B_n(k)}
B_n(k)&:=\frac{1}{\beta^3}\,\mathop{\sum_{|\ell_1|\leq n, |\ell_2|\leq n, |\ell_3|\leq n}}_{\ell_1-\ell_2+\ell_3=k, \ell_1 \neq \ell_2}\,\widehat{V}(\ell_1-\ell_2)\,\widehat{V}(\ell_2-\ell_3)\,\frac{1}{\langle \ell_1 \rangle^2 \langle \ell_2 \rangle^2 \langle \ell_3 \rangle^2}
\\
\label{C_n(k)}
C_n(k)&:=\frac{1}{\beta^3}\,\sum_{|\ell_2|\leq n, |\ell'_2|\leq n}\,\widehat{V}(k-\ell_2)\,\widehat{V}(k-\ell'_2)\,\frac{1}{\langle k \rangle^2 \langle \ell_2 \rangle^2 \langle \ell'_2 \rangle^2}\,.
\end{align}
We now analyse the cases $d=2$ and $d=3$ separately.
\paragraph{\emph{Analysis of $II_n(k)$ when $d=2$}}

%We first estimate $A_n(k)$.
We have
\begin{multline}
\label{A_n_bound_2d}
A_n(k)\lesssim_{\beta} \sum_{\ell_3} \Bigg(\mathop{\sum_{\ell_1,\ell_2}}_{\ell_1-\ell_2=k-\ell_3} \frac{1}{\langle \ell_1 \rangle^2 \langle \ell_2\rangle^2 }\Bigg)\,\big(\widehat{V}(k-\ell_3)\big)^2\,\frac{1}{\langle \ell_1 \rangle^2}
\\
\lesssim \sum_{\ell_3} \frac{\log \, \langle k-\ell_3 \rangle}{\langle k-\ell_3 \rangle^2}\,\big(\widehat{V}(k-\ell_3)\big)^2\,\frac{1}{\langle \ell_3\rangle^2}
\lesssim \sum_{\ell_3} \frac{1}{\langle k-\ell_3 \rangle^2 \langle \ell_3\rangle^2} \lesssim \frac{\log\,\langle k \rangle}{\langle k \rangle^2}\,.
\end{multline}
Here, we used Lemma \ref{discrete_convolution_2D} with $\delta=2$ and $M=\rho=0$ twice and we recalled \eqref{V_hat_estimates}.

Using
\begin{equation}
\label{AM-GM_inequality}
\widehat{V}(\ell_1-\ell_2)\,\widehat{V}(\ell_2-\ell_3) \leq \frac{1}{2} \Big[\big(\widehat{V}(\ell_1-\ell_2)\big)^2+\big(\widehat{V}(\ell_2-\ell_3)\big)^2\Big]
\end{equation}
in \eqref{B_n(k)} and arguing analogously as for \eqref{A_n_bound_2d}, we get that
\begin{equation}
\label{B_n_bound_2d}
B_n(k) \lesssim_{\beta} \frac{\log\,\langle k \rangle}{\langle k \rangle^2}\,.
\end{equation}
Finally, by \eqref{V_hat_estimates} and Lemma \ref{discrete_convolution_2D} with $\delta=\epsilon$ and $M=\rho=0$, we have
\begin{equation}
\label{C_n_bound_2d}
C_n(k) \lesssim_{\beta} \frac{1}{\langle k \rangle^2}\,\Bigg(\sum_{\ell} \widehat{V}(k-\ell)\,\frac{1}{\langle \ell \rangle^2}\Bigg)^2 \lesssim \frac{1}{\langle k \rangle^2}\,\bigg(\frac{\log\,\langle k \rangle}{\langle k \rangle^{\epsilon}}\bigg)^2 \lesssim \frac{1}{\langle k \rangle^2}\,.
\end{equation}
Using \eqref{A_n,B_n,C_n}, \eqref{A_n_bound_2d}, \eqref{B_n_bound_2d}, and \eqref{C_n_bound_2d}, we deduce that
\begin{equation}
\label{II_bound_2d}
\|II_n(k)\|_{L^2(\Omega)} \lesssim \frac{\log^{1/2}\,\langle k \rangle}{\langle k \rangle}\,.
\end{equation}

Let $n \geq m$ be given. We recall \eqref{gradient_bound_5} and argue analogously as for \eqref{II_n(k)_L^2_norm_squared} to write

\begin{multline}
\label{II_n(k)_diff_L^2_norm_squared}
\|II_n(k)-II_m(k)\|_{L^2(\Omega)}^2=\frac{1}{\beta^3}\,\int \,\mathop{\mathop{\sum_{|\ell_1|\leq n, |\ell_2|\leq n, |\ell_3|\leq n}}_{\max\{|\ell_1|,|\ell_2|,|\ell_3|\}>m}}_{\ell_1-\ell_2+\ell_3=k, \ell_1 \neq \ell_2} \,\mathop{\mathop{\sum_{|\ell'_1|\leq n, |\ell'_2|\leq n, |\ell'_3|\leq n}}_{\max\{|\ell'_1|,|\ell'_2|,|\ell'_3|\}>m}}_{\ell'_1-\ell'_2+\ell'_3=k, \ell'_1 \neq \ell'_2}  \widehat{V}(\ell_1-\ell_2)\, \widehat{V}(\ell'_1-\ell'_2)\,
\\
\times
\frac{g_{\ell_1}(\omega)\,\overline{g_{\ell_2}(\omega)}\,g_{\ell_3}(\omega)}{\langle \ell_1 \rangle\,\langle \ell_2 \rangle\,\langle \ell_3 \rangle}\,\frac{\overline{g_{\ell'_1}(\omega)}\,g_{\ell'_2}(\omega)\,\overline{g_{\ell'_3}(\omega)}}{\langle \ell'_1 \rangle\,\langle \ell'_2 \rangle\,\langle \ell'_3 \rangle}\,d \omega\,.
\end{multline}
As in \eqref{A_n,B_n,C_n}, we have
\begin{equation}
\label{A_n,B_n,C_n_diff}
\eqref{II_n(k)_diff_L^2_norm_squared}\leq A_{n,m}(k)+B_{n,m}(k)+C_{n,m}(k)\,,
\end{equation}
where we modify \eqref{A_n(k)}, \eqref{B_n(k)}, and \eqref{C_n(k)} as
\begin{align}
\label{A_{n,m}(k)}
A_{n,m}(k)&:=\frac{1}{\beta^3}\,\mathop{\mathop{\sum_{|\ell_1|\leq n, |\ell_2|\leq n, |\ell_3|\leq n}}_{\max\{|\ell_1|,|\ell_2|,|\ell_3|\}>m}}_{\ell_1-\ell_2+\ell_3=k, \ell_1 \neq \ell_2} \,\big(\widehat{V}(\ell_1-\ell_2)\big)^2\,\frac{1}{\langle \ell_1 \rangle^2 \langle \ell_2 \rangle^2 \langle \ell_3 \rangle^2}
\\
\label{B_{n,m}(k)}
B_{n,m}(k)&:=\frac{1}{\beta^3}\,\mathop{\mathop{\sum_{|\ell_1|\leq n, |\ell_2|\leq n, |\ell_3|\leq n}}_{\max\{|\ell_1|,|\ell_2|,|\ell_3|\}>m}}_{\ell_1-\ell_2+\ell_3=k, \ell_1 \neq \ell_2} \,\widehat{V}(\ell_1-\ell_2)\,\widehat{V}(\ell_2-\ell_3)\,\frac{1}{\langle \ell_1 \rangle^2 \langle \ell_2 \rangle^2 \langle \ell_3 \rangle^2}
\\
\label{C_{n,m}(k)}
C_{n,m}(k)&:=\frac{1}{\beta^3}\,\mathop{\sum_{|\ell_2|\leq n, |\ell'_2|\leq n}}_{\max\{|k|,|\ell_2|,|\ell'_2|\}>m}\,\widehat{V}(k-\ell_2)\,\widehat{V}(k-\ell'_2)\,\frac{1}{\langle k \rangle^2 \langle \ell_2 \rangle^2 \langle \ell'_2 \rangle^2}\,.
\end{align}
We observe that for any $\rho \in (0,2)$, we have
\begin{equation}
\label{A_{nm}_bound_2d}
A_{n,m}(k) \lesssim_{\beta,\rho}  \frac{\log\,\langle k \rangle}{\langle k \rangle^{2-\rho}\langle m \rangle^{\rho}}\,.
\end{equation}
In order to obtain \eqref{A_{nm}_bound_2d}, we argue similarly as for \eqref{A_n_bound_2d}.
If $\max\{|\ell_1|,|\ell_2|\}>m$, in our first application of Lemma \ref{discrete_convolution_2D}, we take $M=m$. We then use \eqref{V_hat_estimates} and argue as for \eqref{A_n_bound_2d}. If $|\ell_3|>m$, we take $M=0$ in the first application and $M=m$ in the second application of Lemma \ref{discrete_convolution_2D}.

Similarly, for any $\rho \in (0,2)$, we have
\begin{equation}
\label{B_{nm}_bound_2d}
B_{n,m}(k) \lesssim_{\beta,\rho}  \frac{\log\,\langle k \rangle}{\langle k \rangle^{2-\rho}\langle m \rangle^{\rho}}\,.
\end{equation}

Finally, for any $\rho \in (0,\epsilon)$, we have
\begin{equation}
\label{C_{nm}_bound_2d}
C_{n,m}(k) \lesssim_{\beta,\rho}  \frac{\log\,\langle k \rangle}{\langle k \rangle^{2-\rho}\langle m \rangle^{\rho}}\,.
\end{equation}
In order to deduce \eqref{C_{nm}_bound_2d}, we need to consider the contributions $\max\{|\ell_2|,|\ell'_2|\}>m$ and $|k|>m$ separately. If $\max\{|\ell_2|,|\ell'_2|\}>m$, we argue as for \eqref{C_n_bound_2d}, but in one of the applications of Lemma \ref{discrete_convolution_2D} we take $M=m$. If $|k|>m$, then \eqref{C_{nm}_bound_2d} follows from \eqref{C_n_bound_2d}.

Using \eqref{A_n,B_n,C_n_diff}, \eqref{A_{nm}_bound_2d}, \eqref{B_{nm}_bound_2d}, and \eqref{C_{nm}_bound_2d}, it follows that for $\theta>0$ sufficiently small
\begin{equation}
\label{II_Cauchy_2d}
\|II_n(k)-II_m(k)\|_{L^2(\Omega)} \lesssim_{\beta,\theta} \frac{\log^{1/2}\,\langle k \rangle}{\langle k \rangle^{1-\theta}\langle m \rangle^{\theta}}\,.
\end{equation}

Recalling \eqref{gradient_bound_5} and using \eqref{phi_omega} followed by Lemma \ref{Wiener_chaos} in \eqref{gradient_Cauchy_bound_1}, we have that for $\theta>0$ sufficiently small
\begin{align}
\notag
&\|\nabla h_n^I(u)-\nabla h_m^I(u)\|_{L^p(\mu_{\beta,0};H^{-s})}
\\
\notag
&\lesssim_p
\Bigg(\sum_{k, |k|\leq m} \langle k \rangle^{-2s}\,\|I_n(k)-I_m(k)\|_{L^2(\Omega)}^{2}+\sum_{k, |k|\leq m} \langle k \rangle^{-2s}\,\|II_n(k)-II_m(k)\|_{L^2(\Omega)}^{2}\Bigg)^{1/2}
\\
\label{gradient_Cauchy_2d}
&+\Bigg(\sum_{k , m<|k|\leq n} \langle k \rangle^{-2s}\,\|I_n(k)\|_{L^2(\Omega)}^{2}+\sum_{k , m<|k|\leq n} \langle k \rangle^{-2s}\,\|II_n(k)\|_{L^2(\Omega)}^{2}\Bigg)^{1/2} \lesssim_{s,\beta,\theta} \frac{1}{\langle m \rangle^{\theta}}\,.
\end{align}
Here, we used \eqref{I_bound}, \eqref{I_Cauchy}, \eqref{II_bound_2d}, \eqref{II_Cauchy_2d}, and the assumption that $s>0$. Therefore, $(h_n^I(u))$ is a Cauchy sequence in $L^p(\mu_{\beta,0};H^{-s})$.

\paragraph{\emph{Analysis of $II_n(k)$ when $d=3$}}
We now show that for $d=3$, \eqref{II_bound_2d} and \eqref{II_Cauchy_2d} get replaced by
\begin{equation}
\label{II_bound_3d}
\|II_n(k)\|_{L^2(\Omega)} \lesssim_{\beta} \frac{1}{\langle k \rangle^{2+\epsilon}}
\end{equation}
and
\begin{equation}
\label{II_Cauchy_3d}
\|II_n(k)-II_m(k)\|_{L^2(\Omega)} \lesssim_{\beta,\theta} \frac{1}{\langle k \rangle^{2+\epsilon-\theta} \langle m \rangle^{\theta}},
\end{equation}
for $\theta \in [0,\frac{1}{2})$, whenever $n \geq m$.
Using \eqref{I_bound}, \eqref{I_Cauchy}, \eqref{II_bound_3d}, \eqref{II_Cauchy_3d}, the fact that $s>\frac{1}{2}$, and arguing as in \eqref{gradient_Cauchy_2d}, we indeed deduce that $(h_n^I(u))$ is a Cauchy sequence in $L^p(\mu_{\beta,0};H^{-s})$.

Our goal is now to show \eqref{II_bound_3d} and \eqref{II_Cauchy_3d}.
With $A_n(k), B_n(k), C_n(k)$ defined as in \eqref{A_n(k)}, \eqref{B_n(k)}, \eqref{C_n(k)}, we have the following estimates. By arguing similarly as in \eqref{A_n_bound_2d}, we have
\begin{equation}
\label{A_n_bound_3d}
A_n(k)\lesssim_{\beta}
\sum_{\ell_3} \frac{1}{\langle k-\ell_3 \rangle}\,\big(\widehat{V}(k-\ell_3)\big)^2\,\frac{1}{\langle \ell_3\rangle^2}
\lesssim \sum_{\ell_3} \frac{1}{\langle k-\ell_3 \rangle^{5+2\epsilon} \langle \ell_3\rangle^2} \lesssim \frac{1}{\langle k \rangle^{4+2\epsilon}}\,.
\end{equation}
We again use \eqref{AM-GM_inequality} in \eqref{B_n(k)} and argue as in the proof of \eqref{A_n_bound_3d} to deduce that
\begin{equation}
\label{B_n_bound_3d}
B_n(k)\lesssim_{\beta}  \frac{1}{\langle k \rangle^{4+2\epsilon}}\,.
\end{equation}

Here, we first used Lemma \ref{discrete_convolution_3D} with $\delta=0$ and $M=\rho=0$. Then, we recalled \eqref{V_hat_estimates} and we used Lemma \ref{discrete_convolution_3D} with $\delta=2+\epsilon$ and $M=\rho=0$.

Similarly, when $d=3$, \eqref{V_hat_estimates} and Lemma \ref{discrete_convolution_3D} with $\delta=\epsilon$ and $M=\rho=0$ imply that
\begin{equation}
\label{C_n_bound_3d_auxiliary}
\sum_{\ell} \widehat{V}(k-\ell)\,\frac{1}{\langle \ell \rangle^2} \lesssim \sum_{\ell} \frac{1}{\langle k-\ell \rangle^{2+\epsilon}\langle \ell \rangle^2} \lesssim \frac{1}{\langle k \rangle^{1+\epsilon}}\,.
\end{equation}
Using \eqref{C_n_bound_3d_auxiliary} and arguing as in \eqref{C_n_bound_2d}, we get that
\begin{equation}
\label{C_n_bound_3d}
C_n(k) \lesssim_{\beta}  \frac{1}{\langle k \rangle^{4+2\epsilon}}\,.
\end{equation}
We hence obtain \eqref{II_bound_3d} from \eqref{A_n_bound_3d}, \eqref{B_n_bound_3d}, and \eqref{C_n_bound_3d}.
We now show \eqref{II_Cauchy_3d}. With $A_{n,m}(k)$, $B_{n,m}(k)$, $C_{n,m}(k)$ defined as in \eqref{A_{n,m}(k)}, \eqref{B_{n,m}(k)}, \eqref{C_{n,m}(k)}, we have the following estimates.

We observe that for any $\rho \in [0,1)$, we have
\begin{equation}
\label{A_{nm}_bound_3d}
A_{n,m}(k) \lesssim_{\beta,\rho}  \frac{1}{\langle k \rangle^{4+2\epsilon-\rho}\langle m \rangle^{\rho}}\,.
\end{equation}
Here, we argue as for \eqref{A_{nm}_bound_3d}.
By analogous arguments, we have
\begin{equation}
\label{B_{nm}_bound_3d}
B_{n,m}(k) \lesssim_{\beta,\rho}  \frac{1}{\langle k \rangle^{4+2\epsilon-\rho}\langle m \rangle^{\rho}}\,.
\end{equation}
Finally, modifying the proof of \eqref{C_n_bound_3d} analogously as in \eqref{C_{nm}_bound_2d}, we have that for any $\rho \in [0,1)$,
\begin{equation}
\label{C_{nm}_bound_3d}
C_{n,m}(k) \lesssim_{\beta,\rho}  \frac{1}{\langle k \rangle^{4+2\epsilon-\rho}\langle m \rangle^{\rho}}\,.
\end{equation}
We hence obtain \eqref{II_Cauchy_3d} from \eqref{A_{nm}_bound_3d}, \eqref{B_{nm}_bound_3d}, \eqref{C_{nm}_bound_3d}.

\paragraph{\textbf{(iii) The NLS equation on $\T$}}
We use the Sobolev embedding $H^{\zeta}(\T) \hookrightarrow L^q(\T)$ for $\zeta \in (\frac{1}{2}-\frac{1}{q},\frac{1}{2})$ in \eqref{eq.nl.4} to deduce that
\begin{equation*}
h^I(u) \lesssim_\zeta \|u\|_{H^{\zeta}}^q\,.
\end{equation*}
We then deduce that $h^I(u) \in L^p(\mu_{\beta,0})$ as in part (i).
A direct calculation shows that
\begin{equation*}
\nabla h^I(u)=|u|^{q-2}\,u\,.
\end{equation*}
By using Lemma \ref{product_lemma} $q-1$ times for sufficiently small $\alpha$, and by arguing as in \eqref{gradient_1_bound_1}, we get that for some $\zeta' \in (-s,\frac{1}{2})$
\begin{equation*}
\|\nabla h^I(u)\|_{L^p(\mu_{\beta,0}; H^{-s})}
\lesssim_{s,q,\zeta'} \bigg(\int \|u\|_{H^{\zeta'}}^{(q-1)p} \, d \mu_{\beta,0}\bigg)^{1/p}<\infty\,.
\end{equation*}
We conclude that $h^I \in \mathbb D^{1,p}(\mu_{\beta,0})$.

\noindent
\paragraph{\textbf{(iv) The NLS equation on $\T^2$}}

In case $h^I$ is given by \eqref{eq.nl.1}, the result is a consequence of
\cite[Proposition 1.1]{MR3844655} and \cite[Proposition 1.3]{MR3844655}. Indeed, it is proved there that $(h^I_n)_{n\in\N}$ and $(\nabla h^I_n)_{n\in\N}$ are Cauchy sequences in $L^p(\mu_{\beta,0})$ and $L^p(\mu_{\beta,0};H^{-s})$ respectively. We omit the details
\end{proof}
Thus, as a consequence of Theorems \ref{thm.infdN} and \ref{thm.KMSGibbs}, one concludes that the above NLS dynamical systems on the torus $\T^d$ admit each a unique KMS state given by the Gibbs measure $\mu_{\beta}=
z_\beta^{-1}e^{-\beta h^I}\mu_{\beta,0}$ with $z_\beta$ an appropriate normalization constant. Note that uniqueness here is in the sense of Theorem \ref{thm.KMSGibbs} and it is among  measures that are absolutely continuous with respect to $\mu_{\beta,0}$ with a density $\varrho\in \mathbb{D}^{1,2}(\mu_{\beta,0})$. Finally, one remarks that such a result suggests the study of general dynamical systems satisfying  the condition $h^I\in\mathbb D^{1,2}(\mu_{\beta,0})$ without relying  on the precise form of $h^I$. The above discussion is summarized below.
\begin{cor}
\label{cor_nls_kms}
Let $\beta>0$ and $s> -1/2$ satisfying \eqref{s_choice}. Consider $h^I$ to be one of the nonlinear Borel functionals $h^I: H^{-s}\to\R$  of the  Hartree or  NLS equations given respectively by \eqref{eq.nl.3}, \eqref{eq.nl.3_Wick}, \eqref{eq.nl.4} and \eqref{eq.nl.1}.  Then:
\begin{itemize}
 \item [(i)] The Gaussian measure $\mu_{\beta,0}$ is the unique  $(\beta,X_0)$-KMS state of the vector field $X_0=-iA$.
  \item [(ii)] The nonlinear functionals $h^I\in \mathbb D^{1,p}(\mu_{\beta,0})$ and $e^{-\beta h^I}\in L^p(\mu_{\beta,0})$ for all $1\leq p<\infty$.

  \item [(iii)] The Gibbs measure
  $$
  \mu_\beta=\frac{e^{-\beta h^I}}{\int_{H^{-s}} e^{-\beta h^I} d\mu_{\beta,0}}\mu_{\beta,0}\,
  $$
  is a stationary solution of the Liouville equation \eqref{eq.stat}.

   \item [(iv)] The Gibbs measure $\mu_\beta$ is the unique $(\beta,X)$-KMS state, for the vector field $X=-i( A+\nabla h^I)$, among all the absolutely continuous measures $\mu$ with respect to $\mu_{\beta,0}$  such that $\frac{d\mu}{d\mu_{\beta,0}}\in \mathbb{D}^{1,2}(H^{-s})$.
\end{itemize}
\end{cor}
\begin{remark}
\label{rem:dom}
The previous corollary extends also  to the NLS nonlinearity \eqref{eq.nl.1} on $2$-dimensional compact Riemannian  manifolds without boundary or on bounded domains in $\R^2$  following \cite[Proposition 4.3, 4.5 and 4.6]{MR3844655}. We omit the details.
\end{remark}

\subsection{Nonlinear wave (Klein-Gordon) equations}

One can study Gibbs measures for nonlinear wave (Klein-Gordon) equations by means of probabilistic and PDE methods, see e.g.\ \cite{Bringmann1, Bringmann2, MR3869074,MR2425134,MR785258,MR1277197,MR4133695} and the references therein. The nonlinearities that are usually considered are similar to the ones recalled in the above Subsection \ref{sub.sec.NLS}. The main difference comes from the use of the real structure of fields rather than the complex one.
 Specifically, we  consider the nonlinear wave (Klein-Gordon) equation on the torus $\T^d$, $d=1,2,3$,
\begin{equation}
\label{wave_eq}
\begin{cases}
\partial^2_t u+u-\Delta u+\nabla h^I(u)=0 &\quad \\
u_{|t=0}=f_0, \; \partial_tu_{|t=0}=f_1 &
\end{cases}
\end{equation}
where $\nabla h^I$ is the Malliavin derivative of some functional $h^I$ that we will specify below. Before proceeding, we explain how the  nonlinear wave equation \eqref{wave_eq} fits within the general framework of Subsection \ref{subsec.nl}.

\medskip
\paragraph{\emph{Framework for wave equations:}}  Consider the  Hilbert space $H=L_\R^2(\T^d)\oplus L_\R^2(\T^d)$  where  $L_\R^2$ stands for real-valued square integrable  functions and define  the  Sobolev spaces for $\gamma\in\R$ as
\begin{equation}
\label{hs_wa}
H^\gamma=H^{\gamma}_\R(\T^d)\oplus H^{\gamma-1}_\R(\T^d)\,.
\end{equation}
The nonlinear wave equation \eqref{wave_eq} takes the form
\[
\partial_t \begin{pmatrix} u \\ v\end{pmatrix} =X\begin{pmatrix} u \\ v\end{pmatrix} =X_0\begin{pmatrix} u \\ v\end{pmatrix}+ \begin{pmatrix} 0 \\ -\nabla h^I(u) \end{pmatrix}
\]
with the vector field $X_0$ given by
\begin{equation}
\label{field_wa_0}
X_0\equiv\begin{bmatrix} 0 & \mathds 1\\ \Delta-\mathds 1 & 0\end{bmatrix} = \overbrace{\begin{bmatrix} 0 & \mathds 1\\ -\mathds 1 & 0\end{bmatrix}}^J
\;\overbrace{\begin{bmatrix} -\Delta+\mathds 1 & 0\\ 0 & \mathds 1 \end{bmatrix}}^A
\end{equation}
with $J$ is a compatible complex structure on $H$ and $A$ is a positive linear operator. Remark that  $H$ is endowed with a canonical symplectic structure induced by $J$ and given by
$$
\sigma(u\oplus v, u'\oplus v')=\langle u\oplus v, J u'\oplus v'\rangle_{L^2_\R\oplus L_\R^2}\,.
$$
Moreover, $H$ can be considered as a complex Hilbert space according to \eqref{complex_str_1}-\eqref{complex_str_2} and the couple $(J,A)$ defines a  complex linear Hamiltonian system as in \eqref{fre-Ham1}-\eqref{fre-Ham2}. Now, the vector field $X$  can be written as
\begin{equation}
\label{vec_field_wa}
X\begin{pmatrix} u \\ v\end{pmatrix} =J \left(A\begin{pmatrix} u \\ v\end{pmatrix}+ \begin{pmatrix} \nabla h^I(u) \\ 0 \end{pmatrix}\right)\,.
\end{equation}
 Although the operator $A$ does not have a compact resolvent, it still possible to do the same analysis as before. Indeed, one uses the Sobolev spaces in \eqref{hs_wa} instead of the definition given in Subsection \ref{sub.sec.sobolevset}.

\medskip
\paragraph{\emph{Results for wave equations:}} The Gaussian Gibbs measure $\mu_{\beta,0}$ in this case is defined as a product measure such that
$$
d\mu_{\beta,0}(u,v)=d\mu_{\beta,0}^1(u) \, d\mu_{\beta,0}^2(v),
$$
with  $\mu_{\beta,0}^1$ and  $\mu_{\beta,0}^2$ are the Gaussian measures on the distribution space $\mathscr{D}'(\T^d)$ with covariance operators $ \beta^{-1} (-\Delta+\mathds 1)^{-1}$ and ${\beta}^{-1} \mathds 1$ respectively. The existence and uniqueness of such measures  follow from Corollary \ref{cor.1}. According to Theorem \ref{thm.gauss}, for $s>-1/2$ satisfying \eqref{s_choice}, the measure $\mu_{\beta,0}$ coincides with the centered Gaussian measure with  covariance operator
$$
\beta^{-1} \begin{bmatrix}(-\Delta+\mathds 1)^{-(1+s)} & 0\\ 0 & (-\Delta+\mathds 1)^{-(1+s)} \end{bmatrix}
$$
 on the Sobolev space $H^{-s}$  given in \eqref{hs_wa}. In particular, $\mu_{\beta,0}$ is a Borel probability measure over $H^{-s}$. Therefore,
 Theorem \ref{thm.nucgaus} shows that $\mu_{\beta,0}$ is a $(\beta,X_0)$-KMS state with the vector field $X_0$  in \eqref{field_wa_0}.  Moreover,  using the arguments of Subsection \ref{sub.sec.NLS} one can define rigourously the Gibbs measure of the nonlinear wave equation as
$$
\mu_\beta=\frac{1}{z_\beta}e^{-\beta h^I} \mu_{\beta,0}^1\otimes \mu_{\beta,0}^2\,,
$$
where $z_\beta$ is a normalization constant and $h^I$ is one of the following possibilities
\begin{equation}
\label{nonlin_wa}
h^I=
\begin{cases}
  \eqref{eq.nl.3} \text{ or } \eqref{eq.nl.4} & \text{ if } d=1,  \\
  \eqref{eq.nl.3_Wick} \text{ or } \eqref{eq.nl.1} & \text{ if } d=2,\\
  \eqref{eq.nl.3_Wick} & \text{ if } d=3. \\
\end{cases}
\end{equation}
Note that since the fields are real, the Wick-ordered power nonlinearity $:u_n^{2r}:$
in \eqref{eq.39}  is defined in this case through Hermite polynomials $H_r$ instead of Laguerre polynomials,
\begin{equation}
\label{eq.nl.5}
:u_n^{2r}: \,=\sigma_{n,\beta}^{r/2}\;H_r\bigg(\sigma_{n,\beta}^{-1/2} u_n\bigg)\,.
\end{equation}
In particular,  the nonlinearity $h^I$ depends only on the $u$ variable.
Furthermore, in \eqref{phi_omega}, we add the condition
\begin{equation}
\label{g_k_real}
g_{-k}=\overline{g_k}.
\end{equation}
By arguing analogously as in the proof of Proposition \ref{lem.9}, we deduce that $h^I$ belongs to the Gross-Sobolev spaces $\mathbb{D}^{1,p}(\mu_{\beta,0}^1)$ or $\mathbb{D}^{1,p}(\mu_{\beta,0})$ for all $1\leq p<\infty$ (note that additional condition \eqref{g_k_real} does not increase the $L^p(\Omega)$ norms of the relevant quantities). Consequently the Malliavin derivatives $\nabla h^I$ are well-defined. Thus, we have at hand all the ingredients to  apply Theorem \ref{thm.infdN} and \ref{thm.KMSGibbs}. So, all the statements (i)-(iv) of Corollary \ref{cor_nls_kms}  with the appropriate modifications hold true for the nonlinear wave equation \eqref{wave_eq} with the nonlinearities \eqref{nonlin_wa}. In particular, we emphasize the following result where $H^{-s}$ is defined according to \eqref{hs_wa} and $s$ satisfying \eqref{s_choice}.

\begin{cor}
\label{cor:wave}
The Gibbs measure $\mu_\beta=\frac{1}{z_\beta}e^{-\beta h^I} \mu_{\beta,0}^1\otimes \mu_{\beta,0}^2$ is the unique $(\beta,X)$-KMS state, for the vector field $X=J A+J(\nabla h^I\oplus 0)$, among all the absolutely continuous measures $\mu$ with respect to $\mu_{\beta,0}$  such that
$\frac{d\mu}{d\mu_{\beta,0}}\in \mathbb{D}^{1,2}(H^{-s})$.
\end{cor}

As in Remark \ref{rem:dom}, the above Corollary extends to wave equations with the nonlinearity  \eqref{eq.nl.1} defined on $2$-dimensional compact Riemannian  manifolds without boundary or on bounded domains in $\R^2$  following the discussion in \cite[Section 1.2]{MR4133695}.

\subsection{The focusing NLS and the local KMS condition}
\label{subsec.focusing}
Consider the \emph{focusing NLS} equation on the one-dimensional torus $\T$,
$$
i\partial_t u=-\Delta u+u-|u|^{p-2} u\,,
$$
for $4 \leq p \leq 6$ where the nonlinear energy functional is given by
$$
h^I(u)=-\frac{1}{p} \int_\T |u(x)|^p \, dx\,,
$$
which is similar to \eqref{eq.nl.4} with a negative sign corresponding to a focusing nonlinearity. Recall that here one has the same framework as in Subsection \ref{sub.sec.NLS} with $s=0$.
Although it is not possible to define a global Gibbs measure in this case because of the negative sign in the front of the nonlinear term $h^I$, it is proved in \cite{MR1309539,MR939505} that the Gibbsian local measure
\begin{equation}
\label{eq.muloc}
\mu_{\beta}=\frac{1}{z_{\beta,R}}\,e^{-\beta h^I(\cdot)} \;\mathds 1_{[0,R]}\big(\|\cdot\|^2_{L^2(\T)}\big) \;\mu_{\beta,0}\,,\quad z_{\beta,R}:=\int_{L^2(\T)} e^{-\beta h^I(\cdot)} \;\mathds 1_{[0,R]}\big(\|\cdot\|^2_{L^2(\T)}\big) \;d \mu_{\beta,0}
\end{equation}
is well-defined for some arbitrary  constant $R>0$ if $4\leq p<6$ or $R>0$ sufficiently small  if $p=6$.

\begin{lemma}
\label{lem.8}
Let $\beta>0$ be given  and $\chi\in\mathscr{C}_{c}^\infty(\R)$ such that $\chi(x)=0$ for all $|x|\geq R$   where  $R$ is as in \eqref{eq.muloc}. Then
the Radon-Nikodym derivative of the measure $\mu_\beta$ in \eqref{eq.muloc} with respect to $\mu_{\beta,0}$ satisfies
 $$
 \frac{d\mu_{\beta}}{d\mu_{\beta,0}} \in \mathbb{D}^{1,2}(\mu_{\beta,0})\,.
 $$
\end{lemma}
\begin{proof}
Note that $|\chi(\cdot)|\leq c \mathds 1_{[0,R]}(\cdot)$ for some constant $c>0$. Hence, the statement  $G(\cdot)=e^{-\beta h^I(\cdot)}\; \chi(\|\cdot\|_{L^2(\T)}^2)\in L^q(\mu_{\beta,0})$ for all $q\geq 1$, is essentially proved in \cite{MR1309539,MR939505}.  For an expository summary of the construction, we refer the reader to \cite{Andrew_Rout_PhD_thesis}. Using Lemma \ref{lem.9} and the approximation idea  in Lemma \ref{lem.log}, one shows
\begin{equation}
\label{eq.36}
\nabla G(u)= -\beta \nabla h^I(u) \, G(u)+ 2e^{-\beta h^I(\cdot)}\chi'(\|\cdot\|_{L^2(\T)}^2) \, u
\,\in L^2(\mu_{\beta,0};L^2(\T))\,.
\end{equation}
In fact, take $\theta_k$ the same sequence of functions as in the proof of Theorem \ref{thm.infdN} and
$$
G_k(\cdot)=\theta_k(h^I(\cdot)) \; \chi(\|\cdot\|_{L^2(\T)}^2)\,.
$$
Since one knows by Lemma \ref{lem.9} that the functionals $\|\cdot\|_{L^2(\T)}^2$ and $h^I$ belong to $ \mathbb{D}^{1,p}(\mu_{\beta,0})$, then using  Lemma \ref{lem.chi} and the chain rule \eqref{prod.rule}  one proves in $L^2(\mu_{\beta,0}; L^2(\T))$,
\begin{equation}
\label{eq.37}
\nabla G_k(u)=\theta_k'(h^I(u)) \,\nabla h^{I}(u) \; \chi(\|u\|_{L^2(\T)}^2)+ 2
\theta_k(h^I(u)) \; \chi'(\|u\|_{L^2(\T)}^2)\, u\,.
\end{equation}
Hence, dominated convergence  with the estimates \eqref{eq.th.est} give the following limits in $L^4(\mu_{\beta, 0})$ and $L^4(\mu_{\beta, 0}; L^2(\T))$ respectively,
$$
\lim_k \theta_k'(h^I(\cdot)) \; \chi(\|\cdot\|_{L^2(\T)}^2)= -\beta e^{-\beta h^I(\cdot)} \; \chi(\|\cdot\|_{L^2(\T)}^2)\,,
$$
and
$$
\lim_k  \theta_k(h^I(u)) \; \chi'(\|u\|_{L^2(\T)}^2)\, u= e^{-\beta h^I(\cdot)} \; \chi'(\|\cdot\|_{L^2(\T)}^2)\,u\,.
$$
Thus, the above limits with the H\"older inequality yield the claimed identity \eqref{eq.36} when carrying $k\to\infty$ in the equality \eqref{eq.37}.
\end{proof}

\medskip
We show here that such a measure $\mu_{\beta}$ satisfies a local form of the KMS condition.

\begin{proposition}
\label{prop.locKMS}
Let $\beta>0$ be given and let $\mu_\beta$ be the Borel measure defined in \eqref{eq.muloc}. Let  $\chi\in\mathscr C_{c}^\infty(\R)$ be such that $\chi(x)=0$ for all $|x|\geq R$ with $R$  the radius given in \eqref{eq.muloc}. Then, taking $F(\cdot)=\chi(\|\cdot\|_{L^2(\T)}^2)$, we have that for all $G\in\mathscr{C}_{c,cyl}^\infty(L^2(\T))$
\begin{equation}
\label{KMS-loc}
\int_{L^2(\T)} \{F,G\} \; d\mu_\beta= \beta \int_{L^2(\T)} \Ree\big\langle \nabla F(u), i \Delta u-iu-i\nabla h^I(u)\big\rangle_{L^2(\T)} \;G(u)\; d\mu_\beta\,.
\end{equation}
\end{proposition}

\begin{proof}
It follows by applying the integration by parts formula in Proposition \ref{lem.ibp}. Since
$G\in\mathscr{C}_{c,cyl}^\infty(L^2(\T))$ there exists $n\in\N$  and $\psi\in\mathscr C_{c}^\infty(\R^{2n})$ such that $G(u)=\psi(\pi_n u)$. So, using the Poisson bracket formula \eqref{pois.bra} and the fact that $\partial_{e_j}F(u)=\partial_{f_j}F(u)=0$ if $\|u\|_{L^2(\T)}^2>R$, one shows
$$
\int_{L^2(\T)} \{F,G\} \; d\mu_\beta=\frac{1}{z_{\beta,R}}\,\sum_{j=1}^n \int_{L^2(\T)} \,\bigg(
\partial_{e_j}F(u) \, \partial_{f_j}G(u)-  \partial_{e_j}G(u) \,\partial_{f_j}F(u) \bigg)
 e^{-\beta h^I(u)} \,d\mu_{\beta,0}\,,
$$
Thanks to Lemma \ref{lem.8}, one knows that $\partial_{e_j}F(\cdot) \,e^{-\beta h^I(\cdot)}\in\mathbb D^{1,2}(\mu_{\beta,0})$. Therefore, applying Proposition \ref{lem.ibp} with the function $G$  in \eqref{eq.ibp} replaced by $\widetilde F_j=\partial_{e_j}F \,e^{-\beta h^I}\in\mathbb D^{1,2}(\mu_{\beta,0})$ and $F$ by $G\in\mathscr{C}_{c,cyl}^\infty(L^2(\T))$, one obtains
$$
 \int_{L^2(\T)}   \,\partial_{e_j}F(u) \,e^{-\beta h^I(u)} \, \langle \nabla G(u),f_j\rangle\,d\mu_{\beta,0}=\int_{L^2(\T)}  G(u)  \bigg (-\partial_{f_j}\widetilde F(u) + \beta\, \widetilde F_j(u) \,\langle u, A f_j \rangle\bigg)\,
 d\mu_{\beta,0}\,,
$$
and similarly
$$
 \int_{L^2(\T)}   \,\partial_{f_j}F(u) \,e^{-\beta h^I(u)} \, \langle \nabla G(u),e_j\rangle\,d\mu_{\beta,0}=\int_{L^2(\T)}  G(u)  \bigg (-\partial_{e_j}\overset{\smile}{F}_j(u) + \beta\, \overset{\smile}{F}_j(u) \,\langle u, A e_j \rangle\bigg)\,
 d\mu_{\beta,0}\,,
 $$
where $\overset{\smile}{F}_j=\partial_{f_j}F \,e^{-\beta h^I}\in\mathbb D^{1,2}(\mu_{\beta,0})$ and $A$ is given by \eqref{eq.38}. Remark that the proof of Lemma \ref{lem.8} yields
\begin{eqnarray*}
\partial_{f_j}\widetilde F_j(u)&=& -\beta \partial_{f_j}h^I(u) \,  \partial_{e_j}F(u) \, e^{-\beta h^I(u)}
+\partial_{f_j}\partial_{e_j}F(u) \, e^{-\beta h^I(u)}\,,\\
\partial_{e_j}\overset{\smile}{F}_j(u)&=& -\beta \partial_{e_j}h^I(u) \,  \partial_{f_j}F(u) \, e^{-\beta h^I(u)}
+\partial_{e_j}\partial_{f_j}F(u) \, e^{-\beta h^I(u)}\,.
\end{eqnarray*}
Hence, one concludes
 \begin{eqnarray*}
 \int_{L^2(\T)} \{F,G\} \; d\mu_\beta&=&\frac{\beta}{z_{\beta,R}}\,\sum_{j=1}^n \int_{L^2(\T)}  G(u)  \bigg (
 \partial_{e_j}F(u) \partial_{f_j}h^I(u) -\partial_{e_j}h^I(u) \partial_{f_j}F(u)  \\
 &&
  +\partial_{e_j}F(u) \langle Au, f_j\rangle - \partial_{f_j}F(u) \langle Au, e_j\rangle  \bigg)\, e^{-\beta h^I(u)}\,d\mu_{\beta,0}\\
  &=& \frac{\beta}{z_{\beta,R}}\, \int_{L^2(\T)}  G(u)  \bigg(\big\langle \nabla F(u), -i\nabla h^I(u)\big\rangle
 +\big\langle -i Au,\nabla F(u)\big\rangle  \bigg)\, e^{-\beta h^I(u)}\,d\mu_{\beta,0}\,.
 \end{eqnarray*}
Thus, recalling  that $\langle\cdot, \cdot\rangle=\Ree\langle\cdot, \cdot\rangle_{L^2(\T)}$, one proves the local KMS condition.
\end{proof}

\begin{remark}
It is not difficult to see that one can replace the assumption  on $F$ in Proposition \ref{prop.locKMS} with
%$\mathscr{C}^2$-Fr\'echet differentiable bounded functions
 $F\in\mathscr{C}_{b}^2(L^2(\T))$ such that $F(u)=0$ for all   $u$ with $\|u\|^2_{L^2(\T)}>R'$ and  $R' \in (0,R)$ arbitrary.   Namely, taking $\chi\in\mathscr C^\infty_c(\R)$ such that $\chi(x)=1$ for all $|x| \leq R'$ and $\chi(x)=0$ for all $|x|>R$, then we have $F e^{-\beta h^I}=F \chi(\|\cdot\|^2_{L^2(\T)}) e^{-\beta h^I}$. So, the boundedness of $F$ and  Lemma \ref{lem.8} shows that  $F e^{-\beta h^I}\in L^2(\mu_{\beta,0})$. Moreover, the product rule yields
 $$
 \nabla[ Fe^{-\beta h^I}]= \nabla F\, \chi(\|\cdot\|^2_{L^2(\T)}) e^{-\beta h^I}+ F\, \nabla[ \chi(\|\cdot\|^2_{L^2(\T)}) e^{-\beta h^I}] \in L^{2}(\mu_{\beta,0}; L^2(\T))\,.
 $$
 Thus, one concludes that $Fe^{-\beta h^I}\in\mathbb{D}^{1,2}(\mu_{\beta,0})$ which is the main point in the proof of Proposition \ref{prop.locKMS}.
\end{remark}

\bigskip
\appendix
\section{Malliavin calculus}
\label{appx}

For completeness, we give a short overview of some useful tools  from Malliavin calculus. In particular, the following integration by parts formula is useful.
\begin{proposition}
\label{lem.ibp}
Let $F\in \mathscr C_{b,cyl}^\infty(H^{-s})$ and $G\in \mathbb{D}^{1,2}(\mu_{\beta,0})$ or $F\in \mathbb{D}^{1,2}(\mu_{\beta,0})$ and $G\in \mathscr C_{b,cyl}^\infty(H^{-s})$. Then for any $\varphi\in H^{1}$,
\begin{equation}
\label{eq.ibp}
\int_{H^{-s}} \, G(u) \langle \nabla F(u), \varphi\rangle \, d\mu_{\beta,0} = \int_{H^{-s}} \,
F(u) \bigg (-\langle \nabla G(u),\varphi\rangle + \beta\, G(u) \,\langle u, A\varphi \rangle\bigg)\, d\mu_{\beta,0}.
\end{equation}
\end{proposition}
\begin{proof}
First, one proves that for any $R\in \mathscr C_{b,cyl}^\infty(H^{-s})$ and $\varphi\in {\rm span}_\R\{e_j,f_j;j=1,\ldots,k\}$ for some $k \in \N$, we have
\begin{equation}
\label{eq.25}
\int_{H^{-s}} \,  \langle \nabla R(u), \varphi\rangle \, d\mu_{\beta,0} = \beta \int_{H^{-s}} \, \langle u, A\varphi \rangle \,R(u)\, d\mu_{\beta,0}\,.
\end{equation}
Indeed, the above equality \eqref{eq.25} follows from equation \eqref{eq.13} in Lemma \ref{lem.2} and a standard integration by parts on $\R^{2k}$. On the one hand taking $R=\psi\circ \pi_n$ for some $n\in\N$ with $n \geq k$ and $\psi\in \mathscr C_{b}^\infty(\R^{2n})$ and integrating by parts again, we have
 \begin{align}
 \notag
 \int_{H^{-s}} \,  \langle \nabla R(u), \varphi\rangle \, d\mu_{\beta,0} &= \sum_{j=1}^n\int_{\R^{2n}} \,\bigg(  \langle e_j, \varphi\rangle \, \partial^{(1)}_j \psi+
  \langle f_j, \varphi\rangle \,\partial^{(2)}_j \psi \bigg)\; d \nu^n_{\beta,0}\\
  \label{eq.25A}
  &= \beta \int_{\R^{2n}} \, \bigg\langle \sum_{j=1}^n \lambda_j (x_j e_j + y_j f_j), \varphi\bigg\rangle  \,  \psi(x_1,\dots,x_n; y_1,\dots,y_n) \; d \nu^n_{\beta,0}\,,
  \end{align}
  where $\nu^n_{\beta,0}$ is the Gaussian measure in Lemma \ref{lem.2}. On the other hand, we have
  \begin{align}
  \notag
   \beta \int_{H^{-s}} \, \langle u, A\varphi \rangle \,R(u)\, &d\mu_{\beta,0} = \beta \sum_{j=1}^k\int_{\R^{2n}} \, \lambda_j \bigg( \langle u, e_j\rangle \langle e_j,\varphi\rangle + \langle u, f_j\rangle \langle f_j,\varphi\rangle \bigg) \,  \psi(\pi_n u) \; d \mu_{\beta,0}\\
\label{eq.25B}
   &=
   \beta \int_{\R^{2n}} \, \bigg\langle \sum_{j=1}^k \lambda_j (x_j e_j + y_j f_j), \varphi\bigg\rangle  \,  \psi(x_1,\dots,x_n; y_1,\dots,y_n) \; d \nu^n_{\beta,0}\,,
 \end{align}
where  $\pi_n$ is the mapping in \eqref{eq.pi}.
Since $n \geq k$, we have that
\begin{equation}
\label{eq.25C}
\bigg\langle \sum_{j=1}^k \lambda_j (x_j e_j + y_j f_j), \varphi\bigg\rangle=\bigg\langle \sum_{j=1}^k \lambda_j (x_j e_j + y_j f_j), \varphi\bigg\rangle\,.
\end{equation}
We hence deduce \eqref{eq.25} from \eqref{eq.25A}, \eqref{eq.25B}, and \eqref{eq.25C}.
Now, the identity \eqref{eq.25} extends to all $\varphi\in H^{1}$
thanks to a standard approximation argument and Remark \ref{rem.mu}-(iii). The integration by parts  formula \eqref{eq.ibp}, for any $F,G\in \mathscr C_{b,cyl}^\infty(H^{-s})$, is a straightforward consequence of \eqref{eq.25} with $R=FG\in \mathscr C_{b,cyl}^\infty(H^{-s})$ and the product rule \eqref{prod.rule}. Finally, \eqref{eq.ibp} extends to all $G\in \mathbb{D}^{1,2}(\mu_{\beta,0})$ (resp.~$F\in \mathbb{D}^{1,2}(\mu_{\beta,0})$) by the density of
$\mathscr C_{b,cyl}^\infty(H^{-s})$ in $\mathbb{D}^{1,2}(\mu_{\beta,0})$ with respect to its graph norm
\eqref{eq.normD}.
 \end{proof}

\begin{lemma}
\label{lem.chi}
Let $\chi\in \mathscr C_b^1(\R)$ and $F\in\mathbb{D}^{1,p}(\mu_{\beta,0})$, for $p\in [1,\infty)$, then $\chi(F)\in\mathbb{D}^{1,p}(\mu_{\beta,0})$ and
$$
\nabla \chi(F)=\chi'(F) \nabla F\,.
$$
\end{lemma}
\begin{proof}
 Suppose that we are given $F_n\in\mathscr C_{c,cyl}^\infty(H^{-s})$, $n\in\N$, a sequence such that $F_n\to F$ in $\mathbb{D}^{1,p}(\mu_{\beta,0})$. Then the chain rule yields
$$
\nabla \chi(F_n)=\chi'(F_n) \nabla F_n\,.
$$
Since $F_n\to F$ in $L^p(\mu_{\beta,0})$, there exists a subsequence $(F_{n_k})_k$ such that
 $F_{n_k}\to F$ $\mu_{\beta,0}$-almost everywhere. Therefore, one obtains
 $$
 \lim_k \chi'(F_{n_k}) \nabla F_{n_k}= \chi'(F) \nabla F\,,
 $$
 in $L^p(\mu_{\beta,0})$ and $(\chi(F_{n_k}))_k$ is a Cauchy sequence in  $\mathbb{D}^{1,p}(\mu_{\beta,0})$ which is a Banach space.
 \end{proof}

The space $\mathcal P(H^{-s})$ is defined as the set of smooth cylindrical functions $F:H^{-s}\to \R$ such that there exists $n \in \N$ and
$F(u)=\varphi(\langle u,e_1\rangle,\dots,\langle u,e_n\rangle;\langle u,f_1\rangle,\dots,\langle u,f_n\rangle)$ for all $u\in H^{-s}$, where $\varphi\in\mathscr C^\infty(\R^{2n})$ is such that for all multi-indices  $\alpha\in\N^{2n}$, there exists constants $m_\alpha,c_\alpha\geq 0$  such that
$$
|\partial^\alpha \varphi(x)|\leq c_\alpha (1+|x|^2)^{m_\alpha}\,
$$
for all $x \in \R^{2n}$.

\begin{lemma}
The following inclusions hold true for all $p\in [1,\infty)$,
$$
\mathscr C_{b,cyl}^\infty(H^{-s})\subseteq\mathcal P(H^{-s})\subseteq \mathbb D^{1,p} (\mu_{\beta,0})\,.
$$
\end{lemma}
\begin{proof}
Recall the explicit form of the centered Gaussian measures $\nu_{\beta,0}^n=(\pi_n)_{\sharp}\mu_{\beta,0}$ defined over  $\R^{2n}$ and given in Lemma \ref{lem.2}. Since all the moments of such measures are finite, one concludes for all $m\geq 0$,
\begin{equation}
\label{eq.33}
\int_{H^{-s}} \big(1+\sum_{j=1}^n\langle u,e_j\rangle^2+\langle u,f_j\rangle^2\big)^m \;d\mu_{\beta,0}= \int_{\R^{2n}} \big(1+\sum_{j=1}^n x_j^2+y_j^2\big)^m \;d\nu^n_{\beta,0}
<\infty\,.
\end{equation}
Hence,  $\mathcal P(H^{-s})$ is included in all the spaces $L^p(\mu_{\beta,0})$ for all $p\in [1,\infty)$. Since the gradient of $F\in\mathcal P(H^{-s})$ is also given by the identity \eqref{eq.5}, one obtains using the estimates \eqref{eq.33} that $\nabla F\in L^p(\mu_{\beta,0}; H^{-s})$ for all $p\in [1,\infty)$.
\end{proof}

The following well-known result asserts that a random variable whose Malliavin derivative is zero, is almost surely constant.
\begin{proposition}
\label{lem.cst}
Let $F\in \mathbb D^{1,2}(\mu_{\beta,0})$  such that  $\nabla F=0$ for $\mu_{\beta,0}$-almost surely.  Then $F$ is constant  $\mu_{\beta,0}$-almost surely.
\end{proposition}
\begin{proof}
It is a straightforward consequence of the Wiener chaos decomposition (see e.g. \cite[Proposition 1.2.2 and 1.2.5]{MR2200233}). In fact, consider $H_n(\cdot)$ to be the $n^{th}$ Hermite polynomial and
$$
\Psi_{a,b}(\cdot)= \sqrt{a!\,b!} \; \prod_{j=1}^\infty H_{a_j}\big(\langle \cdot, \tilde e_j\rangle\big) \;
\prod_{j=1}^\infty H_{b_j}\big(\langle \cdot, \tilde f_j\rangle\big)\,,
$$
with $\tilde e_j=\sqrt{\beta \lambda_j} e_j, \tilde f_j=\sqrt{\beta \lambda_j} f_j$ and
$a=(a_j)_{j\in\N},b=(b_j)_{j\in\N}$ such that $a_j,b_j$ are non-negative integers with  $a_j=b_j=0$ except for a finite number of indices. Then such a family $\{\Psi_{a,b}\}$ forms an orthonormal basis of the space $L^2(\mu_{\beta,0})$. Furthermore, a standard computation yields
\begin{equation}
\label{eq.34}
\big\langle \nabla \Psi_{a,b},\nabla \Psi_{a',b'}\big\rangle_{L^2(\mu_{\beta,0}; H^{-s})}=
\beta \,\delta_{a,a'}\,\delta_{b,b'} \;\sum_{j=1}^\infty  \lambda_{j}^{1-s} \;(a_j+b_j)\,.
\end{equation}
So,  using the orthogonal decomposition with respect to the basis $\{\Psi_{a,b}\}$ and  \eqref{eq.34}, one proves
\begin{equation}
\label{eq.35}
\|\nabla F\|_{L^2(\mu_{\beta,0};H^{-s})}^2=
\sum_{a,b} \langle F, \Psi_{a,b}\rangle^2\; \bigg(\sum_{j=1}^\infty \lambda_{j}^{1-s} \;(a_j+b_j)\bigg)\,,
\end{equation}
for all  $F$ in the algebraic vector space spanned by $\{\Psi_{a,b}\}$. Then
a density argument extends \eqref{eq.35} to all  $F\in\mathbb{D}^{1,2}(\mu_{\beta,0})$. Hence, $\nabla F=0$ almost surely implies that $\langle F, \Psi_{a,b}\rangle=0$ for all $a\neq 0$ or $b\neq 0$. Thus, one concludes that $F(\cdot)= \langle F, \Psi_{0,0}\rangle \Psi_{0,0}(\cdot)=c$ for some real constant $c$.
\end{proof}

\section{Proofs of auxiliary facts from Section \ref{sec.Nonlinear}}
\label{Appendix_B}

We note a product estimate in one-dimension. This is used in part (i) of the proof of Proposition \ref{lem.9}.
\begin{lemma}
\label{product_lemma}
Let $\zeta \in (0,1/2)$ and $\alpha \in (0,1)$ be given. The following estimate holds on $\mathbb{T}^1$.
\begin{equation*}
\|fg\|_{H^{\zeta}} \lesssim_{\zeta,\alpha} \|f\|_{H^{\zeta+\frac{\alpha}{2}}}\,\|g\|_{H^{\frac{1-\alpha}{2}}}+
\|f\|_{H^{\frac{1-\alpha}{2}}}\,\|g\|_{H^{\zeta+\frac{\alpha}{2}}}\,.
\end{equation*}
\end{lemma}
\begin{proof}
Let $\mathrm{D}^{\zeta}$ denote the Fourier multiplier with symbol $\langle k \rangle^{\zeta}$ and let $\mathcal{F}^{-1}$ denote the inverse Fourier transform.
We note that
\begin{align*}
\|fg\|_{H^{\zeta}} &\sim \Bigg\|\langle k \rangle^{\zeta} \,\sum_{k' \in \mathbb{Z}} \hat{f} (k-k')\,\hat{g}(k') \Bigg\|_{\ell^2_k} \leq \Bigg\|\langle k \rangle^{\zeta}\, \sum_{k' \in \mathbb{Z}} |\hat{f} (k-k')|\,|\hat{g}(k')| \Bigg\|_{\ell^2_k}
\\
&\lesssim_{\zeta}  \Bigg\| \sum_{k' \in \mathbb{Z}} \langle k-k' \rangle^{\zeta}\,|\hat{f} (k-k')|\,|\hat{g}(k')| \Bigg\|_{\ell^2_k}
+ \Bigg\|\sum_{k' \in \mathbb{Z}} |\hat{f} (k-k')|\,\langle k' \rangle^{\zeta}\,|\hat{g}(k')| \Bigg\|_{\ell^2_k}
\\
&\sim_{\zeta} \Big\|\Big(\mathrm{D}^{\zeta} \mathcal{F}^{-1} |\hat{f}|\Big)\,\mathcal{F}^{-1} |\hat{g}|\Big\|_{L^2}+
\Big\|\mathcal{F}^{-1} |\hat{f}|\,\Big(\mathrm{D}^{\zeta} \mathcal{F}^{-1} |\hat{g}|\Big)\,\Big\|_{L^2}\,,
\end{align*}
which by H\"{o}lder's inequality is
\begin{equation}
\label{product_estimate_1}
\leq \big\|\mathrm{D}^{\zeta} \mathcal{F}^{-1} |\hat{f}|\big\|_{L^{\frac{2}{1-\alpha}}}\,\big\|\mathcal{F}^{-1} |\hat{g}|\big\|_{L^{\frac{2}{\alpha}}}+\big\|\mathcal{F}^{-1} |\hat{f}|\big\|_{L^{\frac{2}{\alpha}}}\,\big\|\mathrm{D}^{\zeta} \mathcal{F}^{-1} |\hat{g}|\big\|_{L^{\frac{2}{1-\alpha}}}\,.
\end{equation}
By using Sobolev embedding, we have
\begin{equation}
\label{product_estimate_2}
H^{\frac{\alpha}{2}}(\mathbb{T}^1) \hookrightarrow L^{\frac{2}{1-\alpha}}(\mathbb{T}^1)\,,\qquad H^{\frac{1-\alpha}{2}}(\mathbb{T}^1) \hookrightarrow L^{\frac{2}{\alpha}}(\mathbb{T}^1)\,.
\end{equation}
Substituting \eqref{product_estimate_2} into \eqref{product_estimate_1} and using the fact that $L^2$-based Sobolev norms are invariant under taking absolute values of the Fourier transform, we deduce the claim.
\end{proof}

Let us now prove two useful discrete convolution estimates. In the estimates below, we are summing over elements in $\Z^d$ (with appropriate constraints).

\begin{lemma}
\label{discrete_convolution_2D}
Let $d=2$. Let $\delta \in (0,2]$ and $M \geq 0$ be given. For all $n \in \Z^2$ and all $\rho \in [0,\delta)$, we have
\begin{equation}
\label{discrete_convolution_2D_bound}
\mathop{\sum_{k+\ell=n}}_{\max \{|k|,|\ell|\} \geq M} \frac{1}{\langle k \rangle^{\delta}\langle \ell \rangle^2} \lesssim_{\delta,\rho} \frac{\log \,\langle n \rangle}{\langle n \rangle^{\delta-\rho}\langle M \rangle^{\rho}}\,.
\end{equation}
\end{lemma}

\begin{lemma}
\label{discrete_convolution_3D}
Let $d=3$.
Let $\delta \geq 0$ and $M \geq 0$ be given. For all $n \in \Z^3$ and all $\rho \in [0,1)$, we have
\begin{equation*}
\mathop{\sum_{k+\ell=n}}_{\max \{|k|,|\ell|\} \geq M} \frac{1}{\langle k \rangle^{2+\delta}\langle \ell \rangle^2} \lesssim_{\delta,\rho} \frac{1}{\langle n \rangle^{1+\delta-\rho}\langle M \rangle^{\rho}}\,.
\end{equation*}
\end{lemma}

\begin{proof}[Proof of Lemma \ref{discrete_convolution_2D}]
We need to consider two cases, depending on the relative sizes of $|k|$ and $|\ell|$.
\\
\textbf{Case A:} $|k| \geq |\ell|$.
\\
In this case, we are estimating
\begin{equation}
\label{2d_case1}
\sum_{\ell,\, |n-\ell| \geq M} \frac{1}{\langle n-\ell \rangle^{\delta}\langle \ell \rangle^2}\,.
\end{equation}
We now need to consider three subcases, depending on the size of $|\ell|$.
\\
\textbf{Subcase A1:} $|\ell| \leq \frac{|n|}{2}$.
\\
Note that then by the triangle inequality $|n-\ell| \geq \frac{|n|}{2}$. In particular, since $|n-\ell| \geq M$, we have that
\begin{equation}
\label{2d_caseA1}
\frac{1}{\langle n-\ell \rangle^{\delta}} \lesssim_{\delta,\rho} \frac{1}{\langle n \rangle^{\delta-\rho}\,\langle M \rangle^{\rho}}\,.
\end{equation}
Furthermore, we have
\begin{equation}
\label{2d_caseA2}
\sum_{|\ell|\leq \frac{|n|}{2}} \frac{1}{\langle \ell \rangle^2} \lesssim \log \,\langle n \rangle\,.
\end{equation}
Combining \eqref{2d_caseA1} and \eqref{2d_caseA2}, we deduce that the contribution to \eqref{2d_case1} from this subcase satisfies the bound in \eqref{discrete_convolution_2D_bound}.
\\
\textbf{Subcase A2:} $\frac{|n|}{2} < |\ell| < 2|n|$.
\\
In this subcase, we have $|\ell|\sim |n|$. Furthermore, we have $M \leq |n-\ell| \leq 3|n|$.
Putting everything together, we get that if $\delta \in (0,2)$,  the contribution to \eqref{2d_case1} is
\begin{equation}
\label{2d_caseA3}
\leq \Bigg(\sum_{\ell,\,|n-\ell|\leq 3|n|} \frac{1}{\langle n-\ell\rangle^{\delta}}\Bigg)\,\frac{1}{\langle n \rangle^2}
\lesssim_{\delta}\frac{\langle n \rangle^{2-\delta}}{\langle n \rangle^2} \lesssim_{\delta,\rho} \frac{1}{\langle n \rangle^{\delta-\rho}\,\langle M \rangle^{\rho}}\,.
\end{equation}
If $\delta=2$, the upper bound gets modified to
\begin{equation}
\label{2d_caseA4}
\frac{\log\,\langle n \rangle}{\langle n \rangle^2} \lesssim_{\rho}  \frac{\log\,\langle n \rangle}{\langle n \rangle^{2-\rho}\langle M\rangle^{\rho}}\,.
\end{equation}
Note that \eqref{2d_caseA3} and \eqref{2d_caseA4} are acceptable upper bounds.
\\
\textbf{Subcase A3:} $|\ell| \geq 2|n|$.
\\
We now have $|n-\ell| \sim |\ell| \gtrsim M$. Therefore, the contribution to \eqref{2d_case1} is
\begin{equation*}
\lesssim_{\delta} \sum_{\ell,\,|\ell|\gtrsim \max\{M,|n|\}}\frac{1}{\langle \ell \rangle^{2+\delta}}\lesssim_{\delta,\rho} \frac{1}{\langle n \rangle^{\delta-\rho}\,\langle M \rangle^{\rho}}\,,
\end{equation*}
which is an acceptable upper bound.
\\
\textbf{Case B:} $|k| \leq |\ell|$.
\\
Since $k+\ell=n$, we have that $|\ell| \geq \frac{|n|}{2}$.
Hence in this case, we are estimating
\begin{equation}
\label{2d_case2}
\sum_{\ell,\, |\ell| \geq \max\{M,\frac{|n|}{2}\}} \frac{1}{\langle n-\ell \rangle^{\delta}\langle \ell \rangle^2}\,.
\end{equation}
We now need to consider two subcases, depending on the size of $|\ell|$.
\\
\textbf{Subcase B1:} $|\ell| \leq 2|n|$.
\\
In this case, we have that $|\ell| \sim |n| \gtrsim M$ and $|n-\ell| \leq 3|n|$. Therefore, the contribution to \eqref{2d_case2} is
\begin{equation*}
\lesssim \Bigg(\sum_{\ell,\,|n-\ell|\leq 3|n|} \frac{1}{\langle n-\ell\rangle^{\delta}}\Bigg)\,\,\frac{1}{\langle n \rangle^2}\lesssim_{\delta,\rho} \frac{\log \,\langle n \rangle}{\langle n \rangle^{\delta-\rho}\langle M \rangle^{\rho}}\,.
\end{equation*}
Here, we argued as in \eqref{2d_caseA3} and \eqref{2d_caseA4}.
\\
\textbf{Subcase B2:} $|\ell|>2|n|$.
\\
In this case, we have that $|n-\ell| \sim |\ell| \gtrsim M$. Hence, the contribution to \eqref{2d_case2} is
\begin{equation*}
\lesssim_{\delta}\sum_{\ell,\, |\ell| \geq \max\{M,\frac{|n|}{2}\}} \frac{1}{\langle \ell \rangle^{2+\delta}}\lesssim_{\delta,\rho} \frac{1}{\langle n \rangle^{\delta-\rho}\,\langle M \rangle^{\rho}}\,.
\end{equation*}
\end{proof}

\begin{proof}[Proof of Lemma \ref{discrete_convolution_3D}]
The proof is similar to that of Lemma \ref{discrete_convolution_2D}. We just outline the main differences.
\\
\textbf{Case A:} $|k| \geq |\ell|$.
\\
In this case, \eqref{2d_case1} gets replaced by
\begin{equation}
\label{3d_case1}
\sum_{\ell,\, |n-\ell| \geq M} \frac{1}{\langle n-\ell \rangle^{2+\delta}\langle \ell \rangle^2}\,.
\end{equation}
We consider three subcases as earlier.
\\
\textbf{Subcase A1:} $|\ell| \leq \frac{|n|}{2}$.
\\
Instead of \eqref{2d_caseA1} and \eqref{2d_caseA2}, we use
\begin{equation*}
\frac{1}{\langle n-\ell \rangle^{2+\delta}} \lesssim_{\delta,\rho} \frac{1}{\langle n \rangle^{2+\delta-\rho}\,\langle M \rangle^{\rho}}\,.
\end{equation*}
and
\begin{equation*}
\sum_{|\ell|\leq \frac{|n|}{2}} \frac{1}{\langle \ell \rangle^2} \lesssim \langle n \rangle\,,
\end{equation*}
which give us the desired bound.
\\
\textbf{Subcase A2:} $\frac{|n|}{2} < |\ell| < 2|n|$.
\\
Here, we note that
\begin{equation}
\label{3d_case1_A2}
\sum_{\ell,\,|n-\ell|\leq 3|n|} \frac{1}{\langle n-\ell\rangle^{2+\delta}}\lesssim_{\delta} \langle n\rangle^{1-\delta}
\end{equation}
and we argue similarly as in \eqref{2d_caseA3}.
\\
\textbf{Subcase A3:} $|\ell| \geq 2|n|$.
\\
We argue as in Subcase A3 in the proof of Lemma \ref{discrete_convolution_2D} and obtain that the contribution to \eqref{3d_case1} is
\begin{equation*}
\lesssim_{\delta} \sum_{\ell,\,|\ell|\gtrsim \max\{M,|n|\}}\frac{1}{\langle \ell \rangle^{4+\delta}}\lesssim_{\delta,\rho} \frac{1}{\langle n \rangle^{1+\delta-\rho}\,\langle M \rangle^{\rho}}\,.
\end{equation*}
\\
\textbf{Case B:} $|k| \leq |\ell|$.
\\
Instead of \eqref{2d_case2}, we need to estimate
\begin{equation}
\label{3d_case2}
\sum_{\ell,\, |\ell| \geq \max\{M,\frac{|n|}{2}\}} \frac{1}{\langle n-\ell \rangle^{2+\delta}\langle \ell \rangle^2}\,.
\end{equation}
We consider two subcases as earlier.
\\
\textbf{Subcase B1:} $|\ell| \leq 2|n|$.
\\
The contribution to \eqref{3d_case2} is
\begin{equation*}
\lesssim \Bigg(\sum_{\ell,\,|n-\ell|\leq 3|n|} \frac{1}{\langle n-\ell\rangle^{2+\delta}}\Bigg)\,\,\frac{1}{\langle n \rangle^2}\lesssim_{\delta,\rho} \frac{1}{\langle n \rangle^{1+\delta-\rho}\,\langle M \rangle^{\rho}}\,.
\end{equation*}
Here, we recalled \eqref{3d_case1_A2}.
\\
\textbf{Subcase B2:} $|\ell|>2|n|$.
\\
The contribution to \eqref{2d_case2} is
\begin{equation*}
\lesssim_{\delta}\sum_{\ell,\, |\ell| \geq \max\{M,\frac{|n|}{2}\}} \frac{1}{\langle \ell \rangle^{4+\delta}}\lesssim_{\delta,\rho} \frac{1}{\langle n \rangle^{1+\delta-\rho}\,\langle M \rangle^{\rho}}\,.
\end{equation*}
\end{proof}

We present the proof of \eqref{eq.nl.3_Wick} for $d=2$ and $V$ as in \eqref{V_hat_estimates}. Let us note that this proof can be deduced from \cite{MR1470880}  and we just present it here for the convenience of the reader.
\begin{proof}[Proof of \eqref{eq.nl.3_Wick}]
We recall \eqref{Wick_ordering_n} and rewrite \eqref{eq.nl.3_Wick_truncated}
\begin{multline}
\label{eq.nl.3_Wick_truncated_rewritten}
h^I_{n}(u) \equiv h^I_{n,\beta}(u)=\frac{1}{4}\,\widehat{V}(0)\,\big[(:|u_n|^2:)\,\widehat{\,}\,(0)\big]^2+\frac{1}{4}\,\sum_{k \neq 0}\big(|u_n|^2\big)\,\widehat{\,}\,(k)\,\big(|u_n|^2\big)\,\widehat{\,}\,(-k)\,\widehat{V}(k)
\\
=:h^{I,1}_{n}(u)+h^{I,2}_{n}(u)\,.
\end{multline}
We show that the sequences $(h^{I,1}_{n}(u))$ and $(h^{I,2}_{n}(u))$ are bounded in $L^p(\mu_{\beta,0})$. By appropriately modifying the proof, using Lemma \ref{discrete_convolution_2D}  and the same arguments as in part (ii) of the proof of Proposition \ref{lem.9}, we get that these sequences are Cauchy in $L^p(\mu_{\beta,0})$. We omit the details of this step.

Recalling \eqref{phi_omega}, \eqref{gradient_bound_4}, and using  and H\"{o}lder's inequality, we get that
\begin{multline}
\label{h^1_bound}
\|h^{I,1}_{n}\|_{L^p(\mu_{\beta,0})} \lesssim_{p,\beta} \Bigg\|\Bigg(\sum_{|\ell|\leq n} \frac{|g_{\ell}(\omega)|^2-1}{\langle \ell \rangle^2}\Bigg)^2\Bigg\|_{L^2(\Omega)} \leq \Bigg\|\sum_{|\ell|\leq n} \frac{|g_{\ell}(\omega)|^2-1}{\langle \ell \rangle^2}\Bigg\|_{L^4(\Omega)}^2
\\
\lesssim \Bigg\|\sum_{|\ell|\leq n} \frac{|g_{\ell}(\omega)|^2-1}{\langle \ell \rangle^2}\Bigg\|_{L^2(\Omega)}^2
\lesssim \sum_{|\ell|\leq n} \frac{1}{\langle \ell \rangle^4} \lesssim 1\,.
\end{multline}
Similarly,
\begin{multline}
\label{h^2_bound}
\|h^{I,2}_{n}\|_{L^p(\mu_{\beta,0})}^2 \lesssim_{p,\beta} \int_{\Omega}
\,\mathop{\sum_{|\ell_1|\leq n, |\ell_2|\leq n, |\ell_3|\leq n, |\ell_4|\leq n}}_{\ell_1-\ell_2+\ell_3-\ell_4=0, \ell_1 \neq \ell_2}
\,\mathop{\sum_{|\ell'_1|\leq n, |\ell'_2|\leq n, |\ell'_3|\leq n, |\ell'_4|\leq n}}_{\ell'_1-\ell'_2+\ell'_3-\ell'_4=0, \ell'_1 \neq \ell'_2}
\widehat{V}(\ell_1-\ell_2)\, \widehat{V}(\ell'_1-\ell'_2)\,
\\
\times
\frac{g_{\ell_1}(\omega)\,\overline{g_{\ell_2}(\omega)}\,g_{\ell_3}(\omega)\,\overline{g_{\ell_4}(\omega)}}{\langle \ell_1 \rangle\,\langle \ell_2 \rangle\,\langle \ell_3 \rangle\,\langle \ell_4 \rangle}\,\frac{\overline{g_{\ell'_1}(\omega)}\,g_{\ell'_2}(\omega)\,\overline{g_{\ell'_3}(\omega)}\,g_{\ell'_4}(\omega)}{\langle \ell'_1 \rangle\,\langle \ell'_2 \rangle\,\langle \ell'_3 \rangle\,\langle \ell'_4 \rangle}\,d \omega \lesssim X+Y\,,
\end{multline}
where
\begin{equation}
\label{X}
X:=\Bigg(\sum_{\ell_1,\ell_2} \widehat{V}(\ell_1-\ell_2)\,\frac{1}{\langle \ell_1 \rangle^2 \langle \ell_2 \rangle^2}\Bigg)^2\,.
\end{equation}
and
\begin{equation}
\label{Y}
Y:=\sum_{\ell_1-\ell_2+\ell_3-\ell_4=0} \widehat{V}(\ell_1-\ell_2)\,\frac{1}{\langle \ell_1 \rangle^2 \langle \ell_2 \rangle^2 \langle \ell_3 \rangle^2 \langle \ell_4 \rangle^2}\,.
\end{equation}
In the last step, we used Wick's theorem and we used \eqref{V_hat_estimates} to bound $\widehat{V} \lesssim 1$ for $Y$. We now estimate the expressions \eqref{X} and \eqref{Y}.

By \eqref{V_hat_estimates}, we have
\begin{equation}
\label{X_bound}
X \lesssim \Bigg(\sum_{\ell_1}\frac{1}{\langle \ell_1 \rangle^2} \sum_{\ell_2} \frac{1}{\langle \ell_1-\ell_2 \rangle^{\epsilon} \langle \ell_2 \rangle^2}\Bigg)^2 \lesssim \Bigg(\sum_{\ell_1}\frac{\log \,\langle \ell_1 \rangle}{\langle \ell_1 \rangle^{2+\epsilon}}\Bigg)^2\lesssim 1\,.
\end{equation}
Above, we used Lemma \ref{discrete_convolution_2D} with $\delta=\epsilon$ and $M=\rho=0$ to estimate the sum in $\ell_2$.

Furthermore, we have
\begin{multline}
\label{Y_bound}
Y = \sum_{\ell_1,\ell_2} \widehat{V}(\ell_1-\ell_2)\,\frac{1}{\langle \ell_1 \rangle^2 \langle \ell_2 \rangle^2} \mathop{\sum_{\ell_3,\ell_4}}_{\ell_3-\ell_4=-\ell_1+\ell_2} \frac{1}{\langle \ell_3 \rangle^2 \langle \ell_4 \rangle^2}
\lesssim \sum_{\ell_1,\ell_2}  \widehat{V}(\ell_1-\ell_2)\,\frac{1}{\langle \ell_1 \rangle^2 \langle \ell_2 \rangle^2} \,\frac{\log\,\langle \ell_1 -\ell_2 \rangle}{\langle \ell_1-\ell_2\rangle^2}
\\
\lesssim  \sum_{\ell_1,\ell_2} \frac{1}{\langle \ell_1 \rangle^2 \langle\ell_2 \rangle^2 \langle \ell_1-\ell_2 \rangle^2} \lesssim \sum_{\ell_1} \frac{\log \,\langle \ell_1 \rangle}{\langle \ell_1 \rangle^4} \lesssim 1\,.
\end{multline}
Above, we used Lemma \ref{discrete_convolution_2D} with $\delta=2$ and $M=\rho=0$ to estimate the sum in $\ell_3,\ell_4$ and in $\ell_2$. We also used \eqref{V_hat_estimates}.
The boundedness claim now follows from \eqref{eq.nl.3_Wick_truncated_rewritten}, \eqref{h^1_bound}, \eqref{h^2_bound},
\eqref{X_bound}, and \eqref{Y_bound}.
\end{proof}

%%%%%%%%%%%%%%%
%\bibliography{../BibTex/MyBibtex,../BibTex/KMS,../BibTex/Random}{}
\bibliographystyle{plain}

\end{document}